\documentclass[11pt]{article} \usepackage[utf8]{inputenc}
\usepackage[T1]{fontenc}
\usepackage{lmodern}
\usepackage[english]{babel}
\usepackage{showlabels}

\usepackage[a4paper,vmargin={3.5cm,3.5cm},hmargin={2.5cm,2.5cm}]{geometry}
\usepackage[font={small,sf}, labelfont={sf,bf}, margin=1cm]{caption}

\usepackage{xspace}

\usepackage[pdftex,colorlinks=true]{hyperref}
\usepackage[expansion]{microtype}
\usepackage[pdftex]{color,graphicx}
\usepackage{amsmath,amsfonts,amssymb,amsthm,mathrsfs,ulem}
\usepackage{stackrel,framed}
\usepackage{subfig}
\usepackage{stmaryrd} 
\usepackage{wasysym}

\usepackage{enumitem}
\setlist{nosep}
\usepackage{makeidx}
\makeindex

\usepackage[columns=1]{idxlayout}

\theoremstyle{plain}

\newtheorem{corollary}{Corollary}
\newtheorem{theorem}[corollary]{Theorem}
\newtheorem{proposition}[corollary]{Proposition}

\newtheorem{lemma}[corollary]{Lemma}

\newtheorem{claim}[corollary]{Claim}

\newtheorem{definition}[corollary]{Definition}

\theoremstyle{remark}
\newtheorem{remark}{Remark}
\graphicspath{{images/}}

\newcommand{\ttt}[2]{#1\xspace #2}

\newcommand{\dghp}{\ensuremath{d_{\mathrm{GHP}}}}

\usepackage[dvipsnames]{xcolor}

\newcommand{\defeq}{\overset{\hbox{\tiny{def}}}{=}}
\renewcommand{\d}{\mathrm{d}}

\newcommand{\E}{\mathbb{E}}

\newcommand{\K}{\mathbb{K}}
\newcommand{\M}{\mathbb{M}}
\newcommand{\N}{\mathbb{N}}

\newcommand{\R}{\mathbb{R}}

\newcommand{\Z}{\mathbb{Z}}

\newcommand{\kA}{\mathcal{A}}
\newcommand{\kB}{\mathcal{B}}
\newcommand{\kC}{\mathcal{C}}
\newcommand{\kD}{\mathcal{D}}
\newcommand{\kE}{\mathcal{E}}
\newcommand{\kF}{\mathcal{F}}
\newcommand{\kG}{\mathcal{G}}
\newcommand{\kH}{\mathcal{H}}

\newcommand{\kK}{\mathcal{K}}
\newcommand{\kL}{\mathcal{L}}
\newcommand{\kM}{\mathcal{M}}

\newcommand{\kP}{\mathcal{P}}
\newcommand{\kQ}{\mathcal{Q}}
\newcommand{\kR}{\mathcal{R}}
\newcommand{\kS}{\mathcal{S}}
\newcommand{\kT}{\mathcal{T}}
\newcommand{\kU}{\mathcal{U}}
\newcommand{\kV}{\mathcal{V}}

\newcommand{\kX}{\mathcal{X}}
\newcommand{\kY}{\mathcal{Y}}
\newcommand{\kZ}{\mathcal{Z}}

\newcommand{\bn}{\mathbf{n}}

\newcommand{\ve}{\varepsilon}

\def\map{\mathrm{m}}
\def\network{\mathrm{N}}
\def\mapk{\mathrm{k}}

\newcommand\carre{{\scalebox{.4}{\ensuremath{\blacksquare}}\xspace}}
\newcommand\trait{-\xspace}

\newcommand\vp{v_+}
\newcommand\vm{v_-}

\newcommand\Gi{G^{(i)}}

\newcommand\nustar{\nu_\star}

\newcommand\oN{\overline{\network}}

\newcommand\xc{x_{\mathrm{c}}}

\newcommand\Cp{C^\bullet}
\newcommand\Mp{M^\bullet}
\newcommand\kCnp{\kC_n^{\bullet}}

\newcommand\cn{\mathfrak{C}_n}
\newcommand\mn{\mathfrak{M}_n}
\newcommand\tn{\mathfrak{t}_n}
\newcommand\cnp{\cn^{\bullet}}
\newcommand\mnp{\mn^{\bullet}}
\newcommand\kn{\mathfrak{K}_n}
\newcommand\qn{\mathfrak{q}_n}

\newcommand\cq{\mathfrak{C}^{(q)}}
\newcommand\cnq{\cn^{(q)}}
\newcommand\mq{\mathfrak{M}^{(q)}}
\newcommand\mnq{\mn^{(q)}}

\newcommand{\cQ}{\ensuremath{\kQ_{(1)}}}

\newcommand\mg{\mathrm{g}}

\newcommand\cTc{\ensuremath{\kT^\carre}\xspace}
\newcommand\cTt{\ensuremath{\kT^\trait}\xspace}
\newcommand\cTtm{\ensuremath{\kT^{\circ\!\trait\!\circ}}\xspace}
\newcommand\cTct{\ensuremath{\kT^{\bullet\!\!\trait}}\xspace}

\newcommand{\cc}{\mathfrak{K}}
\newcommand{\ft}{\mathfrak{t}}
\newcommand{\fqn}{\qn}

\DeclareMathOperator{\dist}{dist}

\newcommand{\dgr}{\ensuremath{\mathrm{d}_{\mathrm{gr}}}}

\newcommand{\dstar}{\ensuremath{\mathrm{d}^\star}}
\newcommand{\dGH}{\ensuremath{\mathrm{d}_\mathrm{GH}}}
\newcommand{\dProk}{\ensuremath{\mathrm{d}_\mathrm{LP}}}
\newcommand{\dtProk}{\ensuremath{\tilde{\mathrm{d}}_\mathrm{LP}}}
\newcommand{\dGHP}{\ensuremath{\mathrm{d}_\mathrm{GHP}}}
\DeclareMathOperator{\dis}{dis}
\newcommand{\maxdeg}{\ensuremath{\mathrm{maxDeg}}}

\renewcommand{\leq}{\leqslant}
\renewcommand{\geq}{\geqslant}

\def\hcEq{\hat{\kE}_q}
\def\tcEq{\tilde{\kE}_q}

\newcommand\wmax{w^{\star}}

\def\Kq{\cc_q}
\def\hKq{\widehat{\cc}_q}
\def\tKq{\widetilde{\cc}_q}
\newcommand\hL{\widehat{L}}
\newcommand\tL{\widetilde{L}}
\newcommand\hdist{\widehat{\mathrm{d}}}
\newcommand\tdist{\widetilde{\mathrm{d}}}

\newcommand\hdelta{\widehat{\delta}}
\newcommand\tdelta{\widetilde{\delta}}

\newcommand\hNuq{\widehat{\kV}_q}
\newcommand\tNuq{\widetilde{\kV}_q}

\newcommand\eps{\epsilon}
\newcommand\bdelta{\boldsymbol\delta}

\def\Dist{\mathrm{Mult}}

\newcommand\Seq{\mathrm{Seq}}

\newcommand\bhN{\widehat{\mathbf{N}}}
\newcommand\btN{\widetilde{\mathbf{N}}}

\newcommand\hN{\widehat{N}}
\newcommand\tN{\widetilde{N}}

\newcommand\ellk{\ell^{(k)}}
\newcommand\hellk{\widehat{\ell}^{(k)}}
\newcommand\tellk{\widetilde{\ell}^{(k)}}

\newcommand\hLk{\widehat{L}^{(k)}}
\newcommand\tLk{\widetilde{L}^{(k)}}

\newcommand\hdistk{\widehat{\mathrm{d}}^{(k)}}
\newcommand\tdistk{\widetilde{\mathrm{d}}^{(k)}}

\newcommand\ck{c^{(k)}}

\newcommand\hell{\widehat{\ell}}
\newcommand\tell{\widetilde{\ell}}

\newcommand\Xk{X^{(k)}}

\newcommand\ha{\hat{a}}

\newcommand\hKqk{\hKq^{(k)}}
\newcommand\tKqk{\tKq^{(k)}}

\newcommand\hgammak{\widehat{\gamma}^{(k)}}
\newcommand\tgammak{\widetilde{\gamma}^{(k)}}

\newcommand{\lablm}{L\xspace}
\newcommand{\labrm}{R\xspace}
\newcommand{\labmm}{M\xspace}
\newcommand{\labtm}{T\xspace}
\newcommand{\labsm}{S\xspace}
\newcommand{\labpm}{P\xspace}
\newcommand{\labhm}{H\xspace}

\newcommand{\nodel}{\labl-node\xspace}
\newcommand{\noder}{\labr-node\xspace}
\newcommand{\nodem}{\labm-node\xspace}
\newcommand{\nodet}{\labt-node\xspace}

\newcommand{\noders}{\labr-nodes\xspace}

\newcommand{\compl}{\labl-component\xspace}
\newcommand{\compr}{\labr-component\xspace}
\newcommand{\compm}{\labm-component\xspace}
\newcommand{\compt}{\labt-component\xspace}

\newcommand{\compls}{\labl-components\xspace}
\newcommand{\comprs}{\labr-components\xspace}
\newcommand{\compms}{\labm-components\xspace}
\newcommand{\compts}{\labt-components\xspace}

\newcommand{\typel}{type \labl}
\newcommand{\typer}{type \labr}
\newcommand{\typem}{type \labm}
\newcommand{\typet}{type \labt}

\newcommand{\labl}{$\lablm$\xspace}
\newcommand{\labr}{$\labrm$\xspace}
\newcommand{\labm}{$\labmm$\xspace}
\newcommand{\labt}{$\labtm$\xspace}
\newcommand{\labs}{$\labsm$\xspace}
\newcommand{\labp}{$\labpm$\xspace}
\newcommand{\labh}{$\labhm$\xspace}

\newcommand{\attr}{\labr-attached\xspace}

\newcommand{\attt}{\labt-attached\xspace}

\newcommand{\deltr}{\delta_\labrm}
\newcommand{\deltt}{\delta_\labtm}

\newcommand{\xir}{\xi_\labrm}
\newcommand{\xit}{\xi_\labtm}

\newcommand{\netl}{\labl-network\xspace}
\newcommand{\netr}{\labs-network\xspace}
\newcommand{\netm}{\labp-network\xspace}
\newcommand{\nett}{\labh-network\xspace}

\newcommand{\netls}{\labl-networks\xspace}
\newcommand{\netrs}{\labs-networks\xspace}
\newcommand{\netms}{\labp-networks\xspace}
\newcommand{\netts}{\labh-networks\xspace}

\newcommand{\cycr}{\labr-cycle\xspace}

\newcommand{\cyl}[2]{\kZ_{#1,#2}}
\hypersetup{
    pdfauthor   = {},%
    pdftitle    = {}}

\begin{document}
\normalem
\title{\bf Random cubic planar graphs converge to\\ the Brownian sphere}
\author{\textsc{Marie Albenque}\footnote{LIX/CNRS, UMR7161, École Polytechnique, \url{albenque@lix.polytechnique.fr}}\, \qquad \textsc{Éric Fusy}\footnote{LIGM/CNRS, UMR8049, Université Gustave Eiffel, \url{eric.fusy@u-pem.fr}}\, \qquad  \textsc{Thomas Lehéricy}\footnote{Institut für Mathematik, Universität Zürich, \url{thomas.lehericy@math.uzh.ch}}}
\date{}
\maketitle

\begin{abstract}
In this paper, the scaling limit of random connected cubic planar graphs (respectively multigraphs) is shown to be the Brownian sphere. 

The proof consists in essentially two main steps. First, thanks to the known decomposition of cubic planar graphs into their 3-connected components, the metric structure of a random cubic planar graph is shown to be well approximated by its unique 3-connected component of linear size, with modified distances. 

Then, Whitney's theorem ensures that a 3-connected cubic planar graph is the dual of a simple triangulation, for which it is known that the scaling limit is the Brownian sphere. %In~\cite{CurienJFLG}, 
Curien and Le Gall have recently developed a framework to study the modification of distances in general triangulations and in their dual. By extending this framework to simple triangulations, it is shown that 3-connected cubic planar graphs with modified distances converge jointly with their dual triangulation to the Brownian sphere.
\end{abstract}

\bigskip

% \noindent{\bf MSC 2010 Classification:} Des codes ...\\
% \noindent{\bf Keywords:} Des mots ...

% \bigskip

%%%----------------------------------------------------------------------------------------------------------

\tableofcontents
\newpage

\section{Introduction}
In recent years, a lot of progress  has been achieved in the understanding of the scaling limit of random planar maps (planar graphs embedded in the sphere). Miermont~\cite{Miebm} and Le Gall~\cite{LGbm} established the first results of convergence of random planar maps towards the Brownian sphere. Since these major results, many other families of maps have also been proved to admit the Brownian sphere as their scaling limit. 
% The first results of convergence towards the Bo It is now known that many families of planar maps admit the Brownian sphere as their scaling limit~\cite{Miebm,LGbm}. 
Moreover, the properties of the Brownian sphere have been thoroughly investigated, we refer to the following surveys and references therein~\cite{LGsurvey,MieSF,MillerSurvey} for nice entry points to this field.

However, much less is known about the scaling limit of random planar \emph{graphs}, which are graphs that can be embedded in the sphere but for which the embedding is not fixed.
In this article, we obtain the first result of convergence of a family of random planar graphs towards the Brownian sphere. 
Our main result is the following (a graph or multigraph is called \emph{cubic} if all its vertices have degree $3$):
\index{e@\textbf{maps and graphs}!dgr@\ttt{$\dgr$}{graph distance}}

\index{f@\textbf{vertex-labeled cubic planar graphs} (replace $C$ by $M$ for multigraphs)!random@\ttt{$\cn$}{uniform in $\kC_n$}}
\begin{theorem}\label{thm:main}
Let $\cn$ (resp. $\mn$) be a uniformly random connected cubic labeled planar graph (resp. multigraph) with $2n$ vertices. Let $\mu_{\cn}$ and $\mu_{\mn}$ denote respectively the uniform distribution on $V(\cn)$ and $V(\mn)$. Let $\dgr$ denote the graph distance. Then, there exist constants $\mathrm{c},\mathrm{c}_m\in(0,\infty)$ such that the following convergence holds:
\begin{equation}
\left(V(\cn),\frac{1}{\mathrm{c}\cdot n^{1/4}}\dgr,\mu_{\cn}\right)\xrightarrow[n\to\infty]{} \map_\infty,
\end{equation}
and
\begin{equation}
\left(V(\mn),\frac{1}{\mathrm{c}_m\cdot n^{1/4}}\dgr,\mu_{\mn}\right)\xrightarrow[n\to\infty]{} \map_\infty,
\end{equation}
where the convergence holds in distribution for the Gromov--Hausdorff--Prokhorov topology for measured metric spaces, and $\map_\infty$ denotes the Brownian sphere. 
%\footnote{All these notions and the following ones used in the introduction will be defined in more details in Section~\ref{sec:preliminaries}.}.
\end{theorem}
While finishing this project, we became aware that Benedikt Stufler obtained simultaneously and independently the scaling limit of random connected cubic planar graphs~\cite{bs22fpp,bs22scaling}, using the same general strategy but with different techniques to resolve the main difficulties.

Our motivation to consider not only cubic planar graphs but also cubic planar multigraphs comes from the fact that they appear naturally as the kernel of random connected planar graphs in the sparse regime (i.e. when the excess of the graph is $o(n)$). The scaling limit of random planar \emph{maps} in the sparse regime has recently been investigated in~\cite{curien2021mesoscopic}, via the study of the scaling limit of their cubic kernel. To extend this result to the case of random sparse planar graphs, the scaling limit of random connected cubic planar multigraphs is thus a decisive step\footnote{A simpler problem involving kernel extraction would be to obtain the scaling limit of random connected planar multigraphs
that are \emph{precubic}, i.e., with vertex-degrees in $\{1,3\}$.}.

\bigskip

Before giving the main ideas behind the proof of this result, let us place it in the context of the existing literature on the enumeration of planar graphs and on the study of their random properties. Unlike planar maps, for which closed enumerative formulas are often available~\cite{tutte1962census}, the enumeration of planar graphs is more involved. Building on  a canonical decomposition of graphs into their 3-connected components~\cite{tutte2019connectivity}, a (differential) system of equations for their generating series can be obtained. Major achievements have consisted in deriving precise asymptotic estimates from these equations. A non-exhaustive list of contributions includes the case of 2-connected planar graphs~\cite{bender2002number}, the case of cubic planar graphs~\cite{BKLM} and cubic planar multigraphs~\cite{kang2012two} and the case of planar graphs~\cite{gimenez2009asymptotic}. We also refer to the survey~\cite{noy2014random} and references therein.

These enumerative results have made possible the study of random properties of planar graphs, and to establish limit laws for various parameters, such as the maximum degree~\cite{drmota2014maximum}. They have also been a key ingredient in the recent proofs of the convergence of random cubic planar graphs \cite{stufler2022uniform} and random planar graphs \cite{stufler2021local} for the local topology. 

However, only few results are available about the metric properties of random planar graphs. It has only been shown that the order of magnitude of the diameter of random planar graphs with $n$ vertices is concentrated around $n^{1/4}$~\cite{chapuy2015diameter}. Let us also mention that the scaling limit of some families of planar graphs have been previously established. However, such results only hold for so-called \emph{subcritical} families of graphs (which include in particular outerplanar graphs and series-parallel graphs) and their scaling limit is Aldous' Brownian Continuum Tree~\cite{panagiotou2016scaling,SurveyStufler}.

\bigskip

\subsection*{General strategy of the proof and organization of the paper}

\paragraph{Decomposition of cubic planar graphs along their 3-connected components.}
The first ingredient in the proof of Theorem~\ref{thm:main} is the known decomposition of cubic planar graphs into 3-connected components. This decomposition allows us to use techniques similar to the ones developed in~\cite{banderier2001random,gao1999size} and to prove that a random cubic planar graph (resp. cubic planar multigraph) admits a unique giant 3-connected component called its \emph{3-connected core}. This is presented in Section~\ref{sec:decompo}. A key feature is that the metric properties of the original graph can be very precisely approximated by the metric properties of its 3-connected core with random edge-lengths. More precisely, the length of a given edge follows asymptotically a distribution denoted $\nustar$. 
% To study this model, a first step is to consider the case where the edge-lengths are i.i.d random variables distributed according to $\nustar$. 
To study this model, a first step is to introduce and study the distribution $\nustar$, this is undertaken in Section~\ref{sec:bounds}.

Dealing with a 3-connected planar graph allows us to enter the world of planar maps. Indeed, by Whitney's theorem~\cite{whitney19922}, a 3-connected planar graph is known to admit a unique embedding (up to mirror). Hence, 3-connected planar graphs and 3-connected planar maps are the same objects. By duality, 3-connected cubic planar graphs correspond to simple triangulations (i.e. triangulations with no loops nor multiple edges). 

\paragraph{Modification of distances in simple triangulations.}
In~\cite{CurienJFLG}, Curien and Le Gall studied modifications of distances in triangulations. Their techniques have also been adapted to quadrangulations and general maps by the third author~\cite{Lehfpp}, and to Eulerian triangulations~\cite{carrance2021convergence}. To establish the scaling limit of random 3-connected cubic planar maps with random i.i.d edge-lengths, we extend their framework in two directions. We deal with simple triangulations rather than general triangulations, and we consider the \emph{general} first-passage percolation distance on the edges of the dual map rather than the dual graph distance or the Eden model (corresponding to exponential random variables on the edges of the dual map). 
To be able to state our result precisely, we first introduce a couple of definitions and notations.

\index{l@\textbf{modification of distances}!distr@\ttt{$\nu$}{weight distribution on $[\eta_0,\infty)$ for $(w_e)_{e\in E(\map^\dagger)}$}}
\index{l@\textbf{modification of distances}!fpp@\ttt{$\dstar_\nu = \dstar$}{first-passage percolation distance on $\map^\dagger$}}
Fix $\mathrm m$ a planar map, and write $\mathrm m^{\dagger}$ for its dual. Let $\nu$ be a probability distribution on $\R_{>0}$. We define a (random) metric on $\mathrm m^{\dagger}$ as follows. Let  $(w_e)_{e\in E(\mathrm m^{\dagger})}$ be a collection of random weights indexed by the set of edges of $\mathrm m^{\dagger}$, such that $(w_e)$ are i.i.d random variables sampled from $\nu$. We define the first-passage percolation distance $\dstar_\nu$ on $\mathrm m^\dagger$ by setting, for any $u,v \in V(\mathrm m^{\dagger})$:
\begin{equation}
\dstar_\nu(u,v)=\inf_{\gamma: u\rightarrow v}\sum_{e\in \gamma}w_e,
\end{equation}
where the infimum runs over all paths $\gamma$ going from $u$ to $v$ in $\mathrm m^{\dagger}$. 
\medskip

A key ingredient in the proof of Theorem~\ref{thm:main} is the following result, which is of independent interest: 
\begin{theorem}\label{th:LGC_simple}
Let $\tn$ be a uniformly random simple triangulation with $n+2$ vertices. Let $\mu_{\ft_n}$ and $\mu_{\ft_n^\dagger}$ denote respectively the uniform distribution on $V(\tn)$ and on $V(\tn^\dagger)$. Assume that the support of $\nu$ is included in $[\eta_0,\infty)$ for $\eta_0>0$, and that $\nu$ has exponential tails.

Then, there exists a constant $c_\nu>0$ such that the following convergence holds jointly in distribution for the Gromov--Hausdorff--Prokhorov topology: 
\begin{equation}
\left(\left(V(\tn),\Big(\frac{3}{4n}\Big)^{1/4}\dgr,\mu_{\ft_n}\right),\left(V(\tn^\dagger),\Big(\frac{3}{4n}\Big)^{1/4}\frac1{c_\nu}\cdot \dstar_\nu,\mu_{\ft_n^\dagger}\right)\right)\xrightarrow[n\rightarrow \infty]{(d)}(\map_{\infty},\map_{\infty}),
\end{equation}
where $\map_{\infty}$ is the Brownian sphere.
\end{theorem}
Note that the convergence of the first component was established by Addario-Berry and the first author in~\cite{SimpleTrig}. By applying this theorem to the case where $\nu=\delta_1$ is the Dirac measure on $\{1\}$, and by applying Whitney's theorem, we obtain as an immediate corollary: 
\index{g@\textbf{3-connected cubic planar graphs}!rand@\ttt{$\kn$}{uniform in $\kK_n$}}
\begin{theorem}\label{thm:3connectedScaling}
Let $\cc_n$ be a uniformly random 3-connected cubic planar graphs with $2n$ vertices. Denote by $\mu_{\cc_n}$ the uniform distribution on $V(\cc_n)$. Then, there exists a constant $c_3>0$ such that the following convergence holds in distribution for the Gromov--Hausdorff--Prokhorov topology:
\begin{equation}
\left(V(\cc_n),\frac{1}{c_3 n^{1/4}}\dgr,\mu_{\cc_n}\right)\xrightarrow[n\rightarrow \infty]{(d)}\map_{\infty},
\end{equation}
where $\map_{\infty}$ is the Brownian sphere.
\end{theorem}
When adapting the framework of Curien and Le Gall to our setting, the main difficulty comes from the fact that we consider simple triangulations rather than (unconstrained) triangulations. The article~\cite{CurienJFLG} relies extensively on the so-called \emph{skeleton decomposition} of triangulations, introduced by Krikun in~\cite{Krikun}. However, the skeleton decomposition does not behave as nicely on simple triangulations as on general triangulations. To circumvent this issue, we apply the skeleton decomposition to the model of quasi-simple triangulations (defined in~Section~\ref{sec:quasi_simple}). The very nice property of quasi-simple triangulations is that their skeleton decomposition can be encoded via exactly the same branching process as general triangulations. Therefore, many results of~\cite{CurienJFLG} extend to our setting effortlessly. This comes at the price that quasi-simple triangulations are not invariant under re-rooting (its root plays indeed a special role), so that some arguments need to be adapted. This is the purpose of Section~\ref{sec:modDistances}. On the other hand, in his independent preprint~\cite{bs22fpp}, Benedikt Stufler deals directly with a skeleton decomposition for simple triangulations, where some forbidden configurations have to be handled.

\paragraph{From i.i.d edge weights to the scaling limit of random cubic planar graphs.}
Given Theorem~\ref{th:LGC_simple}, there remain two last steps to prove Theorem~\ref{thm:main}. The first one deals with the metric property: indeed, we need to relax the hypothesis that the edge lengths are independent. The second one deals with the measure on the vertices. Roughly speaking, we have to prove that the uniform measure on the vertices of a random connected cubic planar graph on one hand, and on the vertices of its 3-connected core on the other hand, do not differ asymptotically. 

More precisely, to prove Theorem~\ref{thm:main}, we establish in fact the following joint convergence result: 
\index{g@\textbf{3-connected cubic planar graphs}!zcore@\ttt{$K(\mg)$}{3-connected core of $\mg$}}
\begin{theorem}\label{th:convJointe}
Let $\cn$ be a uniformly random connected cubic labeled planar graph (resp. multigraph) with $2n$ vertices. Write $K(\cn)$ for its 3-connected core and $\Delta(\cn)$ for the (unlabeled, rooted) dual of $K(\cn)$. We denote respectively by $\mu_{\cn}$, $\mu_{K(\cn)}$ and $\mu_{\Delta(\cn)}$ the uniform measures on $V(\cn)$, $V(K(\cn))$ and $V(\Delta(\cn))$. Moreover, we write $\mathrm{dist}_{\cn}$ for the distance on $K(\cn)$ induced by the graph distance on $\cn$.
\index{g@\textbf{3-connected cubic planar graphs}!zdcore@\ttt{$\mathrm{dist}_{\mg}$}{distance on $K(\mg)$ induced by the graph distance on $\mg$}}

Then, there exist two positive constants $c_1$ and $c_2$ (with $c_2$ explicit) such that:
\begin{multline}
\left(\left(\cn,\frac{1}{c_1 n^{1/4}}\dgr,\mu_{\cn}\right),
\left(K(\cn),\frac{1}{c_1 n^{1/4}}\mathrm{dist}_{\cn},\mu_{K(\cn)}\right),
\left(\Delta(\cn),\frac{1}{c_2 n^{1/4}}\dgr,\mu_{\Delta(\cn)}\right)
\right)\\
\xrightarrow[n\rightarrow \infty]{(d)}(\map_{\infty},\map_{\infty},\map_{\infty}),
\end{multline}
where $\map_{\infty}$ is the Brownian sphere.
Moreover, a similar result holds when replacing $\cn$ by a random connected cubic planar multigraph $\mn$.
\end{theorem}
Given Theorem~\ref{th:LGC_simple}, to establish the joint convergence of the last two coordinates, we construct in Section~\ref{sec:two_point_dep} an explicit coupling between the models with independent edge-lengths (sampled from $\nustar$) and $\left(K(\cn),\mathrm{dist}_{\cn}\right)$. With this coupling, we prove in Theorem~\ref{th:scaling3connex} that the Gromov--Hausdorff distance between both models tends to 0 in probability.

Finally, in Section~\ref{sec:projection}, to establish the joint convergence of the first two coordinates, we rely on the  methodology developed by Addario-Berry and Wen  in~\cite{addario2017joint}. It involves an explicit projection from the vertices of $\cn$ to the vertices of its $3$-connected core $K(\cn)$. The key points are the property that each vertex is projected to a vertex at distance $o(n^{1/4})$ with high probability, and the property that the projection of $\mu_{\cn}$ is asymptotically close to $\mu_{K(\cn)}$ (in the sense of the L\'evy-Prokhorov distance).
\medskip

Let us finish by mentioning that an extensive index of notations concludes the paper.

\subsection*{Acknowledgments}
M.A.'s research was supported by the ANR grant GATO (ANR-16-CE40-0009-01) and the ANR grant IsOMa (ANR-21-CE48-0007). E.F. acknowledges the support of the ANR grant GATO (ANR-16-CE40-0009-01) and the ANR grant Combine (ANR-19-CE48-0011).

We are grateful to an anonymous referee, whose careful reading significantly improved the presentation of this paper.

\section{Preliminaries}

\subsection{Definitions on graphs and maps}

% \index{e@\textbf{maps and graphs}!vertex@\ttt{$v\in V(\map)$}{vertex}}
% \index{e@\textbf{maps and graphs}!vertex@\ttt{$f\in F(\map)$}{face}}
% \index{e@\textbf{maps and graphs}!vertex@\ttt{$e\in E(\map)$}{edge}}

In this article, we consider both multigraphs, and graphs (i.e., multigraphs with no loop nor multiple edges). A multigraph is called \emph{vertex-labeled} if its $n$ vertices
carry distinct labels in $[1..n]$, and is called \emph{half-edge-labeled} if its $2m$ half-edges carry distinct labels in $[1..2m]$. When dealing with random cubic planar multigraphs
(resp. graphs) we will usually consider that they are half-edge-labeled (resp. vertex-labeled). A connected multigraph is \emph{rooted} if it comes with a distinguished oriented edge called
its \emph{root edge}. The tail vertex of the root edge is called the \emph{root vertex}. 

Note that a connected graph or multigraph can be seen as a discrete metric space, where the distance between two vertices is the minimal length over the paths connecting them.
 More generally, we will consider so-called \emph{metric graphs}, i.e., graphs where every edge $e$ carries a length \index{e@\textbf{maps and graphs}!length!edge@\ttt{$\ell(e)$}{of an edge }}$\ell(e)\in\R_+^*$. 
 In such a graph seen as a metric space, the \index{e@\textbf{maps and graphs}!length!path@\ttt{$L(\gamma)$}{of a path}}length of a path $\gamma=(e_1,\ldots,e_m)$ 
 is defined as $L(\gamma)=\ell(e_1)+\cdots +\ell(e_m)$, and the distance between two vertices is the length of a 
 shortest path connecting them (on the other hand, in the context of metric graphs, we will use the notation $|\gamma|$ for the number of edges in $\gamma$).

A \emph{planar map} (shortly, a map) $\map$ is a connected multigraph properly embedded on the sphere, up to orientation-preserving homeomorphism. A \emph{face} of $\map$ 
is a connected component of the sphere cut by the embedding, a \emph{corner} of $\map$ is a sector between two consecutive edges around a vertex (a corner thus lies in a unique face of $\map$).  
\index{e@\textbf{maps and graphs}!degree@\ttt{$\deg_\map(v), \deg_\map(f)$}{degree}}
The \emph{degree} of a vertex (resp. face) of $\map$ is the number of  corners that are incident to it. We write $E(\map), V(\map)$ and $F(\map)$ for the set of edges, vertices and faces of $\map$ respectively.

A \emph{rooted map} is a map that comes
with a distinguished oriented edge called its \emph{root edge}, i.e., the underlying multigraph is rooted.  
The face on the right of the root edge is called the \emph{root face} 
(also called the \emph{outer face}, since it is usually taken as the unbounded face
when the embedding of the map is represented on the plane), and the corner incident to the root vertex just after the root edge in clockwise order is called the \emph{root corner}.  
Similarly, vertices incident to the root face are called \emph{outer vertices}, and \emph{inner} vertices otherwise.

In general, maps are considered as unlabeled (whether on vertices or on half-edges).  
Indeed, in contrast to the situation for multigraphs, the rooting operation is sufficient to suppress the possibility of having non-trivial automorphisms. 
A map is \emph{simple} if the underlying multigraph is a graph, i.e. if there is no double edge and no loop. 
\medskip

\index{e@\textbf{maps and graphs}!dual@\ttt{$\map^\dagger$}{dual map}}
For  $\map$ a map, the \emph{dual map} $\map^\dagger$ is constructed as follows (see Figure~\ref{Fig_illustration_dual}): 
we add a ``type F'' vertex in every face of $\map$. Then for every edge $e$ of $\map$, we draw an edge $e^\dagger$ (the \emph{dual edge} of $e$) between the type F vertices of the two faces that are incident to $e$, in such a way that $e^\dagger$ crosses $e$ and no other edge (including the new ones) --- if both sides of $e$ are incident to the same face, then $e^\dagger$ is a loop. Then $\map^\dagger$ has vertex set the set of type F vertices, and edge set the set of dual edges. If $\map$ is rooted, the root edge of  
$\map^\dagger$ is taken as  the dual of the root edge $e$ of $\map$, oriented in such a way that it crosses $e$ in the counterclockwise direction around the root vertex. 

\begin{figure}[t!]
\begin{center}
\includegraphics[width=11.2cm]{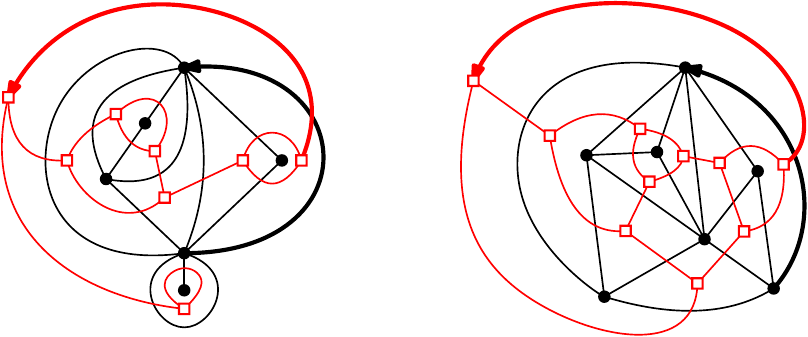}
\end{center}
\caption{Two examples of triangulations superimposed with their dual. The one on the left is non-simple while the one on the right is simple. The dual cubic map is $3$-connected only
for the example on the right.}
\label{Fig_illustration_dual}
\end{figure}

One can see from the definition that $(\map^\dagger)^\dagger$ is the map obtained from $\map$ by reversing the direction of its root edge. Furthermore, the one-to-one correspondence  between the faces of $\map$ and the vertices of $\map^\dagger$ preserves the degree: if $f$ is a face of $\map$, and $f^\dagger$ is the associated vertex of $\map^\dagger$, then $\deg_\map(f) = \deg_{\map^\dagger}(f^\dagger)$. The same holds for vertices of $\map$: every face of $\map^\dagger$ contains a unique vertex of $\map$, and if $v$ is a vertex of $\map$ and $v^\dagger$ is the face of $\map^\dagger$ that contains it, then $\deg_\map(v) = \deg_{\map^\dagger}(v^\dagger)$.

\index{e@\textbf{maps and graphs}!size@\ttt{$\# E/3 = \# V/2 = \# F-2$}{size of a planar cubic (multi)graph or map}}
A multigraph or map is called \emph{cubic} if every vertex has degree $3$. The \emph{size} of such a multigraph or map is the integer $n\geq 1$ such that 
the number of vertices is $2n$ and the number of edges is $3n$. By Euler's formula, a cubic map of size $n$ has $n+2$ faces. 
On the other hand, a \emph{triangulation} is a map where all faces have degree $3$.  
Since taking the dual of a map sends vertices of the primal to faces of the dual and vice-versa, it follows that the dual of triangulations are cubic maps; and conversely, the dual of cubic maps are triangulations. A triangulation with $n+2$ vertices thus has $2n$ faces and $3n$ edges. 

A multigraph is called \emph{$3$-connected} it is simple, has at least $4$ vertices, and at least $3$ vertices need to be deleted to disconnect it. 
 By Whitney's theorem~\cite{whitney19922}, such a graph has exactly two embeddings on the sphere, which differ by a mirror.  
A \emph{$3$-connected map} is a map whose underlying graph is $3$-connected. It is known that duality preserves $3$-connectivity~\cite{mullin1968enumeration}, 
and that a triangulation is $3$-connected if and only if it is simple. 
 
 \medskip 
  
 \begin{remark}\label{rk:3conn}
 For a 3-connected graph $\mapk$ that is vertex-labeled and rooted, we can canonically choose one of the two embeddings $R,R'$ (which differ by a mirror) as follows. 
 In $R$ we let $f_1,f_2$ be the two faces incident to the root edge~$e$.  
Let $v_0$ be the root-vertex, and let $v_1$ (resp. $v_2$) be the unique vertex of $f_1$ (resp. $f_2$) adjacent to $v_0$ and not incident to the root edge. 
By 3-connectivity, we have $v_1\neq v_2$. If the label of $v_1$ is smaller than the label of $v_2$, then we take $R$ as the embedding associated to $\mapk$, otherwise we take $R'$. 
The  \emph{outer edges} of $\mathrm{k}$ are those incident to the outer face in the canonical embedding.  
Note that each $3$-connected (unlabeled) rooted map with $n$ vertices has $n!/2$ preimages under this correspondence.   \dotfill
\end{remark}

\subsection{Gromov--Hausdorff--Prokhorov topology}\label{sec:defGHP}

We now define the Gromov--Hausdorff--Prokhorov distance. For details and proofs, we refer to \cite{burago2001course} and \cite[Section 6]{miermont2009tessellations}.  

Let $(\kX,\d)$ be a metric space. For every $K \subset \kX$ and every $\varepsilon>0$ let $K^\varepsilon=\{x\in\kX: \d(x,K)<\varepsilon\}$. The Hausdorff distance between two compact sets $K_1$ and $K_2$ of $\kX$ is
\begin{equation*}
\mathrm{d}^\kX_{\mathrm{H}}(K_1, K_2) \defeq \inf \{ \varepsilon>0 \ : \ K_1 \subset K_2^\varepsilon \text{ and } K_2 \subset K_1^\varepsilon \} .
\end{equation*}
The Gromov--Hausdorff distance allows us to compare two compact metric spaces by isometrically embedding them in a common, larger metric space, and comparing the Hausdorff distance of their embedding. More formally for every pair of compact metric spaces $(\kX,\d^\kX)$ and $(\kY,\d^\kY)$, we define their Gromov--Hausdorff distance as:\index{d@\textbf{distances}!GH@\ttt{$\dGH$}{Gromov--Hausdorff}}
\begin{equation*}
\dGH((\kX,\d^\kX),(\kY,\d^\kY)) \defeq \inf \{ \mathrm{d}^\kZ_{\mathrm{H}}(\Phi^\kX(\kX), \Phi^\kY(\kY)) \} , 
\end{equation*}
where the infimum is over every metric space $\kZ$ and isometries $\Phi^\kX: \kX \to \kZ$ and $\Phi^\kY: \kY\to\kZ$. Since isometric spaces are at Gromov--Hausdorff ``distance'' zero, $\dGH$ only defines a pseudo-distance on compact metric spaces, but it becomes a true distance when considered on $\M$ the set of compact metric spaces seen up to isometry. Furthermore, $(\M,\dGH)$ is a Polish space. %There exist equivalent descriptions of $\dGH$; see e.g. \cite[Section 3.3]{miermont2014aspects}.
\bigskip

If $(\kZ,\d)$ is a metric space with Borel $\sigma$-algebra $\kB(\kZ)$, we define the Lévy--Prokhorov distance between two Borel probability measures $\mu,\mu'$ on $\kZ$ as follows: \index{d@\textbf{distances}!LP@\ttt{$\dProk(\mu,\mu')$}{Lévy--Prokhorov}}
\begin{equation*}
\dProk^\kZ(\mu,\mu') \defeq \inf \{ \varepsilon>0 \ : \ \forall A \in \kB(\kZ), \mu(A)\leq \nu(A^\varepsilon)+\varepsilon \text{ and } \nu(A) \leq \mu(A^\varepsilon)+\varepsilon \} .
\end{equation*}
If $\kZ$ is a Polish space then the Lévy--Prokhorov distance makes the set of probability measures on $\kZ$ a Polish space, and $\dProk$ metrizes the weak convergence of measures. There is also a characterization of $\dProk^\kZ(\mu,\mu')$ in terms of coupling, see for instance~\cite{GPW}.  A \emph{coupling} between $\mu$ and $\mu'$ is a probability measure $\nu$ on $\kZ\times \kZ$ such that $\pi^1_* \nu = \mu$ and $\pi^2_* \nu = \mu'$ (where $\pi^1$ and $\pi^2$ denote the canonical projections with respect to the first and second coordinate respectively). We denote by $\kM(\mu,\mu')$ the set of couplings between $\mu$ and $\mu'$. \index{d@\textbf{distances}!coupling@\ttt{$\kM(\mu,\mu')$}{couplings}}
  Then, if $(\kZ\times\kZ)_{\eps}$ denotes the set $\{(z,z')\in\kZ\times\kZ\, ,\, \mathrm{d}(z,z')\leq\eps\}$, we have $\dProk^\kZ(\mu,\mu')\leq\eps$ if and only if there exists a coupling $\nu\in \kM(\mu,\mu')$ such that $\nu((\kZ\times\kZ)_{\eps})\geq 1-\eps$. 
\bigskip

A weighted metric space is a metric space equipped with a probability Borel measure. 
The Gromov--Hausdorff--Prokhorov distance $\dGHP$ is a metric on isometry classes of compact weighted metric spaces, defined as follows. 
For any measured spaces $\kX, \kZ$, measurable function $\Phi: \kX \to \kZ$ and measure $\mu$ on $\kX$, write $\Phi_* \mu = \mu(\Phi^{-1}(\cdot))$ for the pushforward of $\mu$ under $\Phi$. 
Let $(\kX, \d^\kX)$ and $(\kY,\mathrm{d}^\kY)$ be two compact metric spaces equipped respectively with the probability measures $\mu^\kX$ and $\mu^\kY$; define \index{d@\textbf{distances}!\ttt{$\dGHP$}{Gromov--Hausdorff--Prokhorov}}
\begin{equation*}
\dGHP((\kX,\d^\kX, \mu^\kX) \, , \, (\kY, \d^\kY, \mu^\kY)) 
\defeq \inf\left\{ \mathrm{d}^\kZ_{\mathrm{H}}(\Phi^\kX(\kX), \Phi^\kY(\kY)) 
\ + \ \dProk^\kZ(\Phi^\kX_*\mu^\kX, \Phi^\kY_* \mu^\kY) \right\} ,
\end{equation*}
where (as in the definition of $\dGH$) the infimum is over every metric space $(\kZ,\d^\kZ)$ and isometries $\Phi^\kX: \kX \to \kZ$ and $\Phi^\kY: \kY \to \kZ$. Because isometric weighted metric spaces are at Gromov--Hausdorff--Prokhorov ``distance'' zero, the function $\dGHP$ only defines a pseudo-metric; to make it a true metric, we consider $\K$ the set of isometry classes of compact weighted metric spaces. Then $\dGHP$ defines a true metric on $\K$, and $(\K, \dGHP)$ is a Polish space.

There exists also an equivalent description of $\dGHP$ based on couplings which is better suited for our purposes. Let us introduce the concept of correspondence. A \emph{correspondence} between two sets $\kX,\kX'$ is a subset $\kR \subset \kX\times \kX'$ such that $\pi^\kX (\kR) = \kX$ and $\pi^{\kX'} (\kR) = \kX'$, where $\pi^\kX$, resp. $\pi^{\kX'}$ is the canonical projection on $\kX$, resp. $\kX'$. We write $\kC(\kX,\kX')$ for the set of correspondences between $\kX$ and $\kX'$. Given a correspondence, define its \emph{distortion}\index{d@\textbf{distances}!distortion@\ttt{$\dis(\kR)$}{distortion}}
\begin{equation*}
\dis(\kR) \defeq \sup \left\{ \left| \d^\kX(x,y) - \d^{\kX'}(x',y') \right| \ : \ (x,y), (x',y') \in \kR \right\} .
\end{equation*}
The Gromov--Hausdorff distance can be expressed in terms of distortion:
\begin{equation*}
\dGH((\kX,\d^\kX),(\kY,\d^\kY)) = \frac{1}{2} \inf \{ \dis(\kR) : \kR \in \kC(\kX,\kY) \} ,
\end{equation*}
and this infimum is attained.

Let $\mathbf{X} = (\kX,\d^\kX,\mu^\kX)$ and $\mathbf{Y} = (\kY,\d^\kY,\mu^\kY)$ be two compact weighted metric spaces. Recall the notation $\kM(\mu^\kX,\mu^\kY)$ for the set of couplings between $\mu^\kX$ and $\mu^\kY$, then we have the following alternate characterization of the Gromov--Hausdorff--Prokhorov distance:
% A \emph{coupling} between $\mu^\kX$ and $\mu^\kY$ is a probability measure $\nu$ on $\kX\times \kY$ such that $\pi^\kX_* \nu = \mu^\kX$ and $\pi^\kY_* \nu = \mu^\kY$; we denote by $\kM(\mu^\kX,\mu^\kY)$ the set of couplings between $\mu^\kX$ and $\mu^\kY$ (and omit the dependence in $\kX$ and $\kY$ for the sake of readability). Then 
\begin{equation*}
\dGHP(\mathbf{X}, \mathbf{Y}) = \inf\{\ve>0 \ : \ \exists \kR\in\kC(\kX,\kY), \exists \nu\in\kM(\mu^\kX,\mu^\kY), \dis(\kR) \leq 2\ve , \nu(\kR) \geq 1-\ve \} .
\end{equation*}

\section{The 3-connected core of cubic planar graphs}\label{sec:decompo}

\subsection{Decomposition into 3-connected components}

\index{h@\textbf{cubic networks}|(}

We recall here a decomposition~\cite{BKLM,kang2012two,noy2020further} of so-called cubic networks (closely related to rooted cubic planar graphs) and of cubic planar (multi-)graphs
into 3-connected components.

\index{h@\textbf{cubic networks}!poles@\ttt{$\vm,\vp$}{poles}}
A \emph{cubic network} is a connected planar multigraph $\network$ with two marked vertices $\vm,\vp$ called the \emph{poles}, such that the poles have degree $1$ and the other vertices have degree $3$. Note that the number of non-pole vertices has to be a positive even number $2n$, then $n\geq 0$ is called the \emph{size} of $\network$. The two edges adjacent to the poles are called the \emph{legs} of $\network$, and the other edges are called \emph{plain}. The \emph{trivial network} is the one of size $0$, with a single edge connecting the two poles. 
Note that a cubic network $\network$ of positive size identifies to a rooted cubic planar multigraph $\oN$, which is obtained by merging the two poles, smoothing out the merged vertex, and orienting the resulting edge from the neighbor of $\vm$ to the neighbor of $\vp$. Then $\network$ is called \emph{polyhedral} if $\oN$ is 3-connected, in which 
case the \emph{outer edges} of $\network$ are the (plain) edges of $\network$ corresponding to the outer edges (without including the root edge) of $\oN$.   
We also define the \emph{double-edge network} as the cubic network of size $1$ where the  neighbors of the poles are connected by a double edge. 

\index{h@\textbf{cubic networks}!add@\ttt{$N_1+N_2$}{\netr}}
Given two non-trivial cubic networks $\network_1,\network_2$, we let $\network:=\network_1+\network_2$ be the cubic network obtained by merging the plus-pole of $\network_1$ with the minus-pole of $\network_2$ and smoothing out the merged vertex, so that the plus-pole of $\network$ is the plus-pole of $\network_2$, and the minus-pole of $\network$ is the minus-pole of $\network_1$. A cubic network $\network$ that can be obtained as $\network=\network_1+\network_2$ is called an \emph{\netr}, in which case $\network$ has a unique decomposition as $\network=\network_1+\network_2$ where $\network_1$ is a (non-trivial) non-\netr, and by unfolding it also has a unique decomposition as $\network=\network_1+\cdots+\network_k$ where each component $\network_i$ is a non-trivial non-\netr.

In this article, \emph{edge-substitution} in a graph $G$ is the operation consisting, for each edge $e=(u_-,u_+)$ of $G$ (canonically endowed with a minus-extremity and a plus-extremity), of substituting it by a cubic network $\network_e$, so that the plus-pole and minus-pole of $\network_e$ are identified respectively with $u_-$ and $u_+$ (note that substitution by the trivial network 
leaves the edge unchanged). If $G$ is itself a cubic network, only the plain edges are substituted. \index{h@\textbf{cubic networks}!netm@\netm (substitution from double-edge)}\index{h@\textbf{cubic networks}!nett@\nett (substitution from polyhedral)}
A cubic network obtained by edge-substitution from the double-edge network (resp. from a polyhedral network) is called an \emph{\netm} (resp. a \emph{\nett}). 

Note that \netrs, \netms and \netts are such that the two poles are at distance greater than $2$. \index{h@\textbf{cubic networks}!\netl}
 An \emph{\netl} is a cubic network $\network$  whose poles are at distance $2$. In an \netl, let $u$ be the common neighbor of the poles, and let $v$ be its other neighbor (note that the edge $\{u,v\}$ is an isthmus).    
Let $\mathrm{l}_1$ be the \netl of size $1$ where $v$ carries a loop. And let $\mathrm{l}_2$ be the \netl of size $2$ where $v$ has two neighbors (apart from $u$) each carrying a loop.
As shown in~\cite{BKLM,kang2012two}, an \netl $\network$ is either obtained from $\mathrm{l}_1$ where only the loop-edge is allowed to be substituted, or is obtained from $\mathrm{l}_2$ where only the two 
loop-edges are allowed to be substituted (the first and second case correspond respectively to $N\backslash v$ having two or three connected components). The two cases are called
 type 1 and type 2, respectively. 

We have the following decomposition result for cubic networks, shown in \cite[Lem.1]{kang2012two} (the version for graphs appeared previously in~\cite[Lem.1]{BKLM}):

\begin{lemma}\label{lem:dec_network}
Any non-trivial cubic network is in exactly one of the following classes: $\{$\netl, \netr, \netm, \nett$\}$.
\end{lemma}

\begin{figure}
\begin{center}
\includegraphics[width=12cm]{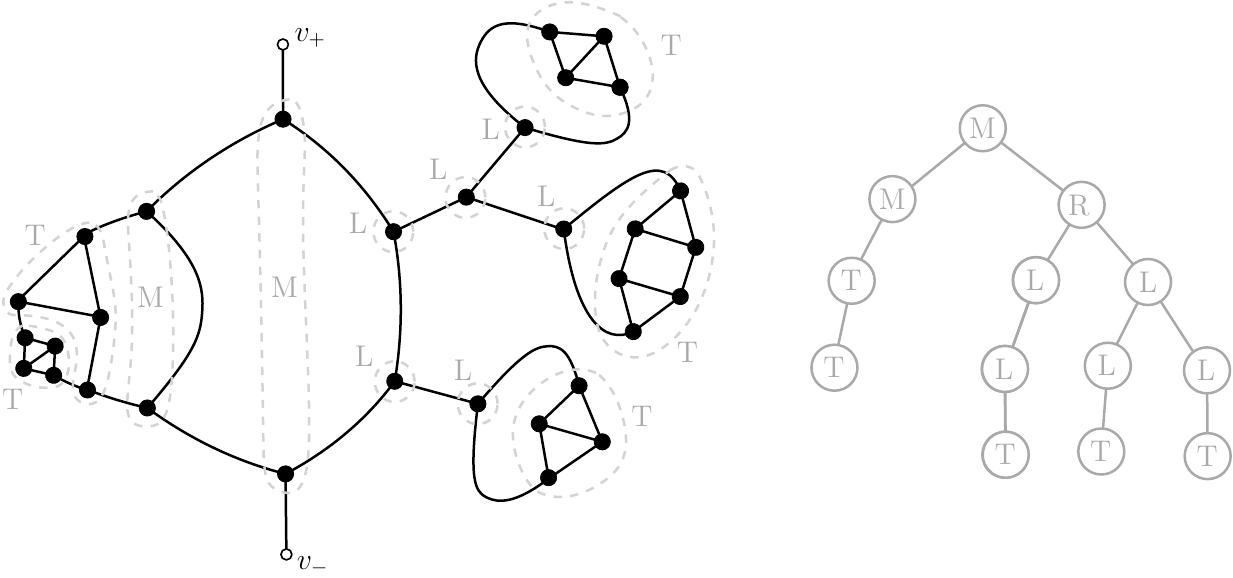}
\end{center}
\caption{Left: a connected (simple) cubic graph $\mg$, realized as a cubic network upon rooting at a (non-isthmus) edge.  
Right: the associated decomposition-tree.\\
The tree is considered as unrooted: it is intrinsic to $\mg$ (i.e., it does not depend on the root-choice),  which is also 
reflected by the difference in the choice of letters (a series-decomposition in the rooted setting corresponds to a ``ring'' of components in the unrooted setting, giving a node of type $R$). 
The left part also shows the partition of the vertices of $\mathrm g$ by components (tree-nodes) of type in $\{\lablm,\labmm,\labtm\}$.}
\label{fig:network}
\end{figure}

\index{h@\textbf{cubic networks}!tree@\ttt{$\tau(N)$}{(rooted) decomposition-tree}}
\index{h@\textbf{cubic networks}!tree@\ttt{$\tau(N)$}{(rooted) decomposition-tree}!L@\ttt{$\lablm$ label}{\netl}}
\index{h@\textbf{cubic networks}!tree@\ttt{$\tau(N)$}{(rooted) decomposition-tree}!R@\ttt{$\labrm$ label}{\netr}}
\index{h@\textbf{cubic networks}!tree@\ttt{$\tau(N)$}{(rooted) decomposition-tree}!M@\ttt{$\labmm$ label}{\netm}}
\index{h@\textbf{cubic networks}!tree@\ttt{$\tau(N)$}{(rooted) decomposition-tree}!T@\ttt{$\labtm$ label}{\nett}}
As detailed in~\cite{noy2020further} and illustrated in Figure~\ref{fig:network}, 
Lemma~\ref{lem:dec_network} yields a \emph{tree-decomposition} of 
 cubic networks, i.e., a rooted (unembedded) tree $\tau(\network)$ can be associated to a non-trivial cubic network $\network$, whose inner nodes have labels in $\{\lablm,\labrm,\labmm,\labtm\}$.  Precisely, if $\network$ is an \netr, of the form $\network_1+\cdots+\network_k$, then the root node of $\tau(\network)$ is labeled \labr and the hanging subtrees are $\tau(\network_1),\ldots,\tau(\network_k)$ (these are not \netr hence do not have root label \labr). 
 If $\network$ is an \netm or \nett, 
 then the root node is labeled \labm (resp. \labt) and the hanging subtrees are the trees associated to the substituted networks. 
Finally, for \netls, the root node has label \labl. In type 1, it has one child  also labeled \labl and of arity in $\{0,1\}$, depending on whether the loop-edge of $\mathrm{l}_1$ is unchanged or substituted by a non-trivial network $\mathrm N'$ (in which case the hanging subtree at the child corresponds to $\tau(\mathrm N')$). In type $2$, it has two children labeled \labl, each corresponding to one of the two loops of $\mathrm{l}_2$. Each child again has arity in $\{0,1\}$, depending on whether the corresponding loop is left unchanged or substituted by a non-trivial network.

%\index{3cgr@3-connected planar graphs!tree@\ttt{$\tau(\mg)$}{(unrooted) decomposition-tree}}%|see{cubnet!tree}}
Each \nodet of the tree $\tau(\network)$ corresponds to a certain rooted 3-connected planar graph $\mapk$ whose vertex-set is a subset of the vertex-set of $\network$. Then $\mapk$ is called
a \emph{3-connected component} of~$\network$.  The hanging subtrees of the node correspond to the (non-trivial) cubic networks substituted at its edges. 
 As detailed in~\cite{noy2020further}, the tree-decomposition and 3-connected components are intrinsic to the underlying unrooted cubic planar graph (essentially it means that 
 rerooting the cubic multigraph at another non-isthmus edge amounts to rerooting the same underlying unrooted tree, a fact established by Tutte~\cite{tutte2019connectivity} for the decomposition of graphs 
 into 3-connected components, which can then be specialized to the cubic case). Accordingly, to any connected cubic planar graph $\mg$, one can associate an unrooted decomposition tree, denoted $\tau(\mg)$.  
 Let us mention the following subtlety: when building the tree in the rooted case, every \noder has arity at least $2$, hence degree at least $3$ if not at the root and degree at least $2$ if at the root. In case the root node is of \typer of arity $2$, then when unrooting the tree the corresponding node of degree $2$ is smoothed out,  so that all \noders in the unrooted tree are of degree at least $3$.   
 The (unrooted) decomposition-tree satisfies the  following properties, which can be visualized in Figure~\ref{fig:network}:

\begin{claim}\label{claim:decomp_tree}
For $\mg$ a connected cubic planar multigraph and $\tau(\mg)$ the associated decomposition-tree, the vertices of $\mg$ are partitioned among the \compls, the \compms, and the \compts (3-connected components) of $\tau(\mg)$ (the \comprs have no contribution). 

The \compls are in one-to-one correspondence with the separating vertices of $\mg$, the \compms are in one-to-one correspondence with vertex-pairs $\{u,v\}$
such that $\mg$ is obtained by edge-substitution in the triple-bond connecting $u$ and $v$. 

The other vertices belong to the 3-connected components.  
Precisely, for $V'$ a vertex-subset of $\mg$ with $|V'|\geq 4$, and $\mathrm{k}$ a  3-connected planar graph (unrooted) with vertex-set $V'$, $\mathrm{k}$ is a 3-connected component of $\mg$ if and only if $\mg$ is $\mathrm{k}$ with networks substituted at the edges. 
\end{claim}

We focus our study first on \emph{simple} cubic planar graphs; accordingly the cubic networks are also simple (extension of the arguments to multigraphs will be given 
in Section~\ref{sec:mul}).  
In the network-decomposition, this yields the constraint
that loops (when dealing with an \netl at the current node) are to be substituted by non-trivial non-\netls, whereas when substituting in the double-edge network (when dealing
with a \netm at the current node), at least one of the two edges is to be substituted by a non-trivial network. On the decomposition-tree for unrooted graphs, being simple is equivalent to the fact that every leaf has label \labt, every \nodel of degree $2$ has exactly one neighbor of \typel, and every \nodem has degree in $\{2,3\}$; for networks the same constraints hold at non-root nodes, while the root node is allowed to be a leaf (i.e., to have arity $1$) of \typel or \typem. 

\index{h@\textbf{cubic networks}!simple@families of cubic networks (g.f. in straight capital), either simple or multigraph!LSPH@\ttt{$\kL, \kS, \kP, \kH$}{per type}}
\index{h@\textbf{cubic networks}!simple@families of cubic networks (g.f. in straight capital), either simple or multigraph!all@\ttt{$\kD$}{all}}
\index{h@\textbf{cubic networks}!simple@families of cubic networks (g.f. in straight capital), either simple or multigraph!K@\ttt{$\kK$}{polyhedral, obtained from 3-connected cubic planar graphs}}
We define families $\kF=\cup_n\kF_n$ of cubic simple networks (with $n$ the size-parameter), where $\kF$ is replaced by $\kL, \kS, \kP, \kH$ according to the type; we also define $\kD$ as the family of all simple cubic networks and $\kK$ as the family of polyhedral networks. 
 Each of these families is vertex-labeled, in the sense that the $2n$ non-pole vertices have distinct labels in $[1..2n]$ (on the other hand the two poles are unlabeled). We denote by $F(z)=\sum_n\frac1{(2n)!}|\kF_n|z^n$ the associated generating function, replacing $F$ by the corresponding letter for each family. 

We can express $K$ in the following way. 
Let $\kT_n$ be the set of rooted simple planar triangulations with $n+2$ (unlabeled) vertices (and $2n$ faces), and $T(x)=\sum_{n\geq  2}|\kT_n|x^n$ the corresponding ordinary generating function, which is known~\cite{tutte1962census} to have the following exact expression:  \index{t@\textbf{rooted simple planar triangulations}!set@\ttt{$\kT_n$}{with $n+2$ vertices and $2n$ faces}}
\begin{equation}\label{exp:T}
  T(x)=\sum_{n\geq 2}\frac{2}{n(n-1)}\binom{4n-3}{n-2}x^n=\xi^2(1-3\xi+\xi^2),
  \end{equation}
  where $\xi$ is the unique formal power series in $x$ with constant term equal to 0 and defined by $x=\xi(1-\xi)^3$.

  By duality, \emph{simple} planar triangulations correspond to \emph{3-connected} cubic planar maps. Thus, applying Whitney's theorem and taking vertex-labeling into account, we have $(2n)!|\kT_n|=2|\kK_n|$, hence $K(x)=T(x)/2$.

Then a decomposition at the root of the decomposition tree of cubic networks yields the following system (from which the coefficients can be extracted):
\begin{equation}\label{eq:syst_D}
\left\{
\begin{array}{rcl}
D(z)&=&1+L(z)+S(z)+P(z)+H(z),\\[.1cm]
L(z)&=&\frac1{2}L(z)^2+\frac1{2}z(D(z)-1-L(z)),\\[.1cm]
S(z)&=&(D(z)-1-S(z))(D(z)-1),\\[.1cm]
P(z)&=&\frac1{2}z(D(z)^2-1),\\[.1cm]
H(z)&=&\frac{1}{D(z)}K\big(zD(z)^3\big)=\frac{1}{2D(z)}T\big(zD(z)^3\big).
\end{array}
\right.
\end{equation}
In the second line the two terms correspond respectively to the root \nodel having arity $2$ or $1$. The last line (with $T(x)$ given in~\eqref{exp:T}) relies on the fact that a cubic network of size $n$ has $3n-1$ plain edges. %We recall that $T$ is the ordinary generating function of rooted simple planar triangulations, given in~\eqref{exp:T}.

\index{f@\textbf{vertex-labeled cubic planar graphs} (replace $C$ by $M$ for multigraphs)!family@\ttt{$\kC_n$}{with $2n$ vertices}}
\index{f@\textbf{vertex-labeled cubic planar graphs} (replace $C$ by $M$ for multigraphs)!pointed@\ttt{$\kC_n^\bullet$}{of $\kC_n$ with one marked vertex}}
Let $\kC=\cup_n\kC_n$ be the family of vertex-labeled connected cubic planar graphs of size $n$, i.e., the $2n$ vertices carry  
distinct labels in $[1..2n]$. Let $C(z)=\sum_{n\geq 1}\frac1{(2n)!}|\kC_n|z^n$ be the associated exponential generating function. 
Using Claim~\ref{claim:decomp_tree}, one obtains the following expression for the exponential generating function $\Cp(z)=2zC'(z)$ of pointed (i.e., with a marked vertex) cubic planar graphs~\cite{noy2020further}:
\begin{equation}\label{eq:Cp}
\Cp(z)=\Big(\frac1{2}\big(D(z)-1-L(z)\big)L(z)+\frac1{6z}L(z)^3\Big)+\frac{z}{6}(D(z)^3-3D(z)+2)+\frac1{3}G(z),
\end{equation}
where we define $G(z):=K(zD(z)^3)$.
The three groups correspond to the marked vertex being involved in an \compl (with the two terms corresponding to the marked vertex
yielding two or three connected components upon its removal), an \compm, and a \compt respectively (the factor $1/3$ in front of $G(z)$ is due to having a marked vertex instead of a marked directed edge). 

\index{h@\textbf{cubic networks}|)}

\subsection{Largest 3-connected component}\label{sec:largest_3comp}

We now use singularity analysis of the involved generating functions, in order to show (in this subsection and the next one) that a random cubic connected planar graph almost surely has a unique giant 3-connected
component (of linear size), and to establish related properties. 
This analysis follows from a well-established methodology~\cite{banderier2001random,gao1999size}.  
Similar results appear in the recent preprint~\cite{stufler2022uniform} (in particular Theorem 1.2). 
We  provide here our own analysis (thereby
explaining the mild differences with the statements and proof of~\cite{stufler2022uniform}) for the sake of completeness, and since we will extend these results to cubic planar multigraphs in Section~\ref{sec:mul}.

\index{c@\textbf{constants}!rho@$\rho \approx 0.101905$ (simple) or $=54/79^{3/2}$ (multigraph)}
%If follows from~\cite{BKLM,noy2020further,fang2016cubic} 
It follows from \cite[Sec.3.1]{noy2020further} that the generating functions $F(z)$ with $F\in\{L,D,G,\Cp\}$ have the same radius of convergence $\rho\approx 0.101905$ and 
these series have singular expansion (in a slit neighborhood) of the form\footnote{In~\cite{noy2020further}, the size-parameter is the number of vertices, whereas it is half the number of vertices here. The analysis is completely similar under this modification, they have dominant singularities at $\pm\sqrt{\rho}$ whereas we have a single dominant singularity at $\rho$. 
}:
\begin{equation}\label{eq:asympt_form}
F(z)=F_0-F_2Z^2+F_3Z^3+O(Z^4),\ \ \  \mathrm{where}\ Z=\sqrt{1-z/\rho}.
\end{equation}
Moreover, these series are analytically continuable to a domain $\{|z|\leq \rho+\eta,\ z-\rho\notin\R_+\}$ for some $\eta>0$.  
By transfer theorems of analytic combinatorics~\cite[VI.3]{flajolet2009analytic}, we have $[z^n]F(z)\sim \kappa \rho^{-n}n^{-5/2}$, where $\kappa=\frac{3F_3}{4\sqrt{\pi}}$.

We say that a pointed connected cubic planar graph is of \typet if the marked vertex is incident to a \compt. 
According to~(\ref{eq:Cp}), this occurs with probability 
\begin{equation}\label{eq:proba_T}
[z^n]\frac1{3}G(z)/[z^n]\Cp(z)\sim \frac1{3}G_3/\Cp_3.
 \end{equation}
 
Let $E(z)=zD(z)^3$ (note that $E(z)$ also has a singular expansion of the form $E(z)=E_0-E_2Z^2+E_3Z^3+O(Z^4)$), 
so that the generating function of cubic connected planar graphs of \typet is 
\[
\frac1{3}G(z)=\frac1{3}K(E(z)).
\] 
As shown in~\cite[Sec.3.1]{noy2020further}, 
this composition scheme is critical, in the sense that the radius of convergence of $y\to K(y)$ is $y_0=\rho D_0^3$, and $K(y)$ has a 
singular expansion of the form $K(y)=K_0-K_2Y^2+K_3Y^3+O(Y^4)$, with $Y=\sqrt{1-y/y_0}$. We can now use the results from~\cite{banderier2001random}. Let $A(x)$ be the 
density function defined by
\[
A(x)=\frac1{\pi}\sum_{k\geq 1}(-1)^{k-1}x^k\frac{\Gamma(1+2k/3)}{\Gamma(1+k)}\mathrm{sin}(2\pi k/3). 
\]
For $c>0$, the probability distribution of density $cA(cx)$ is called \emph{Airy-map} distribution of parameter $c$.

\index{c@\textbf{constants}!alpha@$\alpha \approx 0.8509$ (simple) or $=199/316$ (multigraph)}
% \index{c@\textbf{constants}!c@$c \approx 1.5190$ (simple) or $=\frac{2}{3}\left(\frac{79}{17}\right)^{2/3}$ (multigraph)}
%\index{c@\textbf{constants}!plp@$p'_{\ell}$}
\begin{proposition}[Apply Theorem~5 in~\cite{banderier2001random}]\label{prop:core}
Define the constants
\[p'_{\ell}=1-\frac{E_3G_2}{E_2G_3},\ \ \alpha=\frac{E(\rho)}{\rho E'(\rho)}=\frac{E_0}{E_2},\ \ c=\frac1{\alpha}\left(\frac{E_2}{3E_3}\right)^{2/3}.\] 
\index{f@\textbf{vertex-labeled cubic planar graphs} (replace $C$ by $M$ for multigraphs)!randomp@\ttt{$\cnp$}{uniform in the subset of $c\in\kC_n^\bullet$ of \typet}}
For $\cnp$ the random pointed cubic 3-connected planar graph of \typet, 
let $X_n$ be the size of the 3-connected component incident to the marked vertex.  
Then, for any $x\in\R$ (and uniformly in any compact set of $\R$)
\[
P(X_n=\lfloor \alpha n + x n^{2/3}\rfloor) =  p'_{\ell}\ \!n^{-2/3}\ \!cA(cx)\cdot(1+o(1)).
\]
\end{proposition}
\begin{remark}
The quantities $\alpha,c$ are explicitly computable, $\alpha\approx 0.8509$ is root of the polynomial 
{\footnotesize
\[
\begin{array}{c}
13026164315213824\, x^6 - 10328139725010432\, x^4 - 1810655504690048\, x^3 \\
+ 1409382842895504\, x^2 + 591702923494104\, x + 62103840598801,
\end{array}
\]
}
and $c\approx 1.5190$ is root of the polynomial 
{\scriptsize
\[
\begin{array}{c}
225719792049828011973129\,{x}^{18}-25997088298647780811521072\,{x}^{15
}+1072598816640412427380432704\,{x}^{12}\\
-29427138982726419066720150528
\,{x}^{9}+347939923194525847286624477184\,{x}^{6}\\
-1656815421544715458705316511744\,{x}^{3}+
2651264949483594255344169385984.
\end{array}
\]
}
Compared to~\cite[Theo 1.2]{stufler2022uniform} the constant $c$ differs by a factor $2^{1/3}$; this is due to the size $n$ being the number of vertices in~\cite{stufler2022uniform} 
whereas it is half the number of vertices
here. \dotfill
\end{remark}

The \emph{3-connected core} of a cubic connected planar multigraph is its largest $3$-connected component (if there is a tie, we break the tie by choosing the one whose minimal
label is smaller). We now obtain: 
\index{f@\textbf{vertex-labeled cubic planar graphs} (replace $C$ by $M$ for multigraphs)!random@\ttt{$\cn$}{uniform in $\kC_n$}}
\begin{theorem}\label{theo:largest3comp} 
For $\cn$ the random connected cubic planar graph with $2n$ vertices,  
let $X_n^*$ be the size of the 3-connected core of $\cn$.  
Then, for any $x\in\R$ (and uniformly in any compact set of $\R$), we have, with the constants $\alpha,c$ defined in Proposition~\ref{prop:core}:
\begin{equation}\label{eq:Xns}
P(X_n^*=\lfloor \alpha n + x n^{2/3}\rfloor) =  n^{-2/3}\ \!cA(cx)\cdot(1+o(1)).
\end{equation}
\end{theorem}

Note that Theorem \ref{theo:largest3comp} ensures that $X_n^*$ is concentrated around $\alpha n$ with fluctuations of order $n^{2/3}$,  
and that $(X_n^*-\alpha n)/n^{2/3}$ converges in law to the Airy-map distribution of parameter $c$.

\begin{proof}
Let $\kCnp$ be the set of pointed connected cubic planar graphs with $2n$ vertices, and let $\cnp$ be uniformly random in $\kCnp$. 
In view of~(\ref{eq:proba_T}), if we define 
\begin{equation}
p_{\ell}:=\frac{G_3}{3\Cp_3}p_{\ell}'\ =\ \frac{1}{3\Cp_3}\left(G_3-\frac{E_3G_2}{E_2}\right),
\end{equation}
and define $X_n$ to be $0$ if $\cnp$ is not of \typet, and $X_n$ to be the size of the 3-connected core otherwise,  
then by \eqref{eq:proba_T} and Proposition \ref{prop:core}, we have for any $x\in\R$ (and uniformly in any compact set of $\R$), 
\[
P(X_n=\lfloor \alpha n + x n^{2/3}\rfloor) =  n^{-2/3}\ \!p_{\ell}\ \!cA(cx)\cdot(1+o(1)).
\]
To show~(\ref{eq:Xns}) we rely on a re-rooting argument given in~\cite[Appendix D]{banderier2001random}. 
The graphs with a giant 3-connected component of size $\approx \alpha n$ incident to the marked vertex are in proportion $\sim p_{\ell}$ among $\kCnp$, and they 
 admit $\sim 1/\alpha$ re-rootings at an arbitrary vertex. In addition, they produce different objects in $\kCnp$, because, as observed in~\cite{stufler2022uniform}, the fact that $\alpha>1/2$  ensures that there is at most one such component within the graph, for $n$ large enough.

 Thus $p_{\ell}/\alpha\leq 1$, and to conclude that almost all graphs in $\kCnp$ are of that form it  remains to show that $p_{\ell}=\alpha$. 
 Here we proceed a bit differently from~\cite{stufler2022uniform}.   Rather than computing the precise evaluations of $\alpha$ and $p_{\ell}$, 
 we will show that each of $\alpha,p_{\ell}$ 
 has a rational expression in terms of $\rho, L_0, L_2$ and that these expressions are the same. Thereby  our calculations do not require any evaluations of the involved quantities, treated as formal symbols, nor the exact expression of $K(x)=T(x)/2$ given in~\eqref{exp:T}; we just use the fact that the series in $z$ have a singular expansion of the form~(\ref{eq:asympt_form}). (Thus the coincidence of $p_{\ell}$ and $\alpha$ can be seen as a ``generic'' property of the decomposition grammar.)

 The second line of~(\ref{eq:syst_D}) yields for each $i\in\{0,2,3\}$ a rational expression for $D_i$ in terms of $\{\rho,L_0,L_i\}$, precisely
 \[
 D_0=1+\frac{L_0(2+\rho-L_0)}{\rho},\ \  D_2=\frac{L_0^2 - 2L_0L_2 + L_2\rho - 2L_0 + 2L_2}{\rho},\ \  D_3=\frac{L_3(2+\rho-2L_0)}{\rho}.
 \]
Since $E(z)=zD(z)^3$, this directly yields for each $i\in\{0,2,3\}$ a rational expression for $E_i$ in terms 
 of $\{\rho,L_0,L_i\}$, e.g. $E_3= 3L_3(2+\rho- 2L_0)(L_0^2 - L_0\rho - 2L_0 - \rho)^2/\rho^2$. 
 We can then inject the expressions of $E_0$ and $E_2$ into $\alpha=\frac{E_0}{E_2}$, and find
 \[
 \alpha=\frac{\rho+L_0\rho + 2L_0 -L_0^2}{2L_0^2 - 6L_0L_2 + L_0\rho + 3L_2\rho - 4L_0 + 6L_2 + \rho}.
 \]
It remains to find a similar rational expression for $p_{\ell}=\frac{1}{3\Cp_3}\left(G_3-\frac{E_3G_2}{E_2}\right)$. From~(\ref{eq:syst_D}) we extract the equation 
(noting that the 3rd line yields $S(z)=(D(z)-1)^2/D(z)$)
 \[
 D(z)=\displaystyle 1+ L(z)+\frac{(D(z)-1)^2}{D(z)}+\frac1{2}z(D(z)^2-1)+\frac{1}{D(z)}G(z),
 \]
 from which we extract for each $i\in\{0,2,3\}$ a (quite larger) rational expression for $G_i$ in terms 
 of $\{\rho,L_0,L_i\}$.   From~(\ref{eq:Cp}) we can then extract a rational expression for $\Cp_3$ in terms of $\{L_0,L_3,\rho\}$. From the expressions of $E_2,E_3,G_2,G_3$,
   we obtain a rational expression of $p_{\ell}$ that exactly matches the one of $\alpha$ given above. Thus $p_{\ell}=\alpha$. 
\end{proof}

\subsection{Cubic planar graphs with a marked 3-connected component}\label{sec:marked_3c}

\index{f@\textbf{vertex-labeled cubic planar graphs} (replace $C$ by $M$ for multigraphs)!q-comp@\ttt{$\kC_n^{(q)}$}{of $\kC_n$ with a marked 3-connected component with $q$ vertices}}
\index{f@\textbf{vertex-labeled cubic planar graphs} (replace $C$ by $M$ for multigraphs)!q-comp-rand@\ttt{$\cnq$}{uniform in $\kC_n^{(q)}$}}
For two integers $n,q$, let $\kC_n^{(q)}$ be the set of connected cubic planar graphs of size $n$ having a marked 3-connected component of size $q$, and
denote by  $\cnq$ a uniformly random element in $\kC_n^{(q)}$. 
It will also prove convenient for us to see $\cnq$ as closely related to a model with more independence. The \emph{critical Boltzmann distribution} on $\kD$ is the probability distribution 
\index{boltzmann@\ttt{$\mu^{(\kD)}$}{critical Boltzmann distribution on the combinatorial family $\kD$}}
$\mu^{(\kD)}$  defined as  
\begin{equation}
\mu^{(\kD)}(\mathrm{d})=\frac{\rho^{|\mathrm{d}|}}{(2|\mathrm{d}|)!D(\rho)}\ \ \ \forall\ \mathrm{d}\in\kD,
\end{equation}
where we recall that $\rho$ is the radius of convergence of $D(z)$.

\index{f@\textbf{vertex-labeled cubic planar graphs} (replace $C$ by $M$ for multigraphs)!q-random@\ttt{$\cq$}{generated by substituting the edges of a uniform 3-connected graph of size $q$}}
Let $\cq$ be a random connected cubic planar graph with a marked 3-connected component  of size $q$ constructed in the following way. Letting $\mathrm{K}$ be a uniformly random 3-connected graph of size $q$, we replace (independently) each edge of $\mathrm{K}$ by a cubic network sampled from $\mu^{(\kD)}$. 
The resulting graph is $\cq$. Note that, for any $n,q\geq 1$, $\cnq$ is distributed as $\cq$ conditioned to have size $n$.

% We let $\cq$ be the random connected cubic planar graph with a marked 3-connected component $\mathrm{K}$ of size $q$, each edge of $\mathrm{K}$ being independently replaced by a cubic network drawn under $\mu^{(\kD)}$. Note that, for any $n,q\geq 1$, $\cnq$ is distributed as $\cq$ conditioned to have size $n$.  
As stated next, the conditioning probability is polynomial when $q=\alpha n+O(n^{2/3})$:

\begin{lemma}\label{lem:size_core}
For every $C>0$, we can find $0<c_- < c_+<\infty$ such that for every $q$ with $|q-\alpha n|\leq C n^{2/3}$, we have $n^{2/3}\mathbb{P}(|\cq|=n) \in(c_-,c_+)$. 
\end{lemma}
Informally, this local limit theorem-type result ensures that when $q,n$ are related by $q=\alpha n+O(n^{2/3})$, then $\mathbb{P}(|\cq|=n)=\Theta(n^{-2/3})$. We will
often use such an informal notation thereafter.   
\begin{proof}
Let $n^{(q)}$ be the size of $\cq$. With $E(z)=zD(z)^3$, we have
\[
\mathbb{P}(n^{(q)}=n)=[u^n]p(u)^q,\ \ \mathrm{where}\ p(u):=\frac{E(\rho u)}{E(\rho)}.
\]
With the notation~\eqref{eq:asympt_form}, we have
\[
p(u)=1+\frac1{\alpha}(u-1)+\frac{E_3}{E_0}(1-u)^{3/2}+O((u-1)^2).
\]

A local limit theorem (for a stable law of parameter $3/2$) holds~\cite[Theo.11(ii)]{banderier2001random}: for $n=\frac1{\alpha}q-xq^{2/3}$, uniformly in $x$ on any compact interval, we have
\[
\mathbb{P}(n^{(q)}=n)\sim \frac{\tilde{c}A(\tilde{c}x)}{q^{2/3}},
\]
with $\tilde{c}=(E_0/E_3)^{2/3}$, and $\tilde{c}A(\tilde{c}x)$ the Airy-map distribution of parameter $\tilde{c}$. 
\end{proof}

Let  $\mathrm{C}$ be a connected cubic planar graph with a marked 3-connected component $\mathrm{K}$. Note that $\mathrm{C}$ is obtained from $\mathrm{K}$ by substituting every edge of $\mathrm{K}$  by a cubic network; such a network is called an \emph{attached network of $\mathrm{C}$}. 

\begin{lemma}\label{lem:cnq}
Let $\epsilon>0$. 
When $q,n$ are related by $q=\alpha n+O(n^{2/3})$, then a.a.s. all attached networks of $\cnq$ have size at most $n^{2/3+\epsilon}$. 
\end{lemma}
\begin{proof}
Let $\epsilon'=\epsilon/2$. 
We argue via $\cq$ and conditioning on the size.  
The $3q$ edges of the marked 3-connected component can be canonically ordered (e.g. lexicographic ordering on the pair of labels of the incident vertices), which yields a canonical ordering $M_1,\ldots,M_{3q}$ of the $3q$ attached networks. In $\cq$ these are independently drawn under $\mu^{(\kD)}$, and the size of $\cq$ is $|\cq|=|M_1|+\cdots+|M_{3q}|$. 
Using~\cite[Theo.1(iii)-(b)]{gao1999size}, it follows that 
\[
\mathbb{P}\left(|\cq|< \frac1{\alpha}q-q^{2/3+\epsilon'}\right)=O(\exp(-q^{\epsilon'})).
\] 
(a large deviation result, obtained by a saddle-point estimate).

Similarly, for any $i\in[1..3q]$, we have 
$\mathbb{P}(|\cq|-|M_i|<\frac1{\alpha}q-q^{2/3+\epsilon'})=O(\exp(-q^{\epsilon'}))$ (we apply the same argument to $(3q-1)$ attached networks instead of $3q$). By the union bound, the probability that $|\cq|-|M_i|< \frac1{\alpha}q-q^{2/3+\epsilon'}$ for some $i\in[1..3q]$ is $O(q\exp(-q^{\epsilon'}))$. 
When conditioning on $\cq$ to have size $n$, the event that $|\cq|-|M_i|<\frac1{\alpha}q-q^{2/3+\epsilon'}$ for some $i\in[1..3q]$ is the same as $|M_i|>n-\frac1{\alpha}q+q^{2/3+\epsilon'}$ for some $i\in[1..3q]$.  Hence, when $q,n$ are related by $q=\alpha n+O(n^{2/3})$, Lemma~\ref{lem:size_core} ensures that the probability that $|M_i|>n-\frac1{\alpha}q+q^{2/3+\epsilon'}$ for some $i\in[1..3q]$ is $O(n^{2/3}q\exp(-q^{\epsilon'}))=o(1)$. Moreover, for $n$ large enough, this event is implied by the event that  $|M_i|>n^{2/3+\epsilon}$ for some $i\in[1..3q]$. 
\end{proof}

\begin{remark}
For $q=\alpha n + O(n^{2/3})$, as already observed in~\cite{stufler2022uniform} the fact that $\alpha>1/2$ ensures that for $n$ large enough $\cnq$ coincides with the random connected cubic planar graph $\cn$ conditioned to have its $3$-connected core of size $q$. \dotfill
\end{remark}

\subsection{Extension of the results to cubic planar multigraphs}\label{sec:mul}

\index{h@\textbf{cubic networks}!simple@families of cubic networks (g.f. in straight capital), either simple or multigraph!LSPH@\ttt{$\kL, \kS, \kP, \kH$}{per type}}
\index{h@\textbf{cubic networks}!simple@families of cubic networks (g.f. in straight capital), either simple or multigraph!K@\ttt{$\kK$}{polyhedral, obtained from 3-connected cubic planar graphs}}
In the context of multigraphs, it is somehow more natural (e.g. for kernel extraction relations) to consider that the half-edges are labeled instead of the vertices\footnote{This is equivalent to 
considering vertex-labeled multigraphs with a certain compensation-factor~\cite{janson1993birth}, which for cubic graphs with more than $2$ vertices is $2^{-\#\mathrm{loops}-\#\mathrm{double\ edges}}$.}. 
Thus, let $\kM_n$ denote the set of half-edge-labeled connected cubic planar multigraphs of size $n$ (i.e., with $2n$ vertices), and $\mn$ denotes an element of $\kM_n$ taken uniformly at random.
Regarding cubic networks (allowing for loops and multiedges), we consider that, in size $n$ (i.e., with $2n$ non-pole vertices), 
the $6n-2$ half-edges on plain edges carry distinct labels in $[1..6n-2]$. For each family $\kF$ of cubic networks, with $\kF\in\{\kD,\kL,\kS,\kP,\kH\}$, we let $F(z)=\sum_n\frac1{(6n-2)!}|\kF_n|z^n$
be the associated generating function. 
In this context, the equation-system~\eqref{eq:syst_D} (again obtained from a decomposition at the root of the associated tree) becomes
\begin{equation}\label{eq:syst_D_multigraphs}
\left\{
\begin{array}{rcl}
D(z)&=&1+L(z)+S(z)+P(z)+H(z),\\[.1cm]
L(z)&=&\frac1{2}L(z)^2+\frac1{2}z\,D(z),\\[.1cm]
S(z)&=&(D(z)-1-S(z))(D(z)-1),\\[.1cm]
P(z)&=&\frac1{2}z\,D(z)^2,\\[.1cm]
H(z)&=&\frac{1}{D(z)}K\big(zD(z)^3\big)=\frac{1}{2D(z)}T\big(zD(z)^3\big).
\end{array}
\right.
\end{equation}
 We warn the reader of our abuse of notations, as $D,L,S,P$ and $H$ all denote different objects in the context of graphs and multigraphs; we do so because of the extreme similarity of the proofs and results. Which object we consider should be clear from context.   
Now, with $M(z)=\sum_{n\geq 1}\frac1{(6n)!}|\kM_n|z^n$ and $\Mp(z)=2z\frac{\mathrm{d}}{\mathrm{d}z}M(z)$, the analogue of~\eqref{eq:Cp} is
\begin{equation}\label{eq:Cp_mult}
\Mp(z)=\frac1{2}D(z)\,L(z)+\frac1{6z}L(z)^3+\frac{z}{6}D(z)^3+\frac1{3}G(z),
\end{equation}
where $G(z)=K(E(z))$, with $E(z)=zD(z)^3$ (the series $K(x)$ is the same as before, since $3$-connected graphs are simple).  

\index{c@\textbf{constants}!rho@$\rho \approx 0.101905$ (simple) or $=54/79^{3/2}$ (multigraph)}
We can now rely on the singularity analysis performed in~\cite{fang2016cubic,kang2012two} to obtain the analogue of Theorem~\ref{theo:largest3comp}.  
The generating functions $F(z)$ with $F\in\{L,D,E,G,\Mp\}$ have the same radius of convergence $\rho=54/(79)^{3/2}$ (we still use the letter $\rho$ even if the radius of convergence for multigraphs is different than for simple graphs), and 
these series have singular expansion (in a slit neighborhood) of the form
\begin{equation}\label{eq:asympt_form_multi}
F(z)=F_0-F_2Z^2+F_3Z^3+O(Z^4),\ \ \  \mathrm{where}\ Z=\sqrt{1-z/\rho}.
\end{equation}
As in the case of graphs, the composition scheme in~\eqref{eq:syst_D_multigraphs} is critical, so that we can apply Theorem~5 from~\cite{banderier2001random}, exactly 
in the same way as in Proposition~\ref{prop:core}. Defining the constants
\[p'_{\ell}=1-\frac{E_3G_2}{E_2G_3},\ \ \alpha=\frac{E(\rho)}{\rho E'(\rho)}=\frac{E_0}{E_2},\ \ c=\frac1{\alpha}\left(\frac{E_2}{E_3}\right)^{2/3},\]
we obtain that,  for $\mnp$ the random pointed connected cubic planar multigraph of size $n$, conditioned to have the marked vertex on a 3-connected component,  
the size $X_n$ of the 3-connected component at the marked vertex satisfies
 for any $x\in\R$ (and uniformly in any compact set of $\R$)
\[
P(X_n=\lfloor \alpha n + x n^{2/3}\rfloor) =  p'_{\ell}\ \!n^{-2/3}\ \!cA(cx)\cdot(1+o(1)).
\]
\index{c@\textbf{constants}!alpha@$\alpha \approx 0.8509$ (simple) or $=199/316$ (multigraph)}
% \index{c@\textbf{constants}!c@$c \approx 1.5190$ (simple) or $=\frac{2}{3}\left(\frac{79}{17}\right)^{2/3}$ (multigraph)}
The constants $\alpha,c$ are again explicitly computable, and we find the following values, much simpler than in the case of graphs:
\begin{equation}\label{eq:alpha_c}
\alpha=\frac{199}{316}\approx 0.6297, \ \ \ c=\frac{2}{3}\left(\frac{79}{17}\right)^{2/3}\approx 1.8565.
\end{equation}
   
We can now obtain the analogue of Theorem~\ref{theo:largest3comp}.

\index{f@\textbf{vertex-labeled cubic planar graphs} (replace $C$ by $M$ for multigraphs)!family@\ttt{$\kC_n$}{with $2n$ vertices}}
\index{f@\textbf{vertex-labeled cubic planar graphs} (replace $C$ by $M$ for multigraphs)!random@\ttt{$\cn$}{uniform in $\kC_n$}}
\begin{theorem} 
For $\mn$ a uniformly random connected cubic planar multigraph of size $n$,     
let $X_n^*$ be the size of its $3$-connected core. 
Then, for any $x\in\R$, and uniformly in any compact set of $\R$,  
\begin{equation}\label{eq:Xns_multi}
P(X_n^*=\lfloor \alpha n + x n^{2/3}\rfloor) =  n^{-2/3}\ \!cA(cx)\cdot(1+o(1)),
\end{equation}
where $\alpha,c$ are given in~\eqref{eq:alpha_c}.
\end{theorem}
Again, this implies that $X_n^*$ is concentrated around $\alpha\,n$, and $(X_n^*-\alpha\, n)/n^{2/3}$ converges to the Airy-map distribution of parameter $c$.  
\begin{proof}
Similarly as in Theorem~\ref{theo:largest3comp}, we have to check that the constant $p_{\ell}:=\frac{G_3}{3\Mp_3}p_{\ell}'$ is equal to $\alpha$
(again we do not need the precise algebraic equations of the involved series, the proof is ``generic'' and does not require the expression of $K(x)$). The second
line of~\eqref{eq:syst_D_multigraphs} yields, for each $i\in\{0,2,3\}$, a rational expression of $D_i$   
in terms of $\{\rho, L_0,L_i\}$:
\[
D_0=\frac{L_0(2-L_0)}{\rho},\ \ D_2=\frac{L_0^2 - 2L_0L_2 - 2L_0 + 2L_2}{\rho},\ \  D_3=\frac{2L_3(1-L_0)}{\rho}.
\] 
Using $E(z)=zD(z)^3$, this yields, for each $i\in\{0,2,3\}$, a rational expression for $E_i$ in terms of $\{\rho,L_0,L_i\}$, which 
gives a rational expression for $\alpha=\frac{E_0}{E_2}$. We find
\[
\alpha=\frac{L_0(2 - L_0)}{2(L_0^2 - 3L_0L_2 - 2L_0 + 3L_2)}.
\]
It remains to find a similar rational expression for $p_{\ell}=\frac1{3\Mp_3}\Big(G_3-\frac{E_3G_2}{E_2}\Big)$. 
From~\eqref{eq:syst_D_multigraphs} (note that the third line yields $S(z)=(D(z)-1)^2/D(z)$),   we extract the equation
\[
D(z)=1+L(z)+\frac{(D(z)-1)^2}{D(z)}+\frac1{2}zD(z)^2+\frac1{D(z)}G(z),
\]
which yields, for $i\in\{0,2,3\}$, a rational expression for $G_i$ in terms of $\{\rho,L_0,L_i\}$. Finally, we can extract
from~\eqref{eq:Cp_mult} a rational expression for $\Mp_3$ in terms of $\{\rho,L_0,L_3\}$. Injecting the expressions
of $\Mp_3,G_2,G_3,E_2,E_3$ into $p_{\ell}$, we obtain a rational expression for $p_{\ell}$ that exactly matches the one
of $\alpha$ given above. Hence $p_{\ell}=\alpha$. 
\end{proof}
\index{f@\textbf{vertex-labeled cubic planar graphs} (replace $C$ by $M$ for multigraphs)!q-comp@\ttt{$\kC_n^{(q)}$}{of $\kC_n$ with a marked 3-connected component with $q$ vertices}}
\index{f@\textbf{vertex-labeled cubic planar graphs} (replace $C$ by $M$ for multigraphs)!q-comp-rand@\ttt{$\cnq$}{uniform in $\kC_n^{(q)}$}}
\index{f@\textbf{vertex-labeled cubic planar graphs} (replace $C$ by $M$ for multigraphs)!q-random@\ttt{$\cq$}{generated by substituting the edges of a uniform 3-connected graph of size $q$}}
Similarly as for cubic graphs, for two integers $n,q$, we let $\kM_n^{(q)}$ be the set of connected cubic planar multigraphs of size $n$ having a marked 3-connected component of size $q$, and let $\mnq$ be a uniformly random element in  $\kM_n^{(q)}$. 
And we denote by $\mq$ the random connected cubic planar multigraph with a marked 3-connected component $\mathrm{K}$ of size $q$, each edge of $\mathrm{K}$ being independently replaced by a cubic network drawn under the \emph{critical Boltzmann distribution} on $\kD$, i.e., the probability distribution 
$\mu^{(\kD)}$ ($\kD$ denoting here the family of cubic networks allowing for double edges and loops) defined as  
\begin{equation}
\mu^{(\kD)}(\mathrm{d})=\frac{\rho^{|\mathrm{d}|}}{(2|\mathrm{d}|)!D(\rho)}\ \ \ \forall\ \mathrm{d}\in\kD,
\end{equation}
where $\rho$ is the radius of convergence of $D(z)$. As before, $\mnq$ is $\mq$ conditioned to have size~$n$. 
By a very similar analysis as in Section~\ref{sec:marked_3c}, when $n,q$ are related by $q=\alpha n+O(n^{2/3})$, the probability that $\mq$ has size $n$ is $\Theta(n^{-2/3})$; and for $\epsilon>0$, 
 a.a.s. every cubic network attached at the marked 3-connected component has size smaller than $n^{2/3+\epsilon}$. Moreover, since $\alpha>1/2$, for $q=\alpha n+O(n^{2/3})$ and $n$ large enough, $\mnq$ coincides with the random cubic connected planar multigraph of size $n$ 
 conditioned to have its $3$-connected core of size $q$.

\section{First results about the metric properties of random cubic planar graphs}\label{sec:bounds}

\subsection{Pole-distance in cubic networks under the critical Boltzmann distribution}\label{sec:pole_dista}

\index{h@\textbf{cubic networks}!poledist@\ttt{$\delta(N)$}{pole distance}}
\index{h@\textbf{cubic networks}!poledist@\ttt{$\delta(N)$}{pole distance}!distr@\ttt{$\nustar$}{its distribution under the critical Boltzmann distribution}}
For $\network$ a cubic network, the \emph{pole-distance} of $\network$ is the distance between its two poles, it is denoted $\delta(N)$. We use the notation $\kD$
for the class of networks in both contexts (simple cubic networks, labeled at non-pole vertices, or cubic networks allowing for loops and multiedges, labeled at half-edges on plain edges). 
Let $\nustar$ be the distribution of $\delta(N)$, where $\network$ is drawn under $\mu^{(\kD)}$ (critical Boltzmann distribution on $\kD$).  
Our aim here is to show that $\nustar$ has an exponential tail, 
both in the case of graphs and multigraphs (the reason to show this is that $\nustar$ will be the law of edge-lengths to be considered on the $3$-connected core, and
our subsequent analysis of modified distances in $3$-connected cores assumes the law of edge-lengths to have exponential tails).

\begin{lemma}\label{lem:nustar}
There exist constants $a_\star,b_\star>0$ such that, for $X$ drawn under $\nustar$, 
\[
\mathbb{P}(X>k)\leq a_\star\exp(-b_\star\, k)\ \ \mathrm{for\ all\ }k\geq 0.
\]
\end{lemma}

\begin{proof}
Similarly as in ~\cite[Lem.5.6]{chapuy2015diameter}, the approach is to define another parameter $\chi$ on cubic networks that dominates $\delta$ and is easier to track along the network decomposition, and to  show via generating functions that $\chi$ has an exponential tail. 
For $N\in \kD$, $\chi(N)$ is defined as follows:
\begin{itemize}
\item
if $\network$ is the trivial network then $\chi(N)=1$,
\item
if $\network$ is of \typel then $\chi(N)=2$,
\item
if $\network$ is of \typer, of the form $\network_1+\network_2$ (with $\network_1$ a not-\labr network), then 
$\chi(N)=\chi(\network_1)+\chi(\network_2)$,
\item
if $\network$ is of \typem, with $\network_1$ and $\network_2$ the two networks substituted at each edge, then $\chi(N)=\chi(\network_1)+\chi(\network_2)+2$,
\item
if $\network$ has a 3-connected core $K$, and if $\network_1,\ldots,\network_k$ are the cubic networks substituted at the outer edges of $K$, then $\chi(N)=2+\chi(\network_1)+\cdots+\chi(\network_k)$.  
\end{itemize}
It is easy to check recursively along the decomposition that $\chi(N)\geq \delta(N)$ for any $N\in\kD$. 

For each family $\kF$ of networks, we let $F(z,u)$ be the bivariate generating function (with $z$ for the size and $u$ for the parameter $\chi$). 
We also write $T(z,v)$ for the bivariate generating function of rooted simple triangulations, with $v$ for the root degree minus $1$. By Remark \ref{rk:3conn}, $T(z,v)/2$ is the bivariate (exponential) generating function of 3-connected edge-rooted cubic planar graphs, with $v$ conjugate to the number of outer edges minus $1$. 
 For the case of graphs, the network decomposition gives 
\begin{equation}\label{eq:Dzu}
\left\{
\begin{array}{rcl}
D(z,u)&=&u + L(z,u)+S(z,u)+P(z,u)+H(z,u),\\
L(z,u)&=&u^2L(z),\\
S(z,u)&=&(D(z,u)-u-S(z,u))(D(z,u)-u),\\
P(z,u)&=&\frac{1}{2}u^2(D(z,u)^2-1),\\
H(z,u)&=&u^2\frac1{2D(z)}T(zD(z)^3,D(z,u)/D(z)),
\end{array}
\right.
\end{equation}
Hence, noting that $S(z,u)=\frac{(D(z,u)-u)^2}{D(z,u)-u+1}$, the series $D(z,u)$ is solution of the equation in $y$
\begin{equation}\label{eq:G}
A(z,u,y)=0,
\end{equation}
where
\[
A(z,u,y)=-y+u+u^2L(z)+\frac{(y-u)^2}{y-u+1}+\frac1{2}u^2(y^2-1)+\frac{u^2}{2D(z)}T(zD(z)^3,\frac{y}{D(z)}).
\]

We are now going to prove the property that there exists $u_0>1$ such that $D(\rho,u_0)$ converges; this will prove the lemma, since for $\network$ drawn under $\mu^{(\kD)}$ one has 
\[\mathbb{P}(\delta(N)\geq k)\leq \mathbb{P}(\chi(N)\geq k)=\sum_{r\geq k}\frac{[u^r]D(\rho,u)}{D(\rho)}\leq \frac{D(\rho,u_0)}{D(\rho)}u_0^{-k}.\]
What we will obtain more precisely is that, in a complex neighborhood of $(\rho,1)$ (with $\rho$ the radius of convergence of $D(z)$), we have a singular expansion for $D(z,u)$ of the form
\[
D(z,u)=\widetilde{D}(Z,U),\ \ \mathrm{with}\ Z=\sqrt{1-z/\rho}\ \mathrm{and}\ U=u-1,
\]
where $\widetilde{D}(Z,U)$ is analytic around $(0,0)$ and satisfies $[Z^1]\widetilde{D}(Z,U)=0$, and $[Z^2]\widetilde{D}(Z,0)\neq 0$; and moreover we will check that $\rho$ is the radius of convergence of $z\to D(z,u)$ for $u\in\R$ around~$1$. 

To prove that, we first obtain a similar singular expansion for $T(x,v)$. An explicit expression is given in~\cite{BrownTriangulations}:
\[
T(x,v)=\frac{1}{2vx}\Big(-2v^2x^2+vx\eta(1+2(1-\eta))-(1-\eta)\eta^3-(\eta vx-(1-\eta)\eta^3)\sqrt{1-4vx/\eta^2}\Big),
\] 
where $\eta\equiv \eta(x)$ is the algebraic series given by $x=\eta^3(1-\eta)$, which has an expansion of the form $\eta(x)=\widetilde{a}(x)+\widetilde{b}(x)\cdot(1-x/\xc)^{3/2}$, where $\xc=27/256$, and $\widetilde{a}(x)$ and $\widetilde{b}(x)$ are analytic around $\xc$, and $\eta(\xc)=\widetilde{A}(\xc)=3/4$. Plugging this expansion into the expression of $T(x,v)$, one obtains an expansion of $T(x,v)$ around $(\xc,1)$ of the form
\[
T(x,v)=\widetilde{T}(X,V),\ \ \mathrm{with}\ X=\sqrt{1-x/\xc}\ \mathrm{and}\ V=v-1,
\]   
where $\widetilde{T}(X,V)$ is analytic around $(0,0)$ and satisfies $[X^1]\widetilde{T}(X,V)=0$, and $[X^2]\widetilde{T}(X,0)\neq 0$. 
Next, we have (from~\cite{BKLM,noy2020further}) that $L(z)$ and  $D(z)$ are of the form 
$f(z)=\widetilde{f}(Z)$ with $\widetilde{f}(Z)$ analytic around $0$ and satisfying $\widetilde{f}'(0)=0$ 
and $\widetilde{f}''(0)\neq 0$; and moreover $\rho D(\rho)^3=\xc$. Plugging these expansions into~\eqref{eq:G} (see~\cite[Sec.2.6]{drmota2014maximum} for a similar extraction) we can extract an  expansion for $D(z,u)$ of the form $D(z,u)=\widetilde{D}(Z,U)$, where $\widetilde{D}(Z,U)$ is analytic around $(0,0)$ and satisfies $[Z^1]\widetilde{D}(Z,U)=0$, and $[Z^2]\widetilde{D}(Z,0)\neq 0$.

Taking $Z=0$, this ensures that the univariate function $u\mapsto D(\rho,u)$ is analytic at~$1$. Letting $a_{n,m}:=[z^nu^m]D(z,u)$, we have $D(\rho,u)=\sum_{n,m}a_{n,m}\rho^nu^m$ for $u\leq 1$ (because $\sum_{n,m} a_{n,m} \rho^n=D(\rho)<\infty$). The analyticity of $u\mapsto D(\rho,u)$ at $u=1$ then ensures that the convergence of the above sum extends for $u\leq u_0$ for some $u_0>1$. Hence, there exists $u_0>1$ such that $D(\rho,u_0)$ converges.

The proof for the case of multigraphs is completely similar (the equation-system is the same, except for $P(z,u)=\frac{1}{2}u^2D(z,u)^2$). 
\end{proof}

\subsection{A first bound on the diameter}

\index{e@\textbf{maps and graphs}!diameter@\ttt{$\mathrm{Diam}(\mg)$}{diameter of $\mg$}}
For later purpose, we need to establish a first bound on the diameter of random connected cubic planar graphs (resp. multigraphs), 
which is the analogue of the one for random planar graphs obtained
in~\cite{chapuy2015diameter}. We say that a property $A_{n,\epsilon}$ related to a random object of size $n$, and also formulated in terms of a real parameter $\epsilon>0$, holds \emph{a.a.s. with exponential rate}
\index{aasexpo@convergence a.a.s. with exponential rate}
 if there are positive constants $a,b,c$ such that $P(\mathrm{not}\ A_{n,\epsilon})\leq a\,\exp(-b\,n^{c\epsilon})$ for all $n\geq 0$ and $\epsilon>0$ small enough.  The \emph{diameter} of a graph $\mg$, denoted $\mathrm{Diam}(\mg)$, is the maximum of the pairwise distances between vertices in $\mg$. 

The purpose of this section is to prove the following result:
\begin{proposition}\label{prop:diam}
Let $\cn$ (resp. $\mn$) be the random connected cubic planar graph (resp. multigraph) on $2n$ vertices. Then $\mathrm{Diam}(\cn)\leq n^{1/4+\epsilon}$ (resp. $\mathrm{Diam}(\mn)\leq n^{1/4+\epsilon}$) a.a.s. with exponential rate.
\end{proposition}

As a first step, we establish the bound in the 3-connected case (starting in the dual setting of random simple triangulations).

\index{t@\textbf{rooted simple planar triangulations}!tglrand@\ttt{$\tn$}{uniform on $\kT_n$}}
\begin{lemma}\label{lem:diam_Tn}
Let $\tn$ be the random rooted simple triangulation with $n+2$ vertices.  Then $\mathrm{Diam}(\tn)\leq n^{1/4+\epsilon}$ a.a.s. with exponential rate.
\end{lemma}
\begin{proof}
\index{t@\textbf{rooted simple planar triangulations}!tglpoint@\ttt{$\kT_n'$}{of $\kT_n$ with a marked vertex distinct from the root vertex}}
\index{i@\textbf{quasi-simple triangulations of the $p$-gon}!a@\ttt{$\kQ_{n,p}$}{with $n$ inner vertices}}
\index{i@\textbf{quasi-simple triangulations of the $p$-gon}!of the $1$-gon!gf@\ttt{$G$}{g.f. with respect to $\# V-2$}}
\index{i@\textbf{quasi-simple triangulations of the $p$-gon}!of the $1$-gon!gfi@\ttt{$G^{(i)}$}{g.f. with respect to $\# V-2$ when the root and marked vertex are at distance $i$}}
\index{i@\textbf{quasi-simple triangulations of the $p$-gon}!of the $1$-gon!rand@\ttt{$\qn^{(1)}$}{uniform in $\kQ_{n,1}$}}
\index{i@\textbf{quasi-simple triangulations of the $p$-gon}!of the $1$-gon!randL@\ttt{$L_n$}{distance between root vertex and marked vertex in $\mathfrak{q}^{(1)}_{n-1}$}}
Let $\kT_n'$ be the set of rooted triangulations with a marked vertex (distinct from the root vertex), having $n+2$ vertices. 
We use the fact (to be detailed in Section~\ref{sub:quasiSimpleDef}, and illustrated in Figure~\ref{fig:exRoot1gon}) that 
$\kT_{n-1}'$ injects into the set $\kQ_{n,1}$ of so-called quasi-simple triangulations of the 1-gon having $n$ inner vertices, in such a way that the 
underlying metric space is preserved; and that both $|\kQ_{n,1}|$ and $|\kT_n'|$ are 
$\Theta((256/27)^nn^{-3/2})$.   
Let $G(x)$ be the counting series of quasi-simple triangulations of the 1-gon with respect to the number of vertices minus $2$. For $i\geq 1$, let $\Gi(x)$ be the counting series of quasi-simple triangulations of the 1-gon where the root vertex and the marked vertex are at distance $i$, counted with respect to the number of vertices minus $2$. Let $\qn^{(1)}$ be a uniformly random map in $\kQ_{n,1}$. 
 Denoting by $L_n$ the distance between the root vertex and the marked vertex in $\mathfrak{q}^{(1)}_{n-1}$, we have
\[
\mathbb{P}(L_n=i)=\frac{[x^n]\Gi(x)}{[x^n]G(x)}.
\]
An explicit expression\footnote{The expression is given for symmetric simple triangular $k$-dissections of order $k\geq 3$ such that the central vertex is at distance $i$ from the outer boundary; the quotient of these dissections by the rotation of order $k$ is precisely a quasi-simple triangulation whose root-vertex
is at distance $i$ from the marked vertex.} for $\Gi(x)$ is given in~\cite[Prop.5.3]{albenque2014symmetric}.  
Precisely, there exists an (algebraic) counting series $Y(x)$, with a quartic dominant singularity at $27/256$, 
 \[
 Y(x)\sim 1-24^{1/4}(1-256\,x/27)^{1/4},
 \]
 such that $\Gi(x)$ is expressed as 
 \[
 \Gi(x)=Y^i(x)\cdot R(Y^i(x),Y(x),S(x)),
 \]
 where $S(x)=\sqrt{Y(x)^2+10Y(x)+1}$, and where $R(z_1,z_2,z_3)$ is an explicit rational expression that is continuous and positive at $(z_1=0, z_2=1,z_3=\sqrt{12})$. Let $x_n=27/256\cdot(1-1/n)$. Then $Y(x_n)=1-(24/n)^{1/4}+o(n^{-1/4})$. Hence, for $\epsilon>0$ and $i\geq n^{1/4+\epsilon}$, we
 have $Y^i(x_n)=O(\exp(-b\,n^{\epsilon}))$, where $b$ is any positive constant smaller than $24^{1/4}$ (e.g. $b=24^{1/4}/2$). Given the expression of $\Gi(x)$, we also have
  $\Gi(x_n)=O(\exp(-b\,n^{\epsilon}))$. Hence, for $i\geq n^{1/4+\epsilon}$, we have
  \[
  [x^n]\Gi(x)\leq \Gi(x_n)x_n^{-n}=O\big(\exp(-b\,n^{\epsilon})(256/27)^n\big).
  \]
Since $[x^n]G(x)=|\kQ_{n,1}|=\Theta((256/27)^n n^{-3/2})$, we conclude that $\mathbb{P}(L_n=i)=O(n^{3/2}\exp(-b\,n^{\epsilon}))$
uniformly over $i\geq n^{1/4+\epsilon}$, and thus
\[
\mathbb{P}(L_n\geq n^{1/4+\epsilon})=O\big(n^{5/2}\exp(-b\,n^{\epsilon})\big)=O\big(\exp(-n^{\epsilon/2})\big),
\]
so that $L_n\leq n^{1/4+\epsilon}$ a.a.s. with exponential rate. 
Let $\tn$ be the uniform random rooted simple triangulation with $n+2$ vertices. 
Let $\tilde{L}_n$ be the distance between the root vertex and a random vertex (distinct from the root vertex) in $\tn$. 
Note that $\tilde{L}_n$ is distributed as $L_n$ conditioned on the event that the triangulation associated with $\mathfrak{q}_{n-1}$ is simple (via the 
mapping described in Section~\ref{sub:quasiSimpleDef}). Since that event has probability $\Theta(1)$, 
we conclude that $\tilde{L}_n\leq n^{1/4+\epsilon}$ a.a.s. with exponential rate.  Since the number of choices for a root and a marked vertex is $O(n^2)$, the union-bound ensures that $\mathrm{Diam}(\tn)\leq n^{1/4+\epsilon}$ a.a.s. with exponential rate.
\end{proof}

\index{g@\textbf{3-connected cubic planar graphs}!family@\ttt{$\kK_n$}{with $2n$ vertices}}
\index{g@\textbf{3-connected cubic planar graphs}!rand@\ttt{$\kn$}{uniform in $\kK_n$}}
\begin{lemma}\label{lem:diam_K}
Let $\kn$ be the random 3-connected cubic planar graph on $2n$ vertices. Then $\mathrm{Diam}(\kn)\leq n^{1/4+\epsilon}$ a.a.s. with exponential rate.
\end{lemma}
\begin{proof}
Up to choosing one of the two embeddings,  
and choosing a root edge, $\mathrm{Diam}(\kn)$ is distributed as the diameter of the random rooted 3-connected cubic planar map with $2n$ vertices.   
Let $\map$ be a planar map, and let $\map^\dag$ be its dual. 
 Let $v,v'$ be vertices in $\map$, let $f$ (resp. $f'$) be a face incident to $v$ (resp. to $v'$), and let $v_f$ (resp. $v_{f'}$) be the vertex of $\map^\dag$ corresponding
 to $f$ (resp. to $v'$). From any path $\gamma=(v_0=v,\ldots,v_r=v')$ connecting $v$ to $v'$, 
 it is easy to derive a path $\gamma^\dag$ on $\map^\dag$ between $v_f$ and $v_{f'}$ of length at most $\mathrm{deg}(v_0)+\cdots+\mathrm{deg}(v_r)$.   
\index{e@\textbf{maps and graphs}!degmax@\ttt{$\maxdeg(\map)$}{maximal vertex-degree of $\map$}}
\index{e@\textbf{maps and graphs}!degroot@\ttt{$\mathrm{rootDeg}(\map)$}{degree of the root vertex of $\map$}}
 Hence, if $\maxdeg(\map)$ denotes the maximal vertex-degree in $\map$, then we have 
\[
 \mathrm{Diam}(\map^\dag)\leq  (1+\mathrm{Diam}(\map))\cdot \maxdeg(\map).
 \]
 Hence, $\mathrm{Diam}(\kn)$ is stochastically dominated by  $(1+\mathrm{Diam}(\tn))\cdot \maxdeg(\tn)$. Given Lemma~\ref{lem:diam_Tn}, it remains
 to prove that $\maxdeg(\tn)\leq n^{\epsilon}$ a.a.s. with exponential rate. Letting $\mathrm{rootDeg}(\tn)$ 
 be the root degree in $\tn$, we have (with $T(x,v)$ defined in Section~\ref{sec:pole_dista})
 \[
 \mathbb{P}(\mathrm{rootDeg}(\tn)=k)=\frac{[x^nv^k]T(x,v)}{[x^n]T(x,1)}. 
 \]
 
 The explicit expression~\cite{BrownTriangulations} of $T(x,v)$ (recalled in Section~\ref{sec:pole_dista}) guarantees, that there exists $v_0>1$ such that 
 $C:=T(27/256,v_0)$ converges as a sum (any $v_0\in(1,4/3)$ fits).  We thus have $[x^nv^k]T(x,v)\leq C\, (256/27)^nv_0^{-k}$, so that
 \[
 \mathbb{P}(\mathrm{rootDeg}(\tn)=k)=O(n^{5/2}v_0^{-k}), 
 \] 
On the other hand, a standard re-rooting argument ensures that
 \[
 \mathbb{P}(\maxdeg(\tn)=k)\leq 6n\cdot \mathbb{P}(\mathrm{rootDeg}(\tn)=k)=O(n^{7/2}v_0^{-k}), 
 \] 
 which is $O(v_0^{-n^{\epsilon}/2})$ for $k\geq n^{\epsilon}$. 
Hence, $\maxdeg(\tn)\leq n^{\epsilon}$ a.a.s. with exponential rate. 
 \end{proof}

Let $\mg$ be a (non-rooted) cubic connected planar multigraph, and let $\tau(\mg)$ be the associated decomposition-tree. We call \emph{\attr network within $\mg$} one of the cubic networks assembled around an \cycr of components (node of label \labr in $\tau(\mg)$). Note that such a component must be a non-trivial non-\netr. It corresponds to an edge $e$ with a \labr extremity in $\tau(\mg)$, precisely it is the network associated to the subtree hanging from the non-\labr side of $e$.  
We let $\deltr(\mg)$ be the maximal value of $\delta(N)$ over all \attr networks within $\mg$.
 And we let $\xir(\mg)$ be the maximal degree over all \noders of $\tau(\mg)$. 
Similarly, we define a \emph{\attt network} as a network substituted at one of the edges of a \compt of $\mg$.  It corresponds to an edge $e$ with a \labt extremity in $\tau(\mg)$, precisely it is the network associated to the subtree hanging from the other side of $e$. We let $\deltt(\mg)$ be the maximal value of $\delta(N)$ over all \attt networks within $\mg$. And we let $\xit(\mg)$ be the maximal diameter over all 3-connected components of $\mg$. 

Similarly as in~\cite[Eq.(6)]{chapuy2015diameter}, the following bound on the diameter of $\mg$ holds: 
\begin{claim}\label{claim:bound_diam}
We have \[\mathrm{Diam}(\mg)\leq (1+\mathrm{Diam}(\tau(\mg)))\cdot(2+\xir(\mg)\deltr(\mg)+\xit(\mg)\deltt(\mg)).\]
\end{claim}
\begin{proof}
A diametral path $\gamma$ in $\mg$ induces a path $\gamma'$ in $\tau(\mg)$. Some edges on $\gamma$ correspond to edges of $\gamma'$ (where $\gamma$
leaves a component and enters another component).  The other edges of $\gamma$ are ``consumed'' within the nodes on $\gamma'$. 
 For a \nodet $w$ visited by $\gamma$, corresponding to a $3$-connected component $\mathrm{K}$, the length of $\gamma$ consumed within $w$ is bounded by $\mathrm{Diam}(\mathrm{K})\cdot\mathrm{max}_{N\in\mathrm{att}_w}\delta(N)$, where $\mathrm{att}_w$ is the set of \attt networks at $\mathrm{K}$. 
Similarly, for an \noder $w$, of degree $k$, visited by $\gamma$, the length of $\gamma$ consumed within $w$ is bounded by $k\cdot\mathrm{max}_{N\in\mathrm{att}_w}\delta(N)$, where $\mathrm{att}_w$ is the set of \attr networks at $w$. 
Finally, for an \nodem~$w$, at most one edge of $\gamma$ is consumed by $w$ 
(this happens if one of the three edges of the \compm for $w$ is not substituted, and this edge is  
traversed by $\gamma$).  
\end{proof}

In view of Claim~\ref{claim:bound_diam}, in order to conclude the proof of Proposition~\ref{prop:diam}, it just remains to show the following.

\begin{lemma}
For $\cn$ the random connected cubic planar graph with $2n$ vertices, we have 
\[\mathrm{Diam}(\tau(\cn))\leq n^\epsilon,\ \  \deltr(\cn)\leq n^\epsilon,\ \  \deltt(\cn)\leq n^{\epsilon},\ \ \xir(\cn)\leq n^{\epsilon},\ \ \xit(\cn)\leq n^{1/4+\epsilon}\] 
a.a.s. with exponential rate. 

The same holds when replacing $\cn$ by $\mn$. 
\end{lemma}
\begin{proof}
We give the proof details for graphs (the arguments for multigraphs are completely similar). 
We start with $\xit(\cn)$. Consider a 3-connected component $\mathrm{K}$ in $\cn$, and let $n'\leq n$ be its size. Then $\mathrm{K}$ is uniformly distributed
over cubic 3-connected planar graphs with $2n'$ vertices. If $n'\leq n^{1/4}/2$, then obviously 
$\mathrm{Diam}(\mathrm{K})\leq n^{1/4+\epsilon}$. If $n'\in[n^{1/4}/2,n]$, then by Lemma~\ref{lem:diam_K} 
there are positive constants $a,b,c$ such that, for $\epsilon>0$ small enough, we have (where Lemma~\ref{lem:diam_K} yields the middle inequality): 
\[
\mathbb{P}(\mathrm{Diam}(\mathrm{K})\geq n^{1/4+\epsilon})\leq \mathbb{P}(\mathrm{Diam}(\mathrm{K})\geq n'\,^{1/4+\epsilon})\leq a\,\exp\big(-b\,n'\,^{c\epsilon})\leq a\,\exp\big(-b\,n^{c\epsilon/4}/2^{c\epsilon}\big).
\]

Hence, uniformly over $n'\in[0,n]$, the diameter of a random $3$-connected cubic planar graph of size $n'$ is smaller than $n^{1/4+\epsilon}$ a.a.s. with exponential rate. Since the number of $3$-connected components in $\cn$ is $O(n)$, by the union-bound we conclude that $\xit(\cn)\leq n^{1/4+\epsilon}$ a.a.s. with exponential rate. 

We now deal with $\xir(\cn)$. Let $\bar{S}(z):=L(z)+P(z)+H(z)$ be the generating function for cubic networks that are eligible to be \attr (i.e., non-trivial non-\labr cubic networks). 
From the 1st and 3rd line in~\eqref{eq:syst_D} (or directly, from the decomposition of \netrs), we get $S(z)=\frac{\bar{S}^2}{1-\bar{S}(z))}$.   
Since $S(z)$ converges at $\rho$, we must have $\bar{S}(\rho)\in(0,1)$. For $k\geq 3$, the generating function of connected cubic planar graphs with a marked \compr of degree $k$ is equal to $\bar{S}(z)^k/k$. Hence, for $k\geq 3$, 
the number of connected cubic planar graphs of size $n$ having a marked \compr of degree $k$ is bounded by $(2n)![z^n]\bar{S}(z)^k$, which itself is bounded
by $(2 n)!\beta^k\rho^{-n}$, where $\beta=\bar{S}(\rho)\in(0,1)$. The number of connected cubic planar graphs of size $n$ having a marked \compr of degree larger than $k$ is thus 
bounded by $(2n)!\frac{\beta^k}{1-\beta}\rho^{-n}$. Since the number of \compr is $O(n)$, the number of connected cubic planar graphs of size $n$ having at least one \compr of degree larger than $k$ is  $O((2n)!n\frac{\beta^k}{1-\beta}\rho^{-n})$
On the other hand, the number of connected cubic planar graphs of size $n$ is $\Theta((2n)!\rho^{-n}n^{-7/2})$. Hence, the probability that $\cn$ 
has an \compr of degree larger than $k$ is $O(n^{9/2}\beta^k)$, which is $O(\beta^{n^{\epsilon}/2})$ for $k\geq n^{\epsilon}$, 
 ensuring that $\xir(\cn)\leq n^{\epsilon}$ a.a.s. with exponential rate.

We now consider the parameter $\deltt(\cn)$. Again, we use the fact that a \attt network of size $n'$ in $\cn$ is uniformly distributed over cubic networks of size $n'$. 
This distribution is also $\mu^{(\kD)}$ conditioned to have size $n'$, an event that occurs with probability $\Theta(n'\,^{-5/2})$ (as follows from the asymptotic form of the counting coefficients for cubic networks). Hence, letting $\mathrm{D}_{n'}$ be the random cubic network of size $n'$, Lemma~\ref{lem:nustar} ensures that there are positive constants $a_\star,b_\star$ such that 
\[
\mathbb{P}(\delta(\mathrm{D}_n)>k)\leq a\,n^{5/2}\exp(-b\,k). 
\]
Hence, $\delta(\mathrm{D}_{n'})\leq n^{\epsilon}$ a.a.s. with exponential rate, uniformly over $n'\leq n$. Since the number of \attt networks in $\cn$ is 
$O(n)$, the union bound ensures that  $\deltt(\cn)\leq n^{\epsilon}$ a.a.s. with exponential rate. The same argument is easily adapted to $\deltr(\cn)$ as well
(an \attr network of size $n'$ is a network drawn under $\mu^{(\kD)}$ conditioned to have size $n'$ and to be a non-trival non-\netr, an event 
that holds with probability $\Theta(n'\,^{-5/2})$).  

We now consider $\mathrm{Diam}(\tau(\cn))$. We first rewrite the system~\eqref{eq:syst_D} in a positive form, where $t_n$ is the number of rooted simple 
triangulations with $2n$ faces. 
\[
\left\{
\begin{array}{ll}
L(z)&=\frac{1}{2}L(z)^2 + \frac1{2}z(S(z)+P(z)+H(z)),\\
S(z)&= \frac{(L(z)+P(z)+H(z))^2}{1-(L(z)+P(z)+H(z))} \\
P(z)&=\frac1{2}z\big( (L(z)+S(z)+P(z)+H(z))\cdot(2+L(z)+S(z)+P(z)+H(z)) \big)  \\
H(z)&= \frac1{2}\sum_{n\geq 2}t_n z^n\big(1+L(z)+S(z)+P(z)+H(z)\big)^{3n-1}.
\end{array}
\right.
\]
This system, of the form $\mathbf{y}=\mathbf{F}(z,\mathbf{y})$, is irreducible.  It is also critical in the sense of~\cite[Sec.5]{chapuy2015diameter}, as the quantities $L'(z),S'(z),P'(z),H'(z)$ converge when $z\in(0,\rho)$ tend to $\rho$. For $h\geq 0$, let $L_h(z),S_h(z),P_h(z),H_h(z)$ be the counting series gathering the contributions of $L(z),S(z),P(z),H(z)$ where the associated decomposition-tree has height at most $h$, and let $\mathbf{y}_h=(L_h(z),S_h(z),P_h(z),H_h(z))$. We clearly have $\mathbf{y}_{h+1}=\mathbf{F}(z,\mathbf{y}_h)$ for $h\geq 0$. Thus, the height of the decomposition-tree is a height-parameter for the system, with the terminology of~\cite[Sec.5]{chapuy2015diameter}. We can then rely on~\cite[Lem.5.3.]{chapuy2015diameter}, which ensures that for a random cubic network in each type $\in\{L,S,P,H\}$,  
 the height of the decomposition-tree is bounded by $n^{\epsilon}$ a.a.s. with exponential rate. Since the diameter of a rooted tree is at most twice its height, 
 we easily conclude (via the decomposition~\eqref{eq:Cp} for pointed cubic connected planar graphs in terms of cubic networks)
 that $\mathrm{Diam}(\tau(\cn))\leq n^{\epsilon}$ a.a.s. with exponential rate.
\end{proof}

\section{Skeleton decomposition for quasi-simple triangulations}\label{sec:quasi_simple}
We introduce in this section a modification of simple triangulations, which will prove to be more convenient to deal with, when studying Krikun's skeleton decomposition. We will indeed prove in Sections~\ref{sub:skeleton} and~\ref{sub:branching} that the skeleton of these so-called \emph{quasi-simple triangulations} admit exactly the same encoding by a branching process as general triangulations. 

\subsection{Simple triangulations}

\index{u@\textbf{rooted simple planar triangulations of the $p$-gon}!a@\ttt{$\kT_{n,p}$}{with $n$ inner vertices}}
For every $n\geq 1$, we write $\kT_n$ be the set of rooted simple triangulations with $n+2$ vertices. Then, for every $p\geq 3$ and $n\geq 0$, let $\kT_{n,p}$ be the set of rooted simple triangulations of the $p$-gon with $n$ inner vertices (that is, simple triangulations whose root face is a simple cycle of degree $p$, with $n$ vertices that are not incident to the root face). Observe in particular that $\kT_n=\kT_{n-1,3}$.
Enumerative formulas for $\kT_{n,p}$ have been obtained by Brown in~\cite{BrownTriangulations} and read: 
\begin{equation}\label{eq:enumSimpleTrig}
|\kT_{n,p}|=\frac{2 (2p-3)!}{(p-1)!(p-3)!}\frac{(4n+2p-5)!}{n!(3n+2p-3)!}\quad\text{for}\quad n\geq 0,\,p\geq 3.
\end{equation}
We deduce from this expression and from Stirling's formula that:
\begin{equation}\label{eq:asymSimpleTrig}
|\kT_{n,p}|\stackrel[n\rightarrow \infty]{}{\sim} C^\star(p)\left(\frac{256}{27}\right)^{n}n^{-5/2},
\end{equation}
with 
\begin{equation}
C^\star(p)=\frac{\sqrt{3}}{64\sqrt{2\pi}}(p-2)\binom{2(p-1)}{p-1}\left(\frac{16}{9}\right)^{p-1}\stackrel[p\rightarrow \infty]{}{\sim} \frac{\sqrt{6}}{64\pi}\sqrt{p}\left(\frac{64}{9}\right)^{p-1} = \frac{9\sqrt{6}}{4096\pi}\sqrt{p}\left(\frac{64}{9}\right)^{p}.
\end{equation}
\index{u@\textbf{rooted simple planar triangulations of the $p$-gon}!gf@\ttt{$T_p$}{generating series, radius of convergence = $27/256$}}
\index{u@\textbf{rooted simple planar triangulations of the $p$-gon}!Z@$Z(p) = T_p\left(\frac{27}{256}\right)$}
Because of the polynomial correction in $n^{-5/2}$ (usual for planar maps), the generating series $T_p(x):=\sum_{n\geq 0} |\kT_{n,p}|x^n$ evaluated at $x=27/256$ converge for every value of $p$. Following \cite[(4.5)]{BrownTriangulations}, we get: 
\begin{equation}\label{eq:defZp}
Z(p):=T_p\left(\frac{27}{256}\right)= \frac{1}{p(2p-3)}\left(\frac{16}{9}\right)^{p-2}\binom{2(p-1)}{p-1}.
\end{equation}
From which we deduce: 
\begin{equation}\label{eq:asymZp}
Z(p)\stackrel[p\rightarrow \infty]{}{\sim} \frac{2}{\sqrt{\pi}}\left(\frac{64}{9}\right)^{p-2}p^{-5/2}=\frac{18}{64\sqrt{\pi}}\left(\frac{64}{9}\right)^{p-1}p^{-5/2}.
\end{equation}

The \emph{critical Boltzmann distribution} on simple triangulations of the $p$-gon is the one assigning probability $(27/256)^n/Z(p)$ to the elements in $\kT_{n,p}$,
for $n\geq 0$.

\subsection{Quasi-simple triangulations: definition and enumeration}\label{sub:quasiSimpleDef}

\begin{figure}[t]
  \begin{minipage}[b]{0.45\linewidth}
   \centering
     \includegraphics[scale=0.8,page=2]{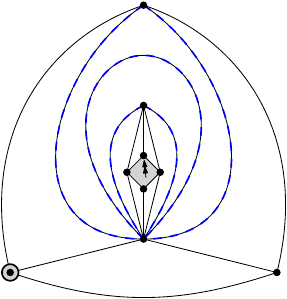}
     \caption{\label{fig:exQuasi}A quasi-simple triangulation of the 4-gon. Separating 1- and 2-cycles are represented in dashed blue edges.\\ The pointed vertex is encircled.}
  \end{minipage}
\hfill
  \begin{minipage}[b]{0.45\linewidth}
   \centering
     \includegraphics[scale=0.5,page=3]{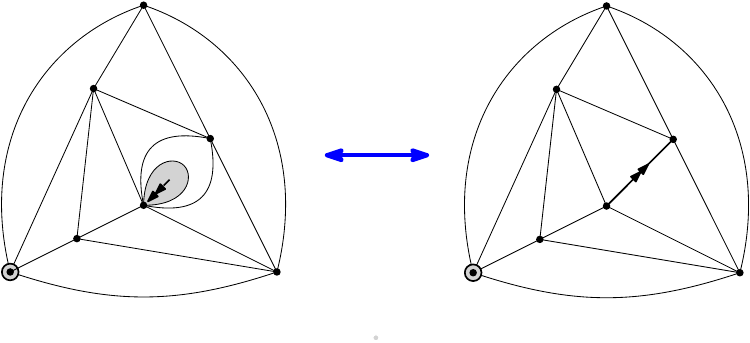}
     \caption{\label{fig:exRoot1gon}Mapping $\Psi$ that turns a quasi-simple triangulation of the 1-gon into a quasi-simple triangulation of the sphere.\\ The pointed vertex is encircled.}
  \end{minipage}
\end{figure}

\index{i@\textbf{quasi-simple triangulations of the $p$-gon}!a@\ttt{$\kQ_{n,p}$}{with $n$ inner vertices}}
\begin{definition}
For $p\geq 1$, let $\mathrm{q}$ be a triangulation of the $p$-gon with a marked inner vertex $v^\star$. We call $(\mathrm{q},v^\star)$ a \emph{quasi-simple triangulation} if all its 1-cycles and 2-cycles separate the $p$-gon and the pointed vertex, see Figure~\ref{fig:exQuasi}. 
\end{definition}

\begin{remark}
Observe that in a quasi-simple triangulation, each vertex can carry at most one loop, otherwise the two loops would form a 2-cycle that does not contain the pointed vertex. 

\index{i@\textbf{quasi-simple triangulations of the $p$-gon}!aa@\ttt{$\Psi$}{maps quasi-simple triangulations of the $1$-gon to (not necessarily simple) vertex-pointed triangulations of the sphere}}
Therefore, we can turn a quasi-simple triangulation of the 1-gon into a triangulation of the sphere by merging the two edges that form a triangular face with the root edge, rooting the modified map at the merged edge so that the tail vertex is the one incident to the deleted loop, see Figure~\ref{fig:exRoot1gon}. This mapping is denoted $\Psi$. Note 
that $\Psi$ preserves the underlying metric space.

Reciprocally, the set of pointed and rooted simple triangulations of the sphere can be seen as a (strict) subset of quasi-simple triangulations of the 1-gon by opening the root edge and inserting a loop incident to the root vertex inside the 2-gon created. \dotfill
\end{remark}

\medskip

For every $p\geq 1$ and $n\geq 0$, let $\kQ_{n,p}$ be the set of quasi-simple triangulations of the $p$-gon with $n$ inner vertices. Enumerative formulas for $\kQ_{n,p}$ have been obtained by Brown in~\cite{BrownTriangulations} and we list them here for future reference. Precisely (up to a change of variable), 
Brown gives in~\cite[(8.12)]{BrownTriangulations} 
a formula for the number of simple triangulations of a $pr$-gon with a rotational symmetry of order $r\geq 3$ and $rs+1$ inner vertices, which are in bijection with quasi-simple triangulations of the $p$-gon with $s+1$ inner 
vertices,  
up to marking the center of rotation and quotienting.   
For $p\geq 1$ and $n\geq 1$, 
this formula rewrites as:
\begin{equation}\label{eq:qsimpleTrigEnum}
 |\kQ_{n,p}| = \frac{(2p)!(4n+2p-5)!}{(p-1)!p!(n-1)!(3n+2p-3)!}.
\end{equation} 
\index{i@\textbf{quasi-simple triangulations of the $p$-gon}!enum@$C(p)$}
The asymptotic behavior of $|\kQ_{n,p}|$ as $n$ goes to infinity can be deduced directly from this expression. It is given in \cite[(9.1)]{BrownTriangulations} and reads: 
\begin{equation}\label{eq:asymptoticQEnum}
|\kQ_{n,p}|\stackrel[n\rightarrow \infty]{}{\sim} C(p)\left(\frac{256}{27}\right)^{n}n^{-3/2},
\end{equation}
with
\begin{equation}
C(p)=\frac{\sqrt{3}}{32\sqrt{2\pi}}p\binom{2p}{p}\left(\frac{16}{9}\right)^{p-1}\stackrel[p\rightarrow \infty]{}{\sim} \frac{\sqrt{3}}{8\sqrt{2}\pi}\sqrt{p}\left(\frac{64}{9}\right)^{p-1} = \frac{9\sqrt{3}}{512\sqrt{2}\pi}\sqrt{p}\left(\frac{64}{9}\right)^{p}.
\end{equation}
Note the unusual polynomial correction in $n^{-3/2}$, which comes from the fact that quasi-simple triangulations are inherently rooted \emph{and} pointed. We hence retrieve the classical behavior of (non-pointed) planar maps with a polynomial correction in $n^{-5/2}$.

Following the same approach as in the proof of~\cite[Lemma~1]{CurienJFLG}, we obtain from~\eqref{eq:qsimpleTrigEnum} the following bounds for $|\kQ_{n,p}|$: 
\begin{lemma}\label{lem:unifBounds}
There exists a constant $c>0$ (independent of $p$ and $n$) such that, for every $n\geq 1$ and $p\geq 1$, we have: 
\[
    |\kQ_{n,p}| \leq c \cdot C(p) n^{-3/2} \left(\frac{256}{27}\right)^n.
\]
\end{lemma}

\subsection{Decomposition of quasi-simple triangulations of the 1-gon into simple components}

\index{t@\textbf{rooted simple planar triangulations}!tglpointa@\ttt{$\cTc$}{with a marked inner vertex}}
\index{t@\textbf{rooted simple planar triangulations}!tglpointa@\ttt{$\cTt$}{with a marked inner edge}}
\index{t@\textbf{rooted simple planar triangulations}!tglpointa@\ttt{$\cTct$}{with a marked inner edge incident to the root vertex}}
\index{t@\textbf{rooted simple planar triangulations}!tglpointa@\ttt{$\cTtm$}{with a marked inner edge non incident to the root vertex}}
\index{i@\textbf{quasi-simple triangulations of the $p$-gon}!of the $1$-gon!a@$\cQ=\cup_n \kQ_{n,1}$}
Let $\cQ$ be the family of quasi-simple triangulations of the 1-gon, where the size-parameter is the number of inner vertices. 
Let \cTc, \cTt, \cTct, \cTtm be the
families of rooted simple triangulations with respectively a marked inner vertex, a marked inner edge,
and a marked inner edge incident (resp. not incident) to the root vertex, where the size-parameter is the number of vertices minus 2 (which is also half the number of faces).  
We use a decomposition of $\cQ$ described in~\cite[Sec.3.2]{albenque2014symmetric} (decomposition 
along the nested sequence of 1-cycles and 2-cycles) that yields $\cQ=\kX\cdot\kU$, where the families $\kU,\kV$ admit the following decomposition-grammar\footnote{We use extensively decomposition grammars in this section, and recall that they automatically give corresponding equations for the associated generating series, see~\cite[ch.1]{flajolet2009analytic}. Disjoint union (denoted by $+$), cartesian product (denoted by $\cdot$) and sequence operators (denoted by $\mathrm{Seq})$ for decompositions grammar correspond respectively to addition, multiplication and $x\mapsto \frac{1}{1-x}$ for generating series.}:
\[
\left\{
\begin{array}{rl}
\kU & = \cTc + \kX \cdot (1+\kU) + \cTct \cdot \kU + \cTtm \cdot \kV,\\
\kV & = \cTc + 2\cdot \kX \cdot (1+\kU) + \cTt\cdot\kV,
\end{array}
\right.
\]
with $\kX$ an atom that accounts for $1$ in the size. 

This system can easily be turned into a regular expression for $\kU=\cQ$ involving only the ``terminal families'' \cTc, \cTt, \cTct, \cTtm (as well as the atom-family $\kX$).  
Note that the second line gives
\[
\kV=\big( \cTc+2\cdot\kX+2\cdot\kX\cdot\kU \big)\cdot\Seq(\cTt).
\]
Injecting it into the first line, we obtain
\begin{equation}\label{eq:reg_expr_U}
\kU=\kF+\kG\cdot\kU,\ \ \mathrm{equivalently}\ \ \kU=\kF\cdot\Seq(\kG),
\end{equation}
where $\kF=\cTc+\kX+\cTtm\cdot(\cTc+2\cdot\kX)\cdot\Seq(\cTt)$ and $\kG=\kX+\cTct+2\cdot\kX\cdot\cTtm\cdot\Seq(\cTt)$. 
\index{i@\textbf{quasi-simple triangulations of the $p$-gon}!of the $1$-gon!rand@\ttt{$\qn^{(1)}$}{uniform in $\kQ_{n,1}$}}
\index{i@\textbf{quasi-simple triangulations of the $p$-gon}!of the $1$-gon!randa@\ttt{$T(\qn^{(1)})$}{the largest simple component of $\qn^{(1)}$}}
\begin{lemma}\label{lem:largestSimple}
Let $\fqn^{(1)}$ be a uniformly random element of size $n$ in $\cQ$. Let $T(\fqn^{(1)})$ be the largest simple component of $\fqn^{(1)}$ along the decomposition~\eqref{eq:reg_expr_U}, 
and let $X_n$ be its half number of faces. Then a.a.s.
$T(\fqn^{(1)})$ is in $\cTc,\cTt$, or $\cTtm$ (not $\cTct$) and follows the uniform distribution for that family, conditioned on its size. In addition, the random variable $n-X_n$ 
%(half number of faces not in $T(\fqn^{(1)})$) 
converges in law, hence is $O(1)$ in probability. 
\end{lemma}
\begin{proof}
We rely on classical results of analytic combinatorics regarding the size of the largest component in combinatorial decompositions~\cite{Go95}. 
For $\alpha\in\R\backslash\N$, a combinatorial class $\kC$ is said to be of singularity-type $\alpha$ if the radius of convergence $\rho$ of  the counting series 
$C(z)=\sum_n|\kC_n|z^n$  is positive, and $C(z)$ can be analytically continued to a domain $\{|z|\leq \rho+\eta,z-\rho\notin\R_+\}$ (for some $\eta>0$), 
where it admits a singular expansion around $\rho$ of 
the form 
\[
C(z)=P(z)+d\,(1-z/\rho)^{\alpha}\,(1+o(1)),
\] 
for some polynomial $P(z)$ of degree smaller than $\alpha$, and some non-zero constant~$d$. Transfer
theorems of analytic combinatorics~\cite[VI.3]{flajolet2009analytic} then guarantee that $|\kC_n|\sim \frac{d}{\Gamma(-\alpha)}\rho^{-n}n^{-\alpha-1}$. The following
dictionary can then be obtained regarding the size of the greatest component (we also include a sum-rule to track which families along the decomposition have asymptotically non-zero probability of having the largest component):

\medskip

\noindent{\bf Sum rule.} Let $\kC=\kA+\kB$, where $\kA,\kB$ have the same radius of convergence $\rho$, and have respective 
singularity types $\alpha,\alpha'\notin\N$, both positive. Let $\cn$ be the random structure in $\kC_n$. Then $\cn$ is almost surely in $\kA$ (resp. $\kB$) if $\alpha<\alpha'$ (resp. $\alpha'<\alpha$). 

\medskip

\noindent{\bf Product rule.} Let $\kC=\kA\cdot\kB$, where $\kA,\kB$ have the same radius of convergence~$\rho$, and have respective 
singularity types $\alpha,\alpha'\notin\N$, both positive. Let $\cn$ be the random structure in $\kC_n$. Let $X_n$ be the size of its largest component (either in $\kA$ or in $\kB$, in $\kA$
if there is a tie).  
Then $n-X_n$ converges in law. In addition, the largest component is a.a.s. in $\kA$ (resp. in $\kB$) if $\alpha<\alpha'$ (resp. $\alpha'<\alpha$). 

\medskip

\noindent{\bf Sequence rule.} Let $\kC=\Seq(\kA)$ (with no object of size $0$ in $\kA$), where $\kC,\kA$ have same radius of convergence and same positive singularity type $\alpha\notin\N$  (non-critical case in~\cite{Go95}). Let $\cn$ be the random structure in $\kC_n$, and let $X_n$ be the size of the largest component (in $\kA$, if there is a tie the leftmost component of largest size is chosen). Then $n-X_n$ converges in law. %(the limit law having  exponential tail). 
% \ceri{enlevé "the limit law having  exponential tail", pas utile et sans doute faux}

\medskip

We can then use these rules repeatedly in the regular expression~\eqref{eq:reg_expr_U} to prove the statement, using also the fact~\cite{albenque2014symmetric} that the counting series of \cTc, \cTt, \cTct, \cTtm all have radius of convergence $27/256$, with singularity type $1/2$ for \cTc, \cTt, \cTtm, and singularity type $3/2$ for \cTct (the fact that the largest component 
is a.a.s. not in $\cTct$ is due to the larger singularity type for this family).  
\end{proof}

We will also need the following statement, ensuring that, for the decomposition~\eqref{eq:reg_expr_U}, the largest simple component of $\qn^{(1)}$ is asymptotically distributed as a uniform random rooted simple triangulation (conditioned on its size):

\begin{lemma}\label{lem:almostUniform}
Recall the definition of $\kT_n$. Then, the distribution induced (upon unmarking) by the 
uniform distribution on $\cTc_n$ (resp. $\cTt_n$) is the uniform distribution on $\kT_n$. And the distribution induced (upon unmarking) by the 
uniform distribution on $\cTtm_n$ is at total-variation distance $o(1)$ from the uniform distribution on $\kT_n$.
\end{lemma}
\begin{proof}
The statement for $\cTc_n$ (resp. $\cTt_n$) is obvious, since the number of allowed markings on an object in $\kT_n$ is $n-1$ (resp. $3n-3$). For $\cTtm_n$, the number
of allowed markings of an object in $\kT_n$, with root degree $d$, is $3n-1-d$.  
As we have seen in the proof of Lemma~\ref{lem:diam_K}, the root degree of the random triangulation in $\kT_n$ is a.a.s. at most $n^{\epsilon}$ for any $\epsilon>0$
(even more is true, it converges in law),  hence 
 the number of allowed markings is concentrated around $3n$, ensuring that the distribution induced by $\cTtm_n$ on $\kT_n$ is at total-variation distance $o(1)$ from 
the uniform distribution on $\kT_n$.  
\end{proof}

\subsection{Quasi-simple triangulations of the cylinder and their skeleton decomposition}\label{sub:skeleton}
In this section, we describe the structure of the neighborhood of the root in a quasi-simple triangulation, and how it can be encoded via the so-called \emph{skeleton decomposition} introduced by Krikun in~\cite{Krikun}. We follow here the presentation given by Curien and Le Gall in~\cite[Section~2.2]{CurienJFLG} for \emph{general} triangulations, and we  emphasize the necessary adjustments to deal with quasi-simple triangulations. 
We first introduce quasi-simple triangulations of the cylinder, which will appear naturally as the hulls of quasi-simple triangulations. 
\index{j@\textbf{quasi-simple triangulations of the cylinder}!cycle-bottom@\ttt{$\partial \Delta$}{bottom cycle, the boundary of the root face}}
\index{j@\textbf{quasi-simple triangulations of the cylinder}!cycle-top@\ttt{$\partial^* \Delta$}{top cycle, the boundary of the top face}}
\begin{definition}
Fix $r\in \Z_{>0}$. A \emph{quasi-simple triangulation} of the cylinder of height $r$ is a rooted planar map such that: 
\begin{enumerate}
\item All its faces are triangles except for two distinguished faces: its root face (also called the bottom face) and the top face.
\item The boundaries of the bottom and of the top face are disjoint simple cycles. 
\item Every vertex incident to the top face is at graph distance exactly $r$ from the boundary of the bottom face, and every edge incident to the top face is also incident to a triangle whose third vertex is at distance $r-1$ from the bottom face. 
\item Cycles of length 1 and 2 necessarily separate the bottom face and the top face. 
\end{enumerate}
If $\Delta$ is a quasi-simple triangulation of the cylinder of height $r$, the boundary of its root face (bottom cycle) is denoted by $\partial \Delta$ and its top cycle (boundary of the top face) is denoted by $\partial^* \Delta$. We call $\Delta$ a quasi-simple triangulation of the $(|\partial \Delta|,|\partial^* \Delta|)$-cylinder. 

\index{j@\textbf{quasi-simple triangulations of the cylinder}!a@\ttt{$\cyl p r$}{of height $r$ with root face of degree $p$}}
Moreover, we denote by $\cyl p r$ the set of quasi-simple triangulations of the cylinder of height $r$ with a root face of degree $p$. 
\end{definition}
Note that this is the same definition as~\cite[Definition~1]{CurienJFLG}, except for item $4$.
To describe the encoding of a triangulation of the cylinder via the skeleton decomposition, we need additional definitions. 
\index{j@\textbf{quasi-simple triangulations of the cylinder}!dball@\ttt{$B_j(\Delta)$}{ball of radius $j$}}
\begin{definition}
Let $\Delta$ be a fixed quasi-simple triangulation of the cylinder of height $r$. For $1\leq j\leq r$, the \emph{ball of radius $j$} of $\Delta$ -- denoted $B_j(\Delta)$ -- is the submap of $\Delta$, which is spanned by its faces which are incident to at least one vertex at distance strictly smaller  than $j$ from $\partial \Delta$. 

\begin{figure}
\centering
\includegraphics[width=0.4\linewidth]{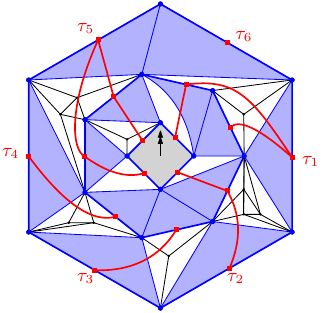}\caption{\label{fig:skeleton}Illustration of the skeleton decomposition of a quasi-simple triangulation of the $(4,6)$-cylinder of height 2.
Downward triangles are represented by blue shaded triangles. Trees of the forest are represented by fat red edges and are labeled from $\tau_1$ to $\tau_6$.}
\end{figure}

Moreover, the \emph{hull of radius $j$} of $\Delta$ -- denoted $B^\bullet_j(\Delta)$ -- is defined as the reunion of $B_j(\Delta)$ with the connected components of $\Delta\backslash B_j(\Delta)$ that do not contain $\partial^* \Delta$.
\end{definition}

\index{j@\textbf{quasi-simple triangulations of the cylinder}!dhull@\ttt{$B^\bullet_j(\Delta)$}{hull of radius $j$}}
\index{j@\textbf{quasi-simple triangulations of the cylinder}!dhullbdy@$\partial_j(\Delta) = \partial^*(B^\bullet_j(\Delta))$, $\partial_0(\Delta) = \partial \Delta$}
\index{j@\textbf{quasi-simple triangulations of the cylinder}!dhulldownwardtgl@$F_j(\Delta)$ downwards triangles incident to edges of $\partial_j(\Delta)$}
\index{j@\textbf{quasi-simple triangulations of the cylinder}!dhulledges@$\bar E(\Delta) = \cup_j E(\partial_j(\Delta))$}
Note that for any $j\in\{1,\ldots,r\}$, $B^\bullet_j(\Delta)$ is itself a quasi-triangulation of the cylinder (of height $j$), and we set $\partial_j(\Delta):=\partial^* (B^\bullet_j(\Delta))$. We extend this notation to $j=0$ and set $\partial_0(\Delta):=\partial \Delta$. Every edge of $\partial_j(\Delta)$ is incident to exactly one triangle whose third vertex belongs to $\partial_{j-1}(\Delta)$. Such triangles are called the \emph{downward triangles at height $j$}, and we write $\mathrm{F}_j(\Delta)$ for the set of all downward triangles at height $j$. Let $\bar E(\Delta)$ be the collection of all edges that belong to one of the cycles $\partial_j(\Delta)$, for $0\leq j \leq r$. The planar embedding of $\Delta$ enables to encode the edges of $\bar E(\Delta)$ as the vertices of a forest $\mathrm{f}$, with $q$ trees where $q:=|\partial^\star|$. This encoding is depicted in~Figure~\ref{fig:skeleton} and we refer to~\cite[page 9]{CurienJFLG} for its formal definition. 

Given the forest $\mathrm{f}:=(\tau_1,\ldots,\tau_q)$, to reconstruct the full quasi-triangulation $\Delta$, we need to fill the ``slots'' that lie between successive downward triangles. More precisely, for $e\in \partial_j(\Delta)$ with $1\leq j \leq r$, we associate to $e$ a (possibly empty) slot in the following way: the slot is included in $B^\bullet_j(\Delta)\backslash B^\bullet_{j-1}(\Delta)$, and is bounded by the edges of $\partial_{j-1}(\Delta)$ that correspond to children of $e$ and by the two ``vertical'' edges that connect the initial vertex of $e$ (when $\partial_j(\Delta)$ is oriented so that $B^\bullet_j(\Delta)$ lies on its left) with vertices of $\partial_{j-1}(\Delta)$. The definition should be clear from Figure~\ref{fig:skeleton}.

Writing $k_{\mathrm{f}}(e)$ for the number of children of $e$ in the forest $\mathrm{f}$, the slot associated to $e$ has to be filled with a \emph{simple} triangulation of the $k_{\mathrm{f}}(e)+2$-gon. Note that, when $k_{\mathrm{f}}(e)=0$, the slot is filled by the edge-triangulation, or in other words the two vertical edges defining the slot are identified (whereas in the setting of~\cite{CurienJFLG} of general triangulations, non-trivial slots with a boundary of length 2 may occur). 
\bigskip

\index{j@\textbf{quasi-simple triangulations of the cylinder}!z@$(p,q,r)$-admissible forest $\mathrm{f}$}
\index{j@\textbf{quasi-simple triangulations of the cylinder}!z@$(p,q,r)$-admissible forest $\mathrm{f}$!degree@\ttt{$k_\mathrm{f}(e)$}{number of children of $e$ in $\mathrm{f}$}}
\index{j@\textbf{quasi-simple triangulations of the cylinder}!z@$(p,q,r)$-admissible forest $\mathrm{f}$!inner@\ttt{$\mathrm{f}^*$}{vertices of $\mathrm{f}$ at height $<r$}}
To characterize the forests that can appear via the encoding, we need the following definition:
\begin{definition}
A forest $\mathrm{f}$ with a marked vertex is said to be $(p,q,r)$-admissible if: 
\begin{enumerate}
 \item The forest $\mathrm{f}:=(\tau_1,\ldots,\tau_q)$ is made of $q$ rooted plane trees, 
 \item The maximal height of these trees is $r$.
\item The total number of vertices of the forest at generation $r$ is $p$, 
\item The distinguished vertex has height $r$, 
\item The distinguished vertex is in $\tau_1$.
\end{enumerate}
We write $\mathrm{f}^\star$ for the set of vertices of $\mathrm f$ with height strictly smaller than $r$. 

\index{j@\textbf{quasi-simple triangulations of the cylinder}!z@$(p,q,r)$-admissible forest $\mathrm{f}$!a@\ttt{$\kF_{p,q,r}$}{set of all}}
\index{j@\textbf{quasi-simple triangulations of the cylinder}!z@$(p,q,r)$-admissible forest $\mathrm{f}$!aa@$\kF_{p,r} = \cup_{q\geq 1} \kF_{p,q,r}$}
Moreover, the set of $(p,q,r)$-admissible forests is denoted $\kF_{p,q,r}$ and we define $\kF_{p,r}:=\cup_{q\geq 1}\kF_{p,q,r}$.
\end{definition}
\begin{figure}[t!]
\centering
\subfloat[A separating loop.]{\quad\quad
      \includegraphics[width=5.4cm,page=2]{images/obstruction.pdf}
      \label{subfig:exLoop}\quad\quad} \qquad
    \subfloat[A separating pair of edges.]{\quad\quad
      \includegraphics[width=5.4cm,page=1]{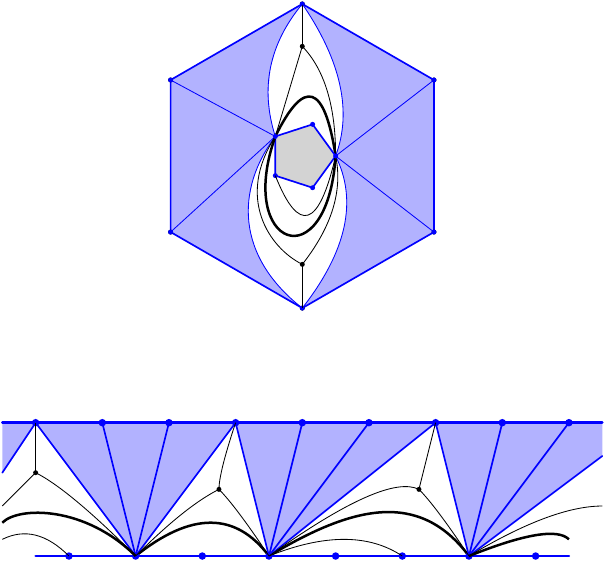}
      \label{subfig:exPair}\quad}\\
    \caption{(Top) Two examples of how a loop or a pair of multiple edges can appear when grafting simple triangulations in the slots of a skeleton. Both the loop and the multiple edges separate the root face (in grey) from the pointed vertex (not represented in the figures). (Bottom) The same maps, represented on the universal cover, are now simple triangulations.}
    \label{fig:loopPair}\end{figure}
\begin{proposition}
The skeleton decomposition described above is a bijection between, quasi-simple triangulations $\Delta$ of the $(p,q)$-cylinder of height $r$, and pairs $(\mathrm{f},(\mathrm{t}_v)_{v\in \mathrm{f}^\star})$, where $\mathrm{f}$ is a $(p,q,r)$-admissible forest and for any $v\in \mathrm{f}^\star$, with $e\in\bar{E}(\Delta)$ the corresponding edge, $\mathrm{t}_v$ is a \emph{simple} triangulation of the $(k_{\mathrm{f}}(e)+2)$-gon.

\index{j@\textbf{quasi-simple triangulations of the cylinder}!skeldec@\ttt{$(\mathfrak F(\Delta),(\mathrm{t}_v)_{v\in \mathfrak{F}(\Delta)^\star})$}{skeleton decomposition of $\Delta$}}
We write $(\mathfrak F(\Delta),(\mathrm{t}_v)_{v\in \mathfrak{F}(\Delta)^\star})$ for the image of $\Delta$ via this bijection. 
\end{proposition}

\begin{proof}
This proposition is very similar to the skeleton decomposition introduced originally by Krikun in~\cite{Krikun}. However, in this paper, he dealt with \emph{loopless} triangulations rather than quasi-simple triangulations. In other words, loops were always forbidden (whereas they are allowed in our setting if they separate $\partial \Delta$ and $\partial \Delta^\star$) and multiple edges were always authorized (whereas we allow them only if they separate $\partial \Delta$ and $\partial \Delta^\star$). This is the only differences between his decomposition and ours. Hence, we only need to prove that we indeed obtain a quasi-simple triangulation. In other words, we only need to check that 1-cycle and 2-cycle separate $\partial \Delta$ and $\partial \Delta^\star$. 

% This can be done by direct inspection, by considering the cases where edges of the 1-cycle or 2-cycle are edges of $\bar E(\Delta)$, ``vertical'' edges or edges of the simple triangulations that are grafted into the slots

One can argue via the universal cover of the cylinder, which is the periodic plane $(\R/\Z)\times\R$. Indeed, the constraint
of being quasi-simple on the cylinder is equivalent to being simple in the universal cover representation. 
Thus, one just has to check that, on $(\R/\Z)\times\R$, the obtained (periodic) triangulation is simple. Since, in the universal cover representation,
each boundary $\partial_j(\Delta)$ can conveniently be drawn as the horizontal line $\{y=j\}$, with the next added layer in the horizontal band $\{j\leq y\leq j+1\}$, the result follows directly. 

For sake of illustration, some configurations in which a loop or a pair of multiple edges can appear are represented on Figure~\ref{fig:loopPair} (note that it is only possible to have a loop or a pair of edges if $\mathfrak{F}(\Delta)$ has only one or two vertices at some generation $j$ with $0<j<r$). 
\end{proof}

\subsection{Hulls and encoding via a branching process}\label{sub:branching}
In this section, we show how the probability to have a fixed neighborhood of the root can be expressed in terms of a branching process via the skeleton decomposition. 

\begin{figure}
\centering
\subfloat[Balls of radius 1 and 2.]{\quad\quad
      \includegraphics[width=5.6cm,page=2]{images/HullBall.pdf}
      \label{subfig:exBall}\quad\quad} 
    \subfloat[Hulls of radius 1 and 2.]{\quad\quad
      \includegraphics[width=5.6cm,page=1]{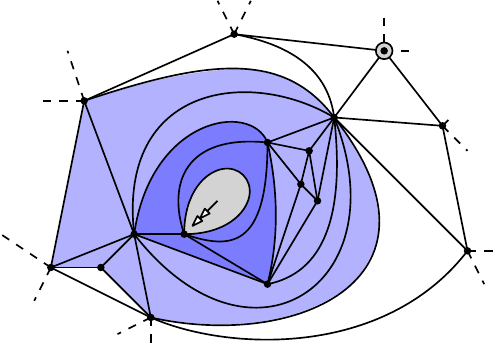}
      \label{subfig:exHull}\quad\quad}\\
    \caption{The balls and the hulls in a quasi simple triangulation of the 1-gon are represented in various shades of blue. The pointed vertex is the encircled vertex.}
    \label{fig:HullsBalls}
\end{figure}
\index{i@\textbf{quasi-simple triangulations of the $p$-gon}!hull-ball@\ttt{$B_R(\mathrm{q})$}{ball of radius $R$ of $\mathrm{q}$}}
\index{i@\textbf{quasi-simple triangulations of the $p$-gon}!hull-hull@\ttt{$B^\bullet_R(\mathrm{q},v^*)$}{hull of radius $R$ of $(\mathrm{q},v^*)$}}
\begin{definition}
Let $(\mathrm{q},v^*)$ be a quasi-simple triangulation of the $p$-gon. For $R\geq 1$, the \emph{ball of radius $R$} of $\mathrm{q}$ -- denoted $B_R(\mathrm{q})$ -- is the submap of $\mathrm{q}$, which is spanned by its faces which are incident to at least one vertex at distance strictly smaller than $R$ from its root face, see Figure~\ref{fig:HullsBalls}.

Moreover, the \emph{hull of radius $R$} of $(\mathrm{q},v^*)$ -- denoted $B^\bullet_R(\mathrm{q},v^\star)$ -- is defined from $B_R(\mathrm{q})$ in the following way. If $\dist(v_\rho,v^\star)\leq R$, then $B^\bullet_R(\mathrm{q},v^\star) = \mathrm{q}$. Otherwise, $B^\bullet_R(\mathrm{q},v^\star)$ is the submap of $\mathrm{q}$, that consists of the reunion of $B_R(\mathrm{q})$ with the connected components of $\mathrm{q}\backslash B_R(\mathrm{q})$ that do not contain $v^\star$.
\end{definition}
It follows directly from the definition of the hull, that $B^\bullet_R(\mathrm{q},v^\star)$ is either a quasi-simple triangulation (with a marked vertex) if $B^\bullet_R(\mathrm{q},v^\star)=(\mathrm{q},v^\star)$, or a quasi-simple triangulation of the cylinder $\Delta$, with $|\partial \Delta|=p$, otherwise.

\index{i@\textbf{quasi-simple triangulations of the $p$-gon}!hull-limk@$\kappa(p)$, $p\geq 1$}
\index{i@\textbf{quasi-simple triangulations of the $p$-gon}!hull-limt@$\theta(k)$, $k\geq 0$}
\index{i@\textbf{quasi-simple triangulations of the $p$-gon}!proba@\ttt{$\mathfrak{q}^{(p)}_n$}{uniform in $\kQ_{n,p}$}}
\begin{proposition}\label{prop:forestHull}
Let $p,q,r$ be positive integers. Let $\Delta$ be a quasi-simple triangulation of the $(p,q)$-cylinder with height $r$. Then, for $\mathfrak q_n^{(p)}$ a uniformly random element of $\kQ_{n,p}$, we have:
\begin{equation}
\mathbb{P}\left({B_r^\bullet(\qn^{(p)})}=\Delta\right) \xrightarrow[n\to\infty]{} 
\frac{\kappa(q)}{\kappa(p)}\prod_{v\in \mathfrak{F}(\Delta)^\star}\theta(k_v)\frac{\rho^{\text{Inn}(\mathfrak s_v)}}{Z(k_v+2)},
 \end{equation} 
 where $\kappa$ and $\theta$ are defined by: 
 \begin{equation}
 \kappa(p)= \frac{p}{4^p}\binom{2p}{p} \text{ for } p\geq 1,\quad \text{and}\quad \theta(k) = \frac{3}{2}\frac{1}{4^k}\frac{(2k)!}{k!(k+2)!} \text{ for } k\geq 0.
 \end{equation}
\end{proposition}
\index{j@\textbf{quasi-simple triangulations of the cylinder}!skeldec-slot@$\text{Inn}(\mathrm{t})$ number of inner vertices of the slot-filling triangulation $\mathrm{t}$}

\begin{remark}\label{rem:branching} It follows from the definition of $\theta$ that: 
\[\sum_{k\geq 0}\theta(k)x^k=1-\left(1+\frac{1}{\sqrt{1-x}}\right)^{-2}.\]
Hence our definition of $\theta$ matches the definition given in~\cite[(22)]{Krikun} and in \cite[(15)]{CurienJFLG}. This is not surprising and can be explained via classical substitution relations between general, loopless and simple triangulations. 

\medskip

Note that Proposition~\ref{prop:forestHull} is fully analogous to~\cite[Lemma~2]{CurienJFLG}, except that here $Z$ and $\rho$ denote respectively the generating series and the radius of convergence associated to \emph{simple} triangulations, whereas in~\cite{CurienJFLG}, they correspond to \emph{general} (a.k.a allowing loops and multiple edges) triangulations.  \dotfill
\end{remark}

\begin{proof}
We set $\rho = 27/256$ to be the radius of convergence of the generating series of quasi-simple triangulations and $\beta = 9/64$ to be inverse of the growth rate of $Z(p)$ defined in~\eqref{eq:defZp}. We denote by $|\Delta|$ the total number of vertices of $\Delta$. Recall that $\kQ_{n,p}$ is the set of quasi-simple triangulations of the $p$-gon with $n$ \emph{inner} vertices. Then: 
\begin{equation}\label{eq:probaDelta}
\mathbb{P}\left({B_r^\bullet(\qn^{(p)})}=\Delta\right) = \frac{|\kQ_{n-|\Delta|+p,q}|}{|\kQ_{n,p}|}\stackrel[n\rightarrow \infty]{}{\sim} \frac{C(q)}{C(p)}\rho^{|\Delta|-p},
\end{equation}
where the asymptotic equivalent is obtained from \eqref{eq:asymptoticQEnum}.
\medskip

To lighten notation, in this proof we set $\mathrm f:=\mathfrak{F}(\Delta)$, and recall that $\mathrm f^*$ is the set of its vertices with height strictly smaller than $r$. According to the skeleton decomposition, the vertices of $\Delta$ can be partitioned into two sets. They either belong to one of the cycles $\delta_j\Delta$, for $0\leq j \leq r$, or they are inner vertices of the simple triangulations that are inserted in the slots. Hence $|\Delta|= |f^\star| + \sum_{v\in \mathrm f^\star}\text{Inn}(S_v)$. Since $|\mathrm f^\star|= q + \sum_{v\in \mathrm f^\star}k_v$, we get from~\eqref{eq:probaDelta} that: 
\[
    \mathbb{P}\left({B_r^\bullet(\qn^{(p)})}=\Delta\right) \stackrel[n\rightarrow \infty]{}{\sim}\frac{C(q)}{C(p)}\rho^{q-p}\prod_{v\in \mathrm f^\star}\rho^{\text{Inn}(S_v)+k_v}.
\]
Following similar computations performed in~\cite{Krikun} and~\cite{CurienJFLG}, and in  order to interpret this quantity in terms of a branching process, we perform a    change of variables. Since $\sum_{v\in \mathrm f^\star}(k_v-1)=p-q$, we can multiply the former expression by $(\rho/\beta)^{p-q-\sum_{v\in \mathrm f^\star}(k_v-1)}$ to get:
\[
    \mathbb{P}\left({B_r^\bullet(\qn^{(p)})}=\Delta\right) \stackrel[n\rightarrow \infty]{}{\sim}\frac{\beta^{q}C(q)}{\beta^{p}C(p)}\prod_{v\in \mathrm f^\star}\rho^{\text{Inn}(S_v)+1}\beta^{k_v-1}.
\]
Now, informally, we would like to decompose each term of the product into two contributions. On the one hand, a term representing the probability for a vertex to have $k$ children in the skeleton decomposition and, on the other hand, a term for the probability (under the critical Boltzmann distribution) that a given simple triangulation of the $(k+2)$-gon 
appears in the corresponding slot. 
To do so, we set $\theta(k):=\rho\beta^{k-1}Z(k+2)$, which yields:
\[
    \mathbb{P}\left({B_r^\bullet(\qn^{(p)})}=\Delta\right) \stackrel[n\rightarrow \infty]{}{\sim}\frac{\beta^{q}C(q)}{\beta^{p}C(p)}\prod_{v\in \mathrm f^\star}\theta(k_v)\frac{\rho^{\text{Inn}(S_v)}}{Z(k_v+2)}.
\]
In view of the definition of $\beta$ and $\rho$ and the expression of $Z(k)$ given in~\eqref{eq:defZp}, this concludes the proof.
\end{proof}

\subsection{Local limits: quasi-simple UIPT and half-plane models}\label{sub:locallimits}

\subsubsection{Quasi-simple UIPT}

\index{j@\textbf{quasi-simple triangulations of the cylinder}!z@$(p,q,r)$-admissible forest $\mathrm{f}$!proba@\ttt{$\mathbf{P}_{p,r}$}{distribution on $\kF_{p,r}$}}
In exactly the same manner as in Section~2.4 of~\cite{CurienJFLG}, we can define a probability measure $\mathbf{P}_{p,r}$ on $\kF_{p,r}$ by setting for every forest $\mathrm{f}\in \kF_{p,q,r}$:
\begin{equation}
\mathbf{P}_{p,r}(\mathrm{f}):= \frac{\kappa(q)}{\kappa(p)}\prod_{v\in \mathrm{f}^\star}\theta(k_v).
\end{equation}
\index{j@\textbf{quasi-simple triangulations of the cylinder}!proba@\ttt{$\mathbb{P}_{p,r}$}{distribution on $\kZ_{p,r}$ induced by $\mathbf{P}_{p,r}$}}
We then define a probability measure $\mathbb{P}_{p,r}$ on the set $\cyl p r$, by saying that under $\mathbb{P}_{p,r}$ the law of the skeleton on the triangulation is given by $\mathbf{P}_{p,r}$ and, conditionally given the skeleton, the triangulations filling the slots are independent critical Boltzmann simple triangulations with a simple boundary (whose  boundary lengths are  determined by the skeleton). Proposition~\ref{prop:forestHull} can then be restated as: 
\begin{equation}\label{eq:cvgLocalHull}
\mathbb{P}\left({B_r^\bullet(\qn^{(p)})}=\Delta\right) \xrightarrow[n\to\infty]{} \mathbb{P}_{p,r}(\Delta),
\end{equation}
for any $\Delta\in \cyl p r$.

\index{i@\textbf{quasi-simple triangulations of the $p$-gon}!proba-limloc@\ttt{$\mathfrak{q}_\infty^{(p)}$}{the quasi-simple triangulation of the plane with boundary of length $p$}}
\index{i@\textbf{quasi-simple triangulations of the $p$-gon}!proba-limlocsphere@\ttt{$\mathfrak{q}_\infty = ``\Psi(\mathfrak{q}_\infty^{(1)})''$}{the quasi-simple triangulation of the sphere}}
It follows from the preceding convergence, that we can define a random infinite \emph{quasi-simple} triangulation of the plane with a simple boundary of length $p$ -- denoted $\mathfrak{q}_{\infty}^{(p)}$ -- such that the distribution of $B_r^{\bullet}(\mathfrak{q}_{\infty}^{(p)})$ is given by $\mathbb{P}_{p,r}$, for every integer $r\geq 1$. The map $\mathfrak{q}_{\infty}^{(p)}$ is called the \emph{uniform infinite quasi simple triangulation} of the $p$-gon and it follows directly from~\eqref{eq:cvgLocalHull} that $\mathfrak{q}_{\infty}^{(p)}$ is the local limit of $\qn^{(p)}$ as $n\rightarrow \infty$.                     

By the correspondence $\Psi$ (illustrated in~Figure~\ref{fig:exRoot1gon}) between quasi-simple triangulations of the 1-gon and quasi-simple triangulations of the sphere,
 from $\mathfrak{q}_{\infty}^{(1)}$ we can define the so-called \emph{quasi-simple triangulation of the sphere}, denoted $\mathfrak{q}_{\infty}$.

\subsubsection{Half-plane simple models}\label{subsub:halfPlaneSimpleModels}
Recall from Remark~\ref{rem:branching} that the branching process encoding the skeleton decomposition of quasi-simple triangulations has exactly the same offspring distribution as the one obtained in~\cite{CurienJFLG} for general triangulations. 
% This implies that~\cite[Section~3]{CurienJFLG} can be applied verbatim to our setting. 

\index{k@\textbf{half-plane models}!lhs@\ttt{$\mathfrak{L}^s$}{simple, of the lower half-plane}}
\index{k@\textbf{half-plane models}!lhg@\ttt{$\mathfrak{L}$}{of the lower half-plane}}
\index{k@\textbf{half-plane models}!uhs@\ttt{$\mathfrak{U}^s$}{simple, of the upper half-plane}}
In this work, two half-plane models of random geometry are introduced, the upper half-plane triangulation and the lower half-plane triangulation and we refer to~\cite[Section~3]{CurienJFLG} for their construction. This allows us to define readily the upper half-plane \emph{simple} triangulation $\mathfrak{U}^s$ and the lower half-plane \emph{simple} triangulation $\mathfrak{L}^s$, by only requiring that in their construction, the slots are filled with \emph{simple} triangulations rather than general triangulations. 

All results stated in Sections~3 and 4 of~\cite{CurienJFLG} can be extended verbatim to our setting. Indeed, in most of the proofs they do not need to know how the slots are filled in, and only study the distribution of the skeleton, which is the same in our setting and in~\cite{CurienJFLG}; the only part that we need to change is in \cite[Lemma 10]{CurienJFLG}, where (to reflect the different constants in the enumeration of simple triangulations) $\rho = 12\sqrt 3$ should be replaced by $\rho = 256/27$ and $\alpha = 12$ by $\alpha = 64/9$.  Notice that triangulations of the half-plane that we obtain are simple triangulations and not quasi-simple triangulations, as can be directly observed from the construction of $\mathfrak{U}^s$ and $\mathfrak{L}^s$.

% \TODO{Figure: lower half-plane}

\subsubsection{Coupling between simple and general models}\label{sub:coupling}
It follows from the preceding section that $\mathfrak{q}^{(1)}_\infty$ and $\mathfrak{L}^s$ can be constructed via exactly the same branching process as the (general) UIPT of the 1-gon and $\mathfrak{L}$ respectively. The only modification between these pairs of models is that the slots are filled in the first case by \emph{simple} critical Boltzmann triangulations with a simple boundary and in the second case by \emph{general} critical Boltzmann triangulations with a simple boundary. 

The simple core of a critical Boltzmann general triangulation with a simple boundary of length $\geq 3$ (obtained by collapsing all 2-cycles into edges) is itself a critical Boltzmann simple triangulation with a boundary. So that, by taking the simple core of each triangulation grafted in the slots, this induces a coupling between the (general) UIPT of the $1$-gon and $\mathfrak{q}^{(1)}_\infty$ on one hand and between $\mathfrak{L}$ and $\mathfrak{L}^s$ on the other hand. This remark will be useful to derive bounds in our model directly from the bounds obtained in~\cite{CurienJFLG}. 

\section{Modification of distances in simple triangulations}\label{sec:modDistances}
The goal of this section is to give the proof of Theorem~\ref{th:LGC_simple}. To do that, we follow closely the work of~Curien and Le Gall~\cite{CurienJFLG} and rely heavily on the skeleton decomposition of quasi-simple triangulations presented in the previous section. 

\index{l@\textbf{modification of distances}!distr@\ttt{$\nu$}{weight distribution on $[\eta_0,\infty)$ for $(w_e)_{e\in E(\map^\dagger)}$}}
Throughout this section, we assume that $\nu$ is a probability distribution with support included in $[\eta_0,\infty)$ for a given $\eta_0>0$, and that there exist constants $A,\lambda>0$ such that, for $X$ a random variable with law $\nu$:
\[
    \mathbb{P}\left(X \geq K\right)\leq A\exp(-\lambda K),\quad \text{ for any }K\in \R_+.
\]
Moreover, we assume that any planar map $\map$ (random or not) considered in this section is equipped with a collection $(w_e)_{e\in E(\map^\dagger)}$ of random weights, such that $(w_e)$ are i.i.d random variables sampled from $\nu$. Recall from the introduction that we defined the first-passage percolation distance $\dstar$ on $\map^\dagger$ by setting:
\index{l@\textbf{modification of distances}!fpp@\ttt{$\dstar_\nu = \dstar$}{first-passage percolation distance on $\map^\dagger$}}
\begin{equation}
\dstar_\nu(u,v)=\dstar(u,v):=\inf_{\gamma: u\rightarrow v}\sum_{e\in \gamma}w_e,\quad \text{ for }u,v \in V(\map^\dagger)
\end{equation}
where the infimum runs over all paths $\gamma$ going from $u$ to $v$ in $\map^\dagger$. For the sake of simplicity, we write $\dstar$ without emphasizing its dependency in $\nu$, it should not bear any confusion. 
\bigskip

Since the skeleton decomposition of quasi-simple triangulations gives the exact same branching process as the skeleton decomposition of general triangulations, many results of~\cite{CurienJFLG} can be applied verbatim to our setting. When this is the case, we do not duplicate the proofs. Rather, we only state the results for which the proofs need to be adapted. Note also that in~\cite{CurienJFLG}, the authors focus first on the case of modifications of distances in triangulations (and not on their dual). They then treat the dual case without giving as many details as for the primal case. 

\begin{remark}
Theorem~\ref{th:LGC_simple} and its proof could be extended verbatim to deal with modification of distances on simple triangulations rather than on their dual as was initially done in~\cite{CurienJFLG}. 
However, for sake of conciseness, we only give the proof of this result in the dual setting. \dotfill
\end{remark}

\subsection{Upper bounds in the infinite models}

\index{k@\textbf{half-plane models}!zdelta@\ttt{$\mathfrak{L}^s_j$}{the path connecting the vertices of $\mathfrak{L}^s$ at depth $-j$}}
\index{k@\textbf{half-plane models}!zdeltaF@\ttt{$F_j(\mathfrak{L}^s)$}{the downward triangles incident to an edge of $\mathfrak{L}^s_j$}}
\index{k@\textbf{half-plane models}!zdeltaroot@\ttt{$f_{0,0}$}{the downward triangle of $\mathfrak{L}^s_j$ incident to its root edge}}
The purpose of this section is to derive statements similar to Lemma~24 and Proposition~25 of~\cite{CurienJFLG} for the lower half-plane simple triangulation $\mathfrak{L}^s$. 
For every $j\leq 0$, let $\mathfrak{L}_j^s$ for the vertices of $\mathfrak{L}^s$ that are the vertices of depth $-j$ of its skeleton; or rather, $\mathfrak{L}_j^s$ is the set (forming an infinite simple path) of all edges of $\mathfrak L^s$ that connect these vertices --- it is the analogue of $\partial_{r+j}(\Delta)$ when $\Delta$ is a quasi-simple triangulation of the cylinder of height $r$. We also denote by $\mathfrak{L}^{s,\dagger}_{j}$ the set of all downward triangles incident to an edge of $\mathfrak{L}^s_j$, and by $f_{0,0}$ the face in $\mathfrak{L}^{s,\dagger}_{0}$ that is incident to the root edge of $\mathfrak{L}^s$. 

\begin{figure}
\begin{center}
\includegraphics[width=9cm]{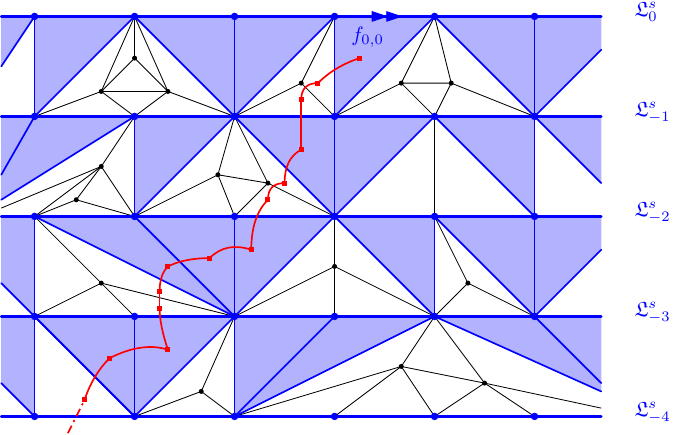}
\end{center}
\caption{Representation of the first 4 layers of $\mathfrak{L}^s$. The downward path started from $f_{0,0}$ is represented in fat red edges.}
\label{fig:LHPs}
\end{figure}

For every downward triangle $f$ in $\mathfrak{L}^{s,\dagger}_{j}$, the downward path from $f$ is an infinite path in the dual map, defined as follows (see \cite[Section 7.2]{CurienJFLG} for a precise definition): starting from $f$, and if $v$ is the vertex of $f$ that is not an endpoint of an edge of $\mathfrak{L}^s_j$, turn around $v$ in counterclockwise order until crossing an edge of $\mathfrak{L}_{j-1}^s$. The face we reach is itself in $\mathfrak{L}^{s,\dagger}_{j-1}$: we continue inductively with the downward path starting from this latter face.

\begin{lemma}\label{lem:expDegree}
For $r\geq 1$, let $\gamma_r$ be the downward path in $\mathfrak{L}^s$ connecting $f_{0,0}$ to a downward triangle incident to $\mathfrak{L}^s_{-r}$, and write $|\gamma_r|$ for the length of $\gamma_r$. There exist two constants $\mu>0$ and $K<\infty$ such that for every integer $r\geq 1$: 
\[
    \mathbf{E}[\exp(\mu\,|\gamma_r|)]\leq K^r.
\]
\end{lemma}
\begin{proof}
The result follows from Lemma~24 of~\cite{CurienJFLG}, and from the coupling described in Section~\ref{sub:coupling}, under which the length 
of the downward path can only decrease. 
\end{proof}

For every $j\leq 0$, let $\mathfrak{L}^{s,\dagger}_{j}$ be the set of downward triangles incident to an edge of $\mathfrak{L}^s_{j}$ in the lower half-plane model. We assume that $\mathfrak{L}^s$ is equipped with the f.p.p distance $\dstar$ defined above (the dual map is infinite here, its vertices corresponding to the triangular faces of $\mathfrak{L}^s$, and its edges corresponding to the edges in $\mathfrak{L}^s$ with a triangular face on each side). 
Then, Proposition~25 of~\cite{CurienJFLG} can be readily adapted into the following result: 
\begin{proposition}
Recall that the support of $\nu$ is included in $[\eta_0,\infty)$. There exists a constant $c_\nu\geq 2\eta_0$ such that: 
\[
    r^{-1}\dstar(f_{0,0},\mathfrak{L}^{s,\dagger}_{-r})\xrightarrow[r\rightarrow \infty]{a.s.}c_\nu.
\]
\end{proposition}
As in the original article, the result follows directly from a subadditive argument together with the fact that $\E[\dstar(f_{0,0},\mathfrak{L}^{s,\dagger}_{-r})]<\infty$ by Lemma~\ref{lem:expDegree} and $\dstar(f_{0,0},\mathfrak{L}^{s,\dagger}_{-r})\geq 2\eta_0 r$ by construction.
\bigskip 

Thanks to Lemma~\ref{lem:expDegree}, the proof of Lemma~26 and Corollary~27 of~\cite{CurienJFLG} can directly be extended to our setting. We obtain:
% \begin{lemma}\label{lem:2627}
% In the quasi-simple UIPT of the 1-gon $\mathfrak{q}_{\infty}^{(1)}$, for any $\nu$ satisfying the conditions stated at the beginning of the section, there exist positive constants $K,\alpha,\beta$ such that, for any integer $0\leq r<s$:
% \begin{equation}
% \mathbb{P}\left(f_s,B_{r}^\bullet(\mathfrak{q}_{\infty}^{(1)})>\alpha (s-r)\right)\leq K \exp(-\beta(s-r)), 
% \end{equation}
% where $f_s$ is the downward triangle at height $s$ chosen uniformly at random, and where $B_0^\bullet(\mathfrak{q}_{\infty}^{(1)})$ should be interpreted as the bottom face.
% \end{lemma}
\begin{lemma}\label{lem:2627}
Consider the quasi-simple UIPT of the 1-gon $\mathfrak{q}_{\infty}^{(1)}$, with i.i.d weights on its edges sampled from $\nu$. There exist positive constants $K',\alpha_0,\beta_0$ such that, for any integer $0\leq r<s$:
\begin{equation}
\mathbb{P}\left( \dstar (f_s,B_{r}^\bullet(\mathfrak{q}_{\infty}^{(1)}))>\alpha_0 (s-r)\right)\leq K' \exp(-\beta_0(s-r)), 
\end{equation}
where $f_s$ is a downward triangle at height $s$ chosen uniformly at random, and where $B_0^\bullet(\mathfrak{q}_{\infty}^{(1)})$ should be interpreted as the bottom face.

Moreover, for any fixed $\delta>0$, define $A_R(\delta)$ to be the event where the bound 
\begin{equation}
\dstar (f,B_r^\bullet(\mathfrak{q}_{\infty}^{(1)}))\leq \alpha_0 (s-r)
\end{equation}
holds for every $0\leq r<s\leq R$ with $s-r\geq \delta R$, and for every downward triangle $f$ at height $s$. Then, there exists another positive constant $\tilde \beta$ such that for any sufficiently large $R$, 
\begin{equation}
\mathbb{P}(A_R(\delta))\geq 1-e^{-\tilde\beta R}.
\end{equation}
\end{lemma}

\subsection{Upper bound in finite simple triangulations}
We finally turn our attention to finite triangulations. The purpose of this section is to obtain a result equivalent to \cite[Lemma~28]{CurienJFLG} but for uniform simple triangulations rather than uniform general triangulations and for the $\dstar$ distance rather than the graph distance or the Eden model. 

\index{e@\textbf{maps and graphs}!inc@\ttt{$v\triangleleft f$}{vertex $v$ and face $f$ are incident}}
 For $x\in V(\mathrm{t}_n)$ and $f\in F(\mathrm{t}_n)$, write $x \triangleleft f$ to mean that the vertex $x$ is incident to the face $f$. Then we have: 
\begin{proposition}\label{prop:lemmm28}
 Let $\alpha_0$ be as in Lemma~\ref{lem:2627}. Let $\tn$ be a uniform rooted simple triangulation with $n+2$ vertices. For $\varepsilon \in (0,1/4)$, define $E_n$ as the event where the bound
 \begin{equation}\label{eq:upperboundFinite}
 \dstar(f,g)\leq \alpha_0 \dgr(x,y) + n^{\varepsilon}
 \end{equation}
 holds for every $x,y \in V(\mathrm{t}_n)$ and $f,g\in F(\mathrm{t}_n)$ such that $x\triangleleft f$ and $y\triangleleft g$. Then: 
 \begin{equation}
 \mathbb{P}(E_n)\xrightarrow[n\rightarrow \infty]{} 1.
 \end{equation}
The same result holds when replacing $\tn$ by $\qn^{(1)}$, where $\qn^{(1)}$ is a uniform quasi-simple triangulation of the 1-gon with $n$ inner vertices. 
\end{proposition} 
\begin{proof}
The main difficulty that arises in our setting and does not exist in~\cite{CurienJFLG} is that we have to derive results for simple triangulations, whereas we established (in particular in the previous section) results about the local limit of \emph{quasi-simple} triangulations. Fortunately, these two models are similar enough to obtain the results we need. 

In this proof, we always consider that the random planar maps are endowed with random weights on their edges, sampled independently under a law $\nu$ with support included in $[\eta_0,\infty)$ for some $\eta_0>0$, and with exponential tails. 
Let $\qn^{(1)}$ be a uniform quasi-simple triangulation of the $1$-gon with $n$ inner vertices. Following the proof of~\cite[Lemma~28]{CurienJFLG}, we write $\rho_n$ for its root vertex, $o_n$ for its pointed vertex and set $d_n:=\dgr(\rho_n,o_n)$. Then, we write $\Theta$ for the set of quasi-simple triangulations $\mathrm t$ that belong to $\mathcal Z_{1,r}$ for some $r>1$ and are such that there exists a face $f$ incident to $\partial^*\mathrm t$ whose $\dstar$ from the bottom face is greater than $\alpha r$. By the exact same chain of arguments as in the proof of~\cite[Lemma~28]{CurienJFLG} that leads to their equation (73), we obtain that there exist two constants $c>0$ and $a>0$, such that: 
\begin{equation}\label{eq:ballNice}
\mathbb{P}\left(d_n>n^{\epsilon}\,;\,B_{d_n-1}^\bullet(\qn^{(1)})\in \Theta\right)\leq c\exp(-an^{\epsilon}). 
\end{equation}
By, again, the same line of arguments as in~\cite{CurienJFLG} and given the definition of $\Theta$, it is then easy to see that on the event $\{d_n>n^\epsilon;B_{d_n-1}^\bullet(\qn^{(1)})\notin \Theta\}$, we have $\dstar(f,g)\leq \alpha d_n+3w^*\cdot\maxdeg(\qn^{(1)})$, for any $f,g\in \qn^{(1)}$, with $\rho_n\triangleleft f$ and $o_n\triangleleft g$ (where we denote by $\maxdeg$ the maximum vertex degree of a map, and by $w^*$ the maximal dual weight of an edge of $\tn$). 
% This result clearly also holds for the quasi-simple triangulation of the sphere via the mapping $\Psi$ illustrated in Figure~\ref{fig:exRoot1gon}. 
% we have $\dgr^\dagger(f,g)\leq \alpha d_n+3\maxdeg(\qn^{(1)}\})$, for any $f,g\in \mathfrak{q}_{n^{(1)}}$, with $\rho_n\triangleleft f$ and $o_n\triangleleft g$ (where we recall that $\maxdeg$ denotes the maximum vertex degree of a map). This result clearly also holds for the quasi-simple triangulation of the sphere obtained via the procedure $\Psi$ illustrated in Figure~\ref{fig:exRoot1gon}. 

The original proof continues with a rerooting argument, which cannot apply here since quasi-simple triangulation are not invariant under rerooting. However, we can transfer the bound obtained in~\eqref{eq:ballNice} from quasi-simple triangulations to simple triangulations (up to changing the value of the constant $c$) in the following way. The asymptotic enumerative results for simple triangulations and quasi-simple triangulations of the $1$-gon given respectively in~\eqref{eq:asymSimpleTrig} and~\eqref{eq:asymptoticQEnum} imply that there exists $\kappa>0$, such that: 
% \ctom{\begin{equation}\label{eq:proportionPos}
% \mathbb{P}\left(\Psi(\qn^{(1)})=t\right) > \kappa \mathbb{P}\left(\mathfrak{t}^\carre_{n-1}=t\right) \quad\text{for every }n \geq 1\text{ and }t\in \cTc_{n-1} ,
% \end{equation}
\begin{equation}\label{eq:propPos}
\mathbb{P}\left(\Psi(\qn^{(1)})\in \cTc_{n-1}\right)> \kappa, \quad\text{for any }n \geq 1,
\end{equation}
% {where $\mathfrak{t}^\carre_n$ is a uniformly random triangulation over $\cTc_n$ and} 
where we recall that $\cTc_n$ denotes the set of rooted simple triangulations with $n+2$ vertices and a marked inner vertex.

Moreover, conditionally on the fact that $\Psi(\qn^{(1)})\in \cTc_{n-1}$, it is clear from the construction that $\qn^{(1)}$ has the same distribution as a uniformly random element of $\cTc_{n-1}$ (as a random measured metric space) \footnote{Note that their dual are not exactly isometric due to the root-edge being doubled, but distances in the dual of both models differ at most by 1.}. 
 % the set of simple triangulations is clearly included in the image of the mapping $\Psi$, \ctom{--- ce début de phrase est redondant avec le changement proposé à \eqref{eq:proportionPos}, supprimer ---} \ceri{and $\Psi$ preserves the underlying metric space}.  

Let $\tn$ be uniformly random in $\cTc_n$. Write $\tilde \rho_n$ for its root vertex, $\tilde o_n$ for its pointed vertex and $\tilde d_n$ for $\dgr(\tilde \rho_n,\tilde o_n)$. It follows from the discussion above and from the discussion following~\eqref{eq:ballNice} that 
there exist two constants $\tilde c>0$ and $\tilde a>0$ such that:
\begin{equation}
\mathbb{P}\left(\tilde d_n>n^{\epsilon}\,;\,\dgr^\dagger(f,g)>\alpha_0 \tilde d_n+3w^\star\cdot\maxdeg(\tn) \text{ whenever }\tilde \rho_n\triangleleft f \text{ and }\tilde o_n \triangleleft g\right)\leq \tilde c \exp(-\tilde a n^{\epsilon}). 
\end{equation}
Now, by rerooting $\tn$ at another oriented edge $e_n$ chosen uniformly and independently from $o_n$, we obtain a pointed and rooted simple triangulation with the same distribution. Hence, the bound obtained above remains valid upon replacing $\tilde \rho_n$ by $\rho'_n$, where $\rho'_n$ is the initial vertex of $e_n$ and $\tilde d_n$ by $\dgr(\tilde o_n,\rho'_n)$.

We can then proceed as in the original proof to establish that: 
\begin{multline}
\E\left[\sum_{u,v\in V(\tn)}\mathbf{1}_{\{\dgr(u,v)>n^{\epsilon}\} }\mathbf{1}_{\{\dgr^{\dagger}(f,g)>\alpha_0 \dgr(u,v)+3w^\star\cdot\maxdeg(\tn)\text{ whenever }u\triangleleft f \text{ and }v\triangleleft g\}}\right]\\\leq 6(n+2)^2 \tilde c \exp(-\tilde a n^{\epsilon}).
\end{multline}
Hence, asymptotically almost surely, the following bound holds: 
\begin{equation}
\dstar(f,g)\leq \alpha_0 \dgr(u,v)+3w^\star\cdot\maxdeg(\tn),
\end{equation}
whenever $u,v\in V(\tn)$, $\dgr(u,v)>n^{\epsilon}$ and $u\triangleleft f$ and $v\triangleleft g$.

Since $\nu$ is assumed to have exponential tails, $w^\star$ can be bounded by $n^{\epsilon}$ outside a set of probability tending to 0 as $n\rightarrow \infty$. Besides, the geometric bound obtained for the degree of the root vertex in~\cite[Lemma~4.1]{AngelSchramm} gives the existence of a constant $A>0$ for which $\maxdeg(\tn)\leq A\log n$ with probability tending to 1 as $n\rightarrow\infty$.
This gives the desired result, except for the requirement that $\dgr(u,v)>n^{\epsilon}$. This case can be dealt with by hand using the fact that $\dstar(f,g)\leq w^\star\cdot\maxdeg(\tn)(\dgr(u,v)+1)$, which concludes the proof for simple triangulations.
\medskip

To finish the proof for quasi-simple triangulations, we rely on Lemmas~\ref{lem:largestSimple} and~\ref{lem:almostUniform}. Fix $\epsilon>0$ and recall that we write $T(\mathfrak q)$ for the largest simple component of a quasi-simple triangulation $\mathfrak q$. By Lemma~\ref{lem:largestSimple}, there exists $K\in \N$ such that: 
\begin{equation}
\mathbb{P}\left(|T(\qn^{(1)})| \geq n-K\right)\geq 1- \epsilon.
\end{equation}
We fix $k \in [0,K]$ and $|T(\qn^{(1)})|=n-k$. Then by Lemma~\ref{lem:almostUniform}, we can couple $T(\qn^{(1)})$ with a uniform simple triangulation $\mathfrak t_{n-k}$ of size $n-k$ such that $\mathbb{P}(T(\qn^{(1)})\neq\mathfrak t_{n-k})< \epsilon$. We extend this coupling to the model with weights on the edges and make sure that the weights also coincide when the triangulations are equal. 

By applying the result of the proposition for $\mathfrak t_{n-k}$, we know that outside a set of probability smaller than $\epsilon$, the bound~\eqref{eq:upperboundFinite} holds. This gives the result when $u,v\in T(\qn^{(1)})$. For $w\notin T(\qn^{(1)})$, we can use the easy fact that there exists $u\in T(\qn^{(1)})$ with $\dgr(u,w)\leq k$, which concludes the proof of the proposition. 
\end{proof}

\subsection{Two-point function}
Thanks to the bounds we obtained in the two previous sections, we are now ready to prove the equivalent of Proposition~21 of \cite{CurienJFLG}, stated below as Proposition~\ref{prop:2point}. Roughly speaking, this proposition establishes that the difference between the 2-point function for the graph distance and the $\dstar$ 2-point function is $o(n^{1/4})$. 

The proof of Proposition~\ref{prop:2point} relies on the following similar result obtained for the \emph{infinite} quasi-simple triangulation. Given the upper bounds established in Lemma~\ref{lem:2627}, we can apply verbatim the proof presented in Proposition 19 (in the primal case) and in Proposition 29 (in the dual case) of~\cite{CurienJFLG} (see the remark at the end of Section \ref{subsub:halfPlaneSimpleModels}) to readily obtain:

\begin{proposition}\label{prop:29}
Let $\epsilon>0$ and $\delta>0$. We can find $\eta \in (0,1/2)$ such that, for every sufficiently large $n$, we have: 
\begin{equation}
\mathbb{P}\left((1-\epsilon)c_{\nu}\lfloor \eta n\rfloor \leq \dstar\left(f,B_{n-\lfloor \eta n\rfloor}(\mathfrak{q}_\infty^{(1)})\right) \leq (1+\epsilon)c_{\nu}\lfloor \eta n\rfloor,\,\forall f\in \mathrm{F}_n(\mathfrak{q}_{\infty}^{(1)})\right) > 1-\delta ,
\end{equation}
where we recall that $\mathrm{F}_n(\mathfrak{q}_{\infty}^{(1)})$ is the set of downward triangles incident to $\partial_n(\mathfrak{q}_{\infty}^{(1)})$. 
Moreover
\begin{equation}
\mathbb{P}\left((c_{\nu}-\epsilon) n \leq \dstar\left(f,B_{0}(\mathfrak{q}_\infty^{(1)})\right) \leq (c_{\nu}+\epsilon) n,\,\forall f\in \mathrm{F}_n(\mathfrak{q}_{\infty}^{(1)})\right) \xrightarrow[n\rightarrow \infty]{}1.
\end{equation}
\end{proposition}

To transfer this result to finite simple triangulations, we rely on the following absolute continuity relation between finite and infinite quasi-simple triangulations (analogous to~\cite[Lemma~22]{CurienJFLG}). Recall that $\cyl 1 r$ denotes the set of quasi-simple triangulation of the cylinder of height $r$ and rooted on a loop. For $\Delta\in \cyl 1 r$, we denote by $|\Delta|$ the number of inner vertices of $\Delta$. Recall that $\mathfrak q_n^{(1)}$ denotes the random quasi-simple triangulation of the 1-gon with $n$ inner vertices. 
Then, we have:
\begin{lemma}\label{ref:absContinuite}
There exists a constant $\bar c$, such that for every $n\geq 1$, for every $r\geq 1$ and every $\Delta \in \cyl 1 r$ such that $n>|\Delta|$, 
\begin{equation}
\mathbb{P}\left(B_r^\bullet(\mathfrak q_n^{(1)})=\Delta\right) \leq \bar c \left(\frac{n}{n-|\Delta|}\right)^{3/2}\mathbb{P}\left(B_r^\bullet(\mathfrak{q}_\infty^{(1)})=\Delta\right). 
\end{equation}
\end{lemma}
\begin{proof}  
The result follows directly from the uniform upper bound obtained in Lemma~\ref{lem:unifBounds}, from a similar lower bound on $|\kQ_{n,1}|$ that we deduce from \eqref{eq:qsimpleTrigEnum}, and from the limit of \eqref{eq:probaDelta} together with the fact that $\mathfrak{q}_\infty^{(1)}$ is the local limit of $\qn^{(1)}$ as $n\rightarrow \infty$.
% Thomas -- version précédente du paragraphe :
%
%The result follows directly from the asymptotic enumerative formulas for quasi-simple triangulations obtained in \eqref{eq:asymptoticQEnum}, from the uniform upper bound obtained in Lemma~\ref{lem:unifBounds}, \ctom{the lower bound $|\kQ_{n,1}| \geq c C(1) (256/27)^n n^{-3/2}$ that we obtain from the explicit formula for $|\kQ_{n,1}|$, the limit of \eqref{eq:probaDelta}} and finally from the fact that $\mathfrak{q}_\infty^{(1)}$ is the local limit of $\qn^{(1)}$ as $n\rightarrow \infty$. 
\end{proof}

We can now state the main result of this section:
\begin{proposition}\label{prop:2point}
Denote by $f_*$ the bottom face of $\qn^{(1)}$. Let $o_n$ be a uniformly distributed inner vertex of $\qn^{(1)}$ and let $f_n$ be a face incident to $o_n$ (fixed in some deterministic manner given $o_n$). Then, for any $\epsilon>0$, 
\begin{equation}\label{eq:2pointIndependant}
\mathbb{P}\left(|\dstar(f_n,f_*)-c_\nu \dgr(o_n,\rho_n)|>\epsilon n^{1/4}\right)\xrightarrow[n\rightarrow \infty]{}0. 
\end{equation}
The same result holds also if we replace $\qn^{(1)}$ by a random rooted simple triangulation having $n+2$ vertices and a random marked vertex $o_n$, 
and $f_*$ by its root face. 
\end{proposition}
\begin{proof}
Again, we follow closely the proof of~\cite[Proposition~21]{CurienJFLG}. Only two modifications must be made. First, if we follow verbatim their proof, then at the very end, we need to adapt the upper bound they obtain. Indeed, at this point they use the fact that the first-passage distance on the primal graph (i.e on the triangulation) is upper-bounded by the graph distance on the primal. In particular, they fix three close values $\alpha_j<\beta_j<\gamma_j$ and restrict their attention to the event $\{\beta_j n^{1/4}\leq \dgr(\rho_n,o_n)\leq \gamma_j n^{1/4}\}$. Then, they can give a deterministic upper bound for the first-passage percolation distance between $o_n$ and the hull of radius $\lfloor \alpha_j n^{1/4}\rfloor$. This last step does not apply directly anymore when working on the dual graph, but can be made valid if we replace the deterministic bound by the probabilistic bound obtained in Proposition~\ref{prop:lemmm28}. 

Secondly, write $\mathfrak{B}_r(\qn^{(1)},o_n)$ for the ball of radius $r$ centered at $o_n$ in $\qn^{(1)}$. We have to establish the following analogous statement of \cite[(62)]{CurienJFLG}, and prove that for any $y<1$, there exists $b\in(0,1)$ such that: 
\begin{equation}
\liminf_{n\rightarrow \infty}\mathbb{P}\left(|\mathfrak{B}_{\lfloor(\beta-\alpha)n^{1/4}\rfloor}(\qn^{(1)},o_n)|>bn\right)\geq y. 
\end{equation}
Given the convergence of simple triangulations to the Brownian sphere in the Gromov--Hausdorff-Prokhorov topology~\cite{SimpleTrig}, this statement is known to be true for uniform simple triangulations. It can then be extended to quasi-simple triangulations, upon extracting their largest simple component and by applying Lemmas~\ref{lem:largestSimple} and~\ref{lem:almostUniform}. This concludes the proof for quasi-simple triangulations.
\medskip

To extend the result to simple triangulations, we rely on the same line of arguments as in the proof of Proposition~\ref{prop:lemmm28}, and particularly on~\eqref{eq:propPos}. 
It states that a positive proportion of quasi-simple triangulations are in fact simple triangulations (after applying the mapping $\Psi$, as depicted in Figure~\ref{fig:exRoot1gon}) 
% \ctom{--- Il y a le même flou là que vers \eqref{eq:proportionPos}, le même problème, et la même correction. Je rappelle que la précédente version de \eqref{eq:proportionPos} n'empêche pas, par exemple, que $\Psi$ envoie une grande masse des triangulations quasi-simples dans la toute petite masse des triangulations simples qui sont sympa et vérifient nos propriétés, alors que tout le reste des triangulations simples foirent. ---}. Moreover, 
Moreover, when $\Psi(\mathfrak{q}^{(1)}_n)$ is a simple triangulation, it yields a direct coupling between a random inner vertex of $\mathfrak{q}^{(1)}_n$ and of $\Psi(\mathfrak{q}^{(1)}_n)$. Since we have just shown that~\eqref{eq:2pointIndependant} holds for quasi-simple triangulations,
and since $\Psi$ preserves the underlying metric space, it has to hold for simple triangulations as well. 
\end{proof}

\subsection{Coupling between \texorpdfstring{$\mu_n$}{mun} and \texorpdfstring{$\mu_n^\dagger$}{mund} and proof of Theorem~\ref{th:LGC_simple}}

\index{g@\textbf{3-connected cubic planar graphs}!rand@\ttt{$\kn$}{uniform in $\kK_n$}}
\index{g@\textbf{3-connected cubic planar graphs}!rand@\ttt{$\kn$}{uniform in $\kK_n$}!GHP@equipped with $\dstar$ and the uniform measure $\mu_n^\dagger$ on $V(\kn)$}
\index{t@\textbf{rooted simple planar triangulations}!tglrand2@$\mathfrak{T}_n = (\tn,\dgr,\mu_n)$ where $\mu_n$ is the uniform measure on $V(\tn)$}
\begin{theorem}\label{th:correspondance}
Let $\tn$ be uniformly random in $\kT_n$. Consider the two following weighted metric spaces $\mathfrak{T}_n=(\tn,\dgr,\mu_n)$ and $\kn=(\tn^\dagger,\dstar,\mu_n^\dagger)$, where $\mu_n$ and $\mu_n^\dagger$ denote respectively the uniform measures on $V(\tn)$ and on $V(\tn^\dagger)$. 

\index{l@\textbf{modification of distances}!zcorr@\ttt{$C_n$}{correspondence between $(\tn,\dgr)$ and $(\tn^\dagger,\dstar)$}}
\index{l@\textbf{modification of distances}!zcoupl@\ttt{$\pi_n$}{coupling between $\mu_n$ and $\mu^\dagger$}}
Then, there exists an explicit correspondence $C_n$ between $(\tn,\dgr)$ and $(\tn^\dagger,\dstar)$ and an explicit coupling $\pi_n$ between $\mu_n$ and $\mu_n^\dagger$, such that:
\[
    \pi_n(C_n)\geq 1-\frac{2}{n+2} \qquad\text{ and }\qquad \dis(C_n)\leq \epsilon n^{1/4}.
\]
\end{theorem}
Before giving the proof of this theorem, notice that in view of the known convergence of simple triangulations towards the Brownian sphere established in~\cite{SimpleTrig}, it concludes the proof of Theorem~\ref{th:LGC_simple}. 

\begin{proof}
The correspondence $C_n$ is defined as the set of pairs:
\[C_n:=\{(v,f),\text{ with }v\in V(\tn),\, f\in F(\tn) \text{ and }v\triangleleft f\}.\]
Next, the coupling we construct between $\mu_n$ and $\mu_n^\dagger$ is based on the theory of Schnyder woods for simple triangulations~\cite{Schnyder}, that we now briefly recall. A simple triangulation $\mathfrak t$ can be canonically endowed by an orientation of its edges, such that for each $v\in V(\mathfrak{t})$, writing $\mathrm{out}(v)$ for the outdegree of $v$, we have:
\begin{equation}
\mathrm{out}(v)=
 \begin{cases}
 3&\text{if }v \text{ is not incident to the root face of }\mathfrak{t}\\
 1&\text{otherwise}.
\end{cases}
\end{equation} 
Note that to ensure that this orientation is canonical, it is enough to require that directed counterclockwise cycles are not allowed, but this is not relevant for our purpose. However, this canonical orientation enables to construct an explicit coupling between $\mu_n$ and $\mu_n^{\dagger}$. 

Letting $o_n$ be a random vertex of $\tn$, we associate to $o_n$ a uniform face of $\tn$ in the following way. If $o_n$ is not incident to the root face of $\tn$, then let $e_n$ be an edge chosen uniformly among one of the three outgoing edges from $o_n$. With probability 1/2, we pick $f_n$ to be the edge on the left of $e_n$, and otherwise $f_n$ is the edge on the right of $e_n$. 

Now, assume $o_n$ is incident to the root face of $\tn$. Then, with probability 1/3,  we let $e_n$ to be the unique outgoing edge from $o_n$, and with probability 2/3, we let $e_n$ to be a uniform edge of $\tn$. In both cases, we proceed as above. 

It is clear from the construction, that $e_n$ is a uniform edge of $\tn$. Since each face has degree~$3$, $f_n$ is then a uniform face of $\tn$. Moreover, apart on an event of probability at most $2/(n+2)$ (corresponding to the case where $o_n$ is incident to the root face and $e_n$ is sampled uniformly in $E(\tn)$), we have $o_n\triangleleft f_n$. 
\medskip

To conclude the proof of the theorem, it is then enough to prove that $\dis(C_n)\leq \epsilon n^{1/4}$, or in other words that: 
\begin{equation}\label{eq:supVF}
\mathbb{P}\left(\sup_{\substack{x,y\in V(\tn),f,g\in F(\tn)\\x\triangleleft f,y\triangleleft g}}|\dstar(f,g)-c_\nu \dgr(x,y)|>\epsilon n^{1/4}\right)\xrightarrow[n\rightarrow \infty]{}0. 
\end{equation}
To prove this last statement, we only need to adapt slightly the framework of~\cite[Theorem~1, p.680]{CurienJFLG}, and thus we only sketch the proof. 

We start by presenting the general strategy of the proof. First, thanks to a rerooting argument, we can extend the result for the 2-point function stated in Proposition~\ref{prop:2point} to control the pairwise distance between any finite number of random points. Next, we rely on the known convergence of simple triangulations to the Brownian sphere in the Gromov--Hausdorff--Prokhorov sense (established in~\cite{SimpleTrig}) to transfer the compactness of the Brownian sphere from the continuum world to the discrete world. This allows us to approximate the metric structure of a large simple random triangulation by restricting our attention to a (large) fixed number of uniform points. Lastly (and this is the only difference with the proof of~\cite[Theorem~1]{CurienJFLG}), Proposition~\ref{prop:lemmm28} extends this approximation result from a simple triangulation to its dual. 
\medskip 

Let us now state more formally those main steps. Let $\ft_n$ be a uniform element of $\kT_n$. Let $\Phi_n : V(\ft_n)\rightarrow F(\ft_n)$ be a deterministic function (given $\ft_n$) which associates to each vertex $v$ of $\ft_n$ a face $f$ with $v\triangleleft \Phi_n(v)$. Let $o'_n$ and $o''_n$ be two independent uniform vertices of $\ft_n$. 
By the same chain of arguments as in~\cite[p.680]{CurienJFLG}, it follows from Proposition~\ref{prop:2point} that for any $\epsilon>0$:
\begin{equation}\label{eq:2ptVertices}
\mathbb{P}\left(|\dstar(\Phi_n(o'_n),\Phi_n(o''_n))-c_\nu \dgr(o'_n,o''_n)|>\epsilon n^{1/4}\right)\xrightarrow[n\rightarrow \infty]{}0. 
\end{equation}

Now let $(o_n^i)_{i\geq 1}$ be a sequence of independent uniform vertices of $\ft_n$. For any $\delta>0$, there exists $N \geq 1 $ such that for every $n\in \N$: 
\begin{equation}\label{eq:compact}
\mathbb{P}\left(\sup_{x\in V(\ft_n)}\left(\inf_{1\leq j\leq N}\dgr(x,o_n^j)\right)<\epsilon n^{1/4}\right)>1-\delta.
\end{equation}
% For $N$ fixed, %so that the latter holds, 
% \eqref{eq:2ptVertices} implies that, for $n$ large enough: 
% \begin{equation}\label{eq:mesh}
% \mathbb{P}\left(\bigcap_{1\leq i\leq j \leq N}\left\{|\dstar(\Phi_n(o_n^i),\Phi_n(o_n^j))-c_\nu\dgr(o_n^i,o_n^j)|\right\}\leq \epsilon n^{1/4}\right)\geq 1-\delta.
% \end{equation}
By the triangle inequality, we can then write: 
% \begin{align*}
% &\sup_{\substack{x,y\in V(\tn),f,g\in F(\tn)\\x\triangleleft f,y\triangleleft g}}|\dstar(f,g)-c_\nu \dgr(x,y)| \\
% \leq &\quad\sup_{1\leq i\leq j \leq N}\left\{\Big|\dstar(\Phi_n(o_n^i),\Phi_n(o_n^j))-c_\nu\dgr(o_n^i,o_n^j)\Big|\right\}
% +2\sup_{x\in V(\ft_n)}\left(\inf_{1\leq j\leq N}\dgr(x,o_n^j)\right)\\
% &+2\sup_{f\in F(\ft_n)}\left(\inf_{1\leq j\leq N}\dstar(f,\Phi_n(o_n^j))\right)\\
% \end{align*}
\begin{multline*}
\sup_{\substack{x,y\in V(\tn),f,g\in F(\tn)\\x\triangleleft f,y\triangleleft g}}|\dstar(f,g)-c_\nu \dgr(x,y)| \\
\leq \quad\sup_{1\leq i\leq j \leq N}\left\{\Big|\dstar(\Phi_n(o_n^i),\Phi_n(o_n^j))-c_\nu\dgr(o_n^i,o_n^j)\Big|\right\}
+2\sup_{x\in V(\ft_n)}\left(\inf_{1\leq j\leq N}\dgr(x,o_n^j)\right)\\
\qquad \qquad \qquad \qquad +2\sup_{f\in F(\ft_n)}\left(\inf_{1\leq j\leq N}\dstar(f,\Phi_n(o_n^j))\right)\\
\end{multline*}
We take $N$ large enough so that~\eqref{eq:compact} holds. By Proposition~\ref{prop:lemmm28}, outside a set of probability at most $\delta$ and provided that $n$ is large enough, the last term in the upper bound can be bounded by $2\sup_{x\in V(\ft_n)}\left(\inf_{1\leq j\leq N}\dgr(x,o_n^j)\right)+n^{1/8}$. 
In view of~\eqref{eq:compact} and~\eqref{eq:2ptVertices} (applied simultaneously to all pairs $(o_n^i,o_n^j)$), this concludes the proof. 
\end{proof}

\section{Scaling limit of the 3-connected core with distances induced by the connected graph}\label{sec:two_point_dep}

Informally, the purpose of this section is to establish a result analogous to Proposition~\ref{prop:2point} (and then a result analogous to~\eqref{eq:supVF}),   
but where we relax the independent hypothesis for the weights on the edges. 
Recall from Section~\ref{sec:marked_3c} that $\kC_n^{(q)}$  is the set of cubic connected planar graphs 
of size $n$ having a marked $3$-connected component of size $q$,
 that $\cnq$ is a uniformly random element of $\kC_n^{(q)}$ (for $q>n/2$ it coincides with the random connected cubic 
 planar graph of size $n$ conditioned to have its $3$-connected core of size $q$). Similarly, $\kM_n^{(q)}$  is the set of cubic connected planar multigraphs 
of size $n$ having a marked $3$-connected component of size $q$, 
 and $\mnq$ denotes a uniformly random element of $\kM_n^{(q)}$ (for $q>n/2$ it coincides with the random connected cubic 
 planar multigraph of size $n$ conditioned to have its $3$-connected core of size $q$). Recall also that for $\mathrm g\in \kC_n$, we write $K(\mathrm g)$ for the 3-connected core of $\mathrm g$. 

 We aim at proving the following result (analogous to~\eqref{eq:supVF},
 with a dependence on the edge-lengths):
\index{g@\textbf{3-connected cubic planar graphs}!zdcore@\ttt{$\mathrm{dist}_{\mg}$}{distance on $K(\mg)$ induced by the graph distance on $\mg$}}
\begin{theorem}\label{th:scaling3connex}
Let $(q(n),n\geq 1)$ be such that $q(n)=\alpha n + O(n^{2/3})$ (where $\alpha\approx 0.85$ is defined in Proposition~\ref{prop:core} and with the informal notation introduced below Lemma~\ref{lem:size_core}). 
Write respectively $\mathrm{dist}_{\cn^{(q(n))}}$ and $\dgr$ for the distance on $K(\cn^{(q(n))})$ induced by the graph distance on $\cn^{(q(n))}$ and for the graph distance on $K(\cn^{(q(n))})$. 

Then, there exists a constant $c>0$ such that, for any $\eps>0$: 
\begin{equation}
\mathbb{P}\left(\sup_{x,y\in V(K(\cn^{(q(n))}))} |\mathrm{dist}_{K(\cn^{(q(n))})}(x,y)-c\, \dgr(x,y) | \geq \eps\, q^{1/4}\right) \xrightarrow[n\to\infty]{} 0
\end{equation}
A similar result holds for the random multigraph $\mn^{(q(n))}$, when $q(n)=\tfrac{199}{316}n+O(n^{2/3})$. 
\end{theorem}

We note that, given Theorem~\ref{thm:3connectedScaling}, this result ensures that the scaling limit of $K(\cn^{(q(n))})$ with  distances (rescaled by $c_1q^{1/4}$ for some $c_1\in(0,\infty)$) induced by the distances in $\cn^{(q(n))}$ is the Brownian sphere. Moreover, combined with Theorem~\ref{th:LGC_simple}, it gives the joint convergence of the two last coordinates in Theorem~\ref{th:convJointe}.
\medskip

To prove this theorem, we first establish the following similar result for the 2-point function (analogous to Proposition~\ref{prop:2point}, with a dependence on the edge-lengths). 
\begin{proposition}\label{th:2pointGraph}
Consider the same setting as in the previous theorem. 
Let $o,o'$ be two random vertices of $K(\cn^{(q(n))})$. Then, there exists a constant $c>0$ (the same as in Theorem~\ref{th:scaling3connex}) such that, for any $\eps>0$: 
\begin{equation}\label{eq:2pointGraphDist}
\mathbb{P}\big( |\mathrm{dist}_{\cn^{(q(n))}}(o,o')-c\, \dgr(o,o') | \geq \eps\, q^{1/4}\big) \xrightarrow[q\to\infty]{} 0
\end{equation}

A similar result holds for the random multigraph $\mn^{(q(n))}$, when $q(n)=\tfrac{199}{316}n+O(n^{2/3})$. 
\end{proposition}

Sections~\ref{sub:coupling_two} to~\ref{sub:proof2pointmod} will be devoted to the proof of this proposition. The proof of Theorem~\ref{th:scaling3connex} will then follow along the same lines as the part proving~\eqref{eq:supVF} in the proof of Theorem~\ref{th:correspondance} and is given in Section~\ref{sub:proofTheo3connexe}.

\index{g@\textbf{3-connected cubic planar graphs}!zzabuse@\ttt{$\cc_q = K(\cn^{(q)})$}{(abuse of notation in Section 7)}}
\index{g@\textbf{3-connected cubic planar graphs}!zzabusea@$\tKq=(\cc_q,\tdist) := (\cc_q, \dist_{\cn^{(q)}})$}
In the rest of Section~\ref{sec:two_point_dep}, we always assume that $n$ and $q(n)$ are related as in Proposition~\ref{th:2pointGraph}. To simplify notations, we drop the index $n$ and simply write $q$ for $q(n)$. 
Since $q(n)=\alpha n+O(n^{2/3})$, it makes sense to write $q\rightarrow \infty$ in place of $n\rightarrow \infty$.
Similarly, with a slight abuse of notation, in all this section, we write $\cc_q$ for $K(\cn^{(q(n))})$ and $\tKq:=(\cc_q,\tdist)$ for the metric space $(K(\cn^{(q(n))}), \mathrm{dist}_{\cn^{(q(n))}})$.

\index{g@\textbf{3-connected cubic planar graphs}!zzabuseb@\ttt{$\hKq=(\cc_q,\hdist)$}{$\cc_q$ equipped with $\dstar_{\nu_*}$ (see cubic networks for $\nu_*$)}}
On the other hand, recall that $\nustar$ is the law of the pole-distance of a critical Boltzmann cubic network as defined in Section~\ref{sec:pole_dista}. We define $\hKq:=(\cc_q,\hdist)$ to be the random metric graph $K(\cn^{(q(n))})$, equipped with the first-passage percolation distance obtained from $(\ell(e))_{e})$ where $(\ell_e)$ is a collection of i.i.d random variables sampled from $\nustar$. 
It follows from Lemma~\ref{lem:nustar} that $\nustar$ has exponential tails. Since the support of $\nustar$ is by definition included in $[1,\infty)$, all the results obtained in the previous section apply to $\hKq$. In particular, by applying Proposition~\ref{prop:2point} twice (once with $\nu$ equals to the Dirac measure on $\{1\}$ and once with $\nu:=\nustar$), there exists a constant $c>0$ such that, for $o,o'$ two random vertices in $\cc_q$ and for every $\eps>0$, we have
\begin{equation}\label{eq:hK}
\mathbb{P}\big( |  \hdist(o,o')-c\, \dgr(o,o') | \geq \eps\, q^{1/4}\big) \xrightarrow[q\to\infty]{} 0
\end{equation}
We emphasize that in fact this statement serves at the very definition of the constant $c>0$ that appears both in Theorem~\ref{th:scaling3connex} and in Proposition~\ref{th:2pointGraph}.
\medskip

To prove Proposition~\ref{th:2pointGraph},  
we are going to construct a coupling between $\hKq$ and $\tKq$, so as to minimize the the number of edges whose edge-lengths differ. 
We will then show that, under this coupling, the weighted distance between two random vertices $o,o'$ differs by $o(q^{1/4})$ in probability.
%such that the proportion of edges whose edge-length differs is $o(1)$ in probability. The result will then follow directly. 

\subsection{Definition of the coupling between \texorpdfstring{$\hKq$}{hKq} and \texorpdfstring{$\tKq$}{tKq}}\label{sub:coupling_two}
To define the coupling between $\hKq$ and $\tKq$, we first introduce a handful of notations. 

\index{g@\textbf{3-connected cubic planar graphs}!zzz@coupling between $\tKq$ and $\hKq$!zzcoupla@\ttt{$\Dist$}{maps from $\N_*^m$ to the sequences $(n_1, n_2, \dots)$ in $\N$ with $\sum_i n_i = m$}}
For $s\in \Z_{\geq 0}$, let $\Sigma_s$ to be the set of sequences $(n_1,n_2,\ldots)$ of nonnegative integers, such that $\sum_{i}n_i =s$. For $\bdelta=(\delta_1,\ldots,\delta_m)$ a sequence of positive integers, we write $N_i(\bdelta)=\#\{j\in[1..m],\ \delta_j=i\}$ and $N_{\geq i}(\bdelta)=\#\{j\in[1..m],\ \delta_j\geq i\}$ for any $i\geq 1$. Next, we define $\Dist(\bdelta)$ as the mapping from $\N_*^m$ to $\Sigma_m$, defined by: 
\[
\Dist(\bdelta):=(N_1(\bdelta),N_2(\bdelta),\ldots)
\]
i.e., $\Dist(\bdelta)$ gives at each position $i\geq 1$ the multiplicity of $i$ in $\bdelta$. 

Moreover, for $i\geq 1$, we denote $\bdelta^{(i)}:=\{j\in[1..m],\ \delta_j=i\}$ the set of indices at which the sequence is equal to $i$. Fix $q\in \Z_{\geq0}$, and let $\tilde \bn:=(\tilde n_1,\tilde n_2,\ldots)$ and $\hat \bn:=(\hat n_1,\hat n_2,\ldots)$ be two sequences of $\Sigma_{3q}$. As illustrated in Figure~\ref{fig:coupling}, 
it is easy to build canonically from $\hat\bn$ and $\tilde\bn$, two sequences
$\hat \bdelta=(\hdelta_1,\ldots,\hdelta_{3q})$ and $\tilde \bdelta=(\tdelta_1,\ldots,\tdelta_{3q})$ such that, for each $i\geq 1$, we have
\begin{equation}\label{eq:couplingSeq}
|\hat \bdelta^{(i)}|=\hat n_i,\ \ \ 
|\tilde \bdelta^{(i)}|=\tilde n_i,\ \ \ 
|\hat \bdelta^{(i)}\cap \tilde \bdelta^{(i)}|=\mathrm{min}(\hat n_i,\tilde n_i).
\end{equation}

\begin{figure}
\begin{center}
\includegraphics[width=12cm]{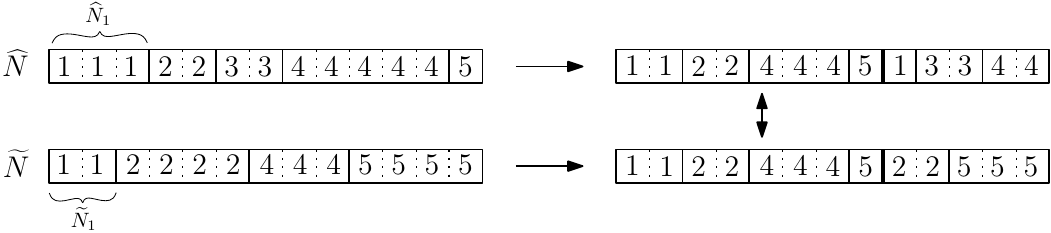}
\end{center}
\caption{The canonical rearrangement of two increasing sequences $\widehat{\bdelta},\widetilde{\bdelta}$, with $\Dist(\widehat{\bdelta})=\hat \bn$ and $\Dist(\widetilde{\bdelta})=\tilde  \bn$.}
\label{fig:coupling}
\end{figure}
\index{g@\textbf{3-connected cubic planar graphs}!zzz@coupling between $\tKq$ and $\hKq$!zzcouplb@\ttt{$\hNuq$}{distribution of $\Dist(\hdelta_1, \dots, \hdelta_{3q})$ with $(\hdelta_k)\sim \nu_*$ i.i.d.}}
For $q\geq 1$, let us now define two probability distributions on $\Sigma_{3q}$. Let $(\hdelta_1,\ldots,\hdelta_{3q})$ be a sequence of i.i.d random variables sampled from $\nustar$. We denote by $\hNuq$ the push-forward distribution on $\Sigma_{3q}$ via the mapping $\Dist(.)$, of the distribution of $(\hdelta_1,\ldots,\hdelta_{3q})$. 

Similarly, let $(\tdelta_1,\ldots,\tdelta_{3q})$ be the random vector giving the edge-lengths in $\tKq$ inherited from $\cn^q$. Then, we denote by $\tNuq$ the push-forward distribution on $\Sigma_{3q}$ via the mapping $\Dist(.)$ of the distribution of $(\tdelta_1,\ldots,\tdelta_{3q})$.
\index{g@\textbf{3-connected cubic planar graphs}!zzz@coupling between $\tKq$ and $\hKq$!zzcouplc@\ttt{$\tNuq$}{distribution of $\Dist(\tdelta_1, \dots, \tdelta_{3q})$ with $(\tdelta_k)$ the $\tdist$-edge lengths in $\tKq$}}
\medskip

\index{g@\textbf{3-connected cubic planar graphs}!zzz@coupling between $\tKq$ and $\hKq$!zzcoupld@$(\bhN,\btN)\sim \hNuq \otimes \tNuq$}
\index{g@\textbf{3-connected cubic planar graphs}!zzz@coupling between $\tKq$ and $\hKq$!zzcouple@\ttt{$\widehat \bdelta, \widetilde \bdelta$}{obtained from $\bhN, \btN$ resp.}}
\index{g@\textbf{3-connected cubic planar graphs}!zzabusea@$\tKq=(\cc_q,\tdist) := (\cc_q, \dist_{\cn^{(q)}})$}
\index{g@\textbf{3-connected cubic planar graphs}!zzabuseb@\ttt{$\hKq=(\cc_q,\hdist)$}{$\cc_q$ equipped with $\dstar_{\nu_*}$ (see cubic networks for $\nu_*$)}}
\index{g@\textbf{3-connected cubic planar graphs}!zzz@coupling between $\tKq$ and $\hKq$!zzcouplf@\ttt{$\hell(e), \tell(e)$}{length of edge $e$ in $\hKq, \tKq$ resp.}}
We can now describe the coupling between $\hKq$ and $\tKq$:
\begin{enumerate}
\item Let $\bhN=(\hN_1,\hN_2,\ldots)$ and $\btN=(\tN_1,\tN_2,\ldots)$ be independently sampled according to $\hNuq$ and $\tNuq$ respectively.
\item Construct (in the canonical way depicted in Figure~\ref{fig:coupling}) two sequences $\widehat \bdelta=(\hdelta_1,\ldots,\hdelta_{3q})$ and $\widetilde \bdelta=(\tdelta_1,\ldots,\tdelta_{3q})$, which satisfy~\eqref{eq:couplingSeq}.
\item
Draw a uniformly random $3$-connected cubic planar graph $\Kq$ with $2q$ vertices, and sample uniformly at random a permutation of its edges. List as $(e_1,\ldots,e_{3q})$ its edges in this random ordering.
\item Define $\hKq=(\Kq,(\hell(e_j)))$ and $\tKq=(\Kq,(\tell(e_j)))$ as the metric graphs built from $\Kq$, where for each $j\in [1..3q]$, $\hell(e_j)=\hdelta_j$ and $\tell(e_j)=\tdelta_j$.
\end{enumerate}
It is clear from the construction that $\hKq$ and $\tKq$ defined in this manner follow the required distribution. In the rest of this section, they are always assumed to be sampled jointly in this manner. 
\medskip

\begin{remark}\label{rem:coupling}
In the following, it will be useful to be able to ``choose an edge in $\Kq$ using some information about $\hKq$ but no information from $\tKq$'', or symmetrically. 
To make rigorous sense of this, we can slightly adapt the sampling procedure as follows to better differentiate that $\hKq$ is first drawn, and $\tKq$ subsequently (we use the notation $E^{=i}$ for the set of edges of length $i$ in a 
metric graph with integer edge-lengths):
\begin{enumerate}
\item Sample $\Kq$, $\bhN$ and $\btN$.
\item Consider a random permutation of $E(\Kq)$, and list its edges as $(e_1,\ldots,e_{3q})$ in this random ordering.
\item For each $i\in \Z_{\geq 1}$, define $E^{=i}(\hKq)$ as $\{e_{\hN_{\geq i+1}+1},\ldots,e_{\hN_{\geq i}}\}$.
\item For each $i\in \Z_{\geq 1}$, reorder uniformly at random the edges of $E^{=i}(\hKq)$ and apply the canonical construction of Figure~\ref{fig:coupling} to construct $\tKq$ following $\btN$.
\end{enumerate}
In this way, for every $i\geq 1$ we have $|E^{=i}(\hKq)\cap E^{=i}(\tKq)|=\mathrm{min}(\hN_i,\tN_i)$, and for every $m\in [\hN_{\geq i+1}+1.. \hN_{\geq i}]$ we have
$\mathbb{P}(e_m\in E^{=i}(\tKq))=\frac{\mathrm{min}(\hN_i,\tN_i)}{\hN_i}$. 

Of course, the roles played by $\hKq$ and $\tKq$ are perfectly symmetric, and we can proceed by first sampling $\tKq$ and then $\hKq$ if needs be. \dotfill
\end{remark}

\subsection{First properties of the coupling}\label{sub:propCoupling}

To quantify the number of edges which do not have the same length in both models, we start by proving the following estimate:
\begin{lemma}\label{lem:Ni}
Let $\bhN=(\hN_1,\hN_2,\ldots)$ and $\btN=(\tN_1,\tN_2,\ldots)$ be independently sampled from $\hNuq$ and from $\tNuq$ respectively.  
Then there exist positive constants $a$ and $A$ such that a.a.s. as $q\to\infty$, the following event $\kE_q$ holds:
\begin{itemize}
\item
for every positive integer $i\leq a\log(q)$, we have $\mathrm{min}(\hN_i,\tN_i)\geq q^{3/4}$ and $|\hN_i-\tN_i|\leq q^{2/3}$. 
\item
for all $i\geq A\log(q)$, we have $\hN_i=\tN_i=0$.
\end{itemize} 
\end{lemma}
\begin{proof}
Let $p_i:=\nustar(i)$. We can obtain a lower bound for $p_i$ by noticing that any cubic network that is a series of $i$ \netls has pole-distance $i+1$. Hence:
\[
p_i\geq \frac{L(\rho)^{i-1}}{D(\rho)}\geq \frac1{D(\rho)}e^{-\kappa i}, \text{ where }\kappa=\ln(1/L(\rho))>0,\]
and $L(\rho)$ is the generating series of $L$-networks evaluated at its radius of convergence. 

Let $a=\frac1{5\kappa}$.  
Then, for $i\leq a\log(q)$, we have $p_i\geq\frac1{D(\rho)}q^{-1/5}$.
On the other hand, by definition, $\hN_i$ follows a $\mathrm{Binomial}(3q,p_i)$ distribution, so that $\E(\hN_i)=3q\,p_i\geq \frac{3}{D(\rho)}q^{4/5}$. 
By Hoeffding's inequality, we have that, for any $x>0$:
\[\mathbb{P}(|\hN_i-3q\,p_i|\geq xq^{1/2})\leq 2\exp(-2x^2/3),\] so that $\mathbb{P}(|\hN_i-3q\,p_i|\geq q^{2/3}/2)\leq 2\exp(-q^{1/3}/6)$. By Lemma~\ref{lem:size_core}, $\tN_i$ has the same law as $\hN_i$ conditioned on an event of probability $\Theta(q^{-2/3})$, which implies that
\[
   \mathbb{P}(|\tN_i-3q\,p_i|\geq q^{2/3}/2)=O(q^{2/3}\exp(-q^{1/3}/6)). 
\]
Hence, 
\[
    \mathbb{P}\left(\forall i\leq a\log(q), |\hN_i-3q\,p_i|\leq q^{2/3}/2\text{ and }|\tN_i-3q\,p_i|\leq q^{2/3}/2\right)\xrightarrow[q\rightarrow \infty]{}1,
\]
which gives the first property (using the lower bound on $3q\,p_i$).

Regarding the second property,  we proved in Lemma~\ref{lem:nustar} that $\nustar$ has exponential tail. Hence, if $(\hdelta_1,\ldots,\hdelta_{3q})$ 
is a sequence of i.i.d. random variables sampled from $\nustar$, then there exists $b>0$, such that $\mathbb{P}(\mathrm{max}(\hdelta_1,\ldots,\hdelta_{3q})\geq i)=O(q\exp(-bi))$. This is equivalent to the fact that $\mathbb{P}(\hN_j\neq 0\ \mathrm{for\ some}\ j\geq i)=O(q\exp(-bi))$.  
Using again Lemma~\ref{lem:size_core}, we can derive from this property, the fact that $\mathbb{P}(\tN_j\neq 0\ \mathrm{for\ some}\ j\geq i)=O(q^{5/2}\exp(-bi))$. Hence, setting $A=3/b$, this gives the desired result.
\end{proof}

\begin{corollary}\label{coro:coupling}
Conditionally on the event $\kE_q$ defined in the previous lemma, for any $e\in E(\Kq)$ such that $\hell(e)\leq a\log(q)$, and chosen without using any information on $\tKq$, we have:
\[\mathbb{P}(\hell(e)\neq\tell(e))\leq q^{-1/12}.\]
Similarly, for any edge $e\in E(\Kq)$, such that $\tell(e)\leq a\log(q)$, 
  and chosen without using any information on $\hKq$, we have $\mathbb{P}(\hell(e)\neq\tell(e))\leq q^{-1/12}$. 
\end{corollary}
\begin{proof}
By construction of the coupling (or rather of its variant described in Remark~\ref{rem:coupling}), in the first case, we have that: 
\begin{equation}\label{eq:probaDiffH}
    \mathbb{P}(\tell(e)\neq\hell(e)|\hell(e)=i)=\frac{\max(\tN_i-\hN_i,0)}{\tN_i}, \text{ for any }i\in \Z_{>0}.
\end{equation}    
Summing over the values of $i$ in $\{1,..,\lfloor a\log(q)\rfloor\}$, and using the bound for $\tN_i$ and $\hN_i$, this proves the result.

The other case is treated in exactly the same manner.
\end{proof}

\subsection{Truncation of edge-lengths}\label{sub:truncation}

\index{g@\textbf{3-connected cubic planar graphs}!zzz@coupling between $\tKq$ and $\hKq$!zzcouplg@$\ell^{(k)}(e) := \min(\ell(e),k)$}
\index{g@\textbf{3-connected cubic planar graphs}!zzz@coupling between $\tKq$ and $\hKq$!zzcouplh@\ttt{$L^{(k)}(\gamma)$}{$\ell^{(k)}$-length of a path $\gamma$}}
\index{g@\textbf{3-connected cubic planar graphs}!zzz@coupling between $\tKq$ and $\hKq$!zzcoupli@\ttt{$\hdistk, \tdistk$}{first-passage percolation distance from $(\hell^{(k)}(e))$, resp. $(\tell^{(k)}(e))$}}
\index{g@\textbf{3-connected cubic planar graphs}!zzz@coupling between $\tKq$ and $\hKq$!zzcouplj@$\hKqk = (\Kq,\hdistk),\tKqk = (\Kq,\tdistk)$}
Another important ingredient of our proof is the truncation of the edge-lengths. 
Fix $k\geq 1$, then for any metric graph $G=(V,E,(\ell(e)_{e\in E})$, we write $\ellk(e):=\min(\ell(e),k)$. For a path $\gamma=(e_1,\ldots,e_m)$ in $G$, the truncated edge-length of $\gamma$ is defined as $\ellk(e_1)+\cdots+\ellk(e_m)$.

Let us apply this notion to $\hKq$ and $\tKq$. For an edge $e$ of $\Kq$, we denote by $\hellk(e)$ (resp. $\tellk(e)$) its truncated edge-length in $\hKq$ (resp. $\tKq$). For $\gamma$ a path in $\Kq$, its truncated edge-length is denoted by $\hLk(\gamma)$ (resp. $\tLk(\gamma)$) in $\hKq$ (resp. $\tKq$). Furthermore, for two vertices $o,o'\in\Kq$, we denote by $\hdistk(o,o')$ (resp.  $\tdistk(o,o')$) the minimum of $\hLk(\gamma)$ (resp. of $\tLk(\gamma)$) over 
all paths connecting $o$ and $o'$.

Finally, we denote by $\hKqk$ (resp. $\tKqk$) the metric graphs $(\Kq,\hdistk)$ (resp. $(\Kq,\tdistk)$). Applying the results of Section~\ref{sec:modDistances} to $\hKqk$ instead of $\hKq$, we obtain 
 that there exists a constant $\ck$ such that, for every $\eps>0$, we have:
\begin{equation}\label{eq:hKk}
\mathbb{P}\big( |  \hdistk(o,o')-\ck\, \dgr(o,o') | \geq \eps\, q^{1/4}\big) \xrightarrow[q\to\infty]{} 0,
\end{equation}
where $o$ and $o'$ are two random vertices of $\Kq$ (picked independently from anything else).
 Note that $\ck$ is increasing with $k$ and is bounded by $c$. 
 
The purpose of this section is now to prove the analogous statement on $\tKqk$:

\begin{lemma}\label{lem:sameLk}
Let $k\geq 1$. Then, with the same constant $\ck$ as in~\eqref{eq:hKk}, we have for every $\eps>0$
\begin{equation}\label{eq:tKk}
\mathbb{P}\big( |  \tdistk(o,o')-\ck\, \dgr(o,o') | \geq \eps\, q^{1/4}\big) \xrightarrow[q\to\infty]{} 0,
\end{equation}
where $o$ and $o'$ are two random vertices of $\Kq$ (picked independently from anything else).
\end{lemma} 
\begin{proof}
Set $\hN_{\geq k}:=\sum_{i\geq k}\hN_i$ and $\tN_{\geq k}:=\sum_{i\geq k}\tN_i$, 
and define $\kE_q^{(k)}$ as the event under which:
\begin{equation}\label{eq:eventEq}
\mathrm{min}(\hN_i,\tN_i)\geq q^{3/4}\quad\text{and}\quad |\hN_i-\tN_i|\leq q^{2/3} \text{ for }i<k, \quad \text{and}\quad \mathrm{min}(\hN_{\geq k},\tN_{\geq k})\geq q^{3/4}.
\end{equation}
This event holds a.a.s. by Lemma~\ref{lem:Ni}.  

Let $e\in E(\Kq)$, chosen without using any information on $\tKqk$. Conditionally on $\kE_q^{(k)}$, by the same reasoning as in Corollary~\ref{coro:coupling}, we have:
\begin{equation}\label{eq:difftruncA}
    \mathbb{P}(\tellk(e)\neq\hellk(e)|\hellk(e) < k)\leq \max_{1\leq i< k}\frac{\max(\tN_i-\hN_i,0)}{\tN_i}\leq q^{-1/12}.
\end{equation}    
On the other hand, let $N_k'$ be the number of edges 
having length at least $k$ in $\hKq$ and smaller than $k$ in $\tKq$. 
By construction of the coupling between $\hKq$ and $\tKq$, 
$N_k'\leq \sum_{i=1}^{k-1}\mathrm{max}(\tN_i-\hN_i,0)$. Hence, if $\kE_q^{(k)}$ holds, $N_k'\leq k\,q^{2/3}$. Therefore, conditionally on $\kE_q^{(k)}$, we have:
\begin{equation}\label{eq:difftruncB}
    \mathbb{P}(\tellk(e)\neq\hellk(e)|\hellk(e) = k)=\frac{N_k'}{\hN_{\geq k}}\leq kq^{-1/12}.
\end{equation}  

\index{g@\textbf{3-connected cubic planar graphs}!zzz@coupling between $\tKq$ and $\hKq$!zzcouplk@\ttt{$\gamma^{(k)}$}{``canonical'' $\mathrm{d}^{(k)}$-shortest path between $o$ and $o'$}}
Now, we define $\hgammak$ as a canonically chosen path among the paths $\gamma$ connecting $o$ to $o'$ such that $\hLk(\gamma)= \hdistk(o,o')$ (e.g. $\hgammak$ is the lexicographically smallest, for the sequence of labels of the visited vertices). For any edge of $\hgammak$, we can apply~\eqref{eq:difftruncA} or~\eqref{eq:difftruncB}.
Moreover, by definition of truncation, for any edge $e$, we clearly have $|\hellk(e)-\tellk(e)|\leq k$. 
Hence, conditionally on $\kE_q^{(k)}$, and in view of~\eqref{eq:difftruncA} and~\eqref{eq:difftruncB}, we have 
\[
\E\left(|\hLk(\hgammak)-\tLk(\hgammak)|\,\Big|\,|\hgammak|=s\right)\leq k^2q^{-1/12}s.
\]
Therefore, conditionally on $\kE_q^{(k)}$, and by Markov's inequality, 
\[
\mathbb{P}\left(|\hLk(\hgammak)-\tLk(\hgammak)|\geq \eps q^{1/4}\,\Big|\,|\hgammak|\leq s\right)\leq \frac{k^2}{\eps}q^{-1/3}s.
\]
It follows directly from Theorem~\ref{thm:3connectedScaling} that a.a.s. $\dgr(o,o')\leq q^{1/4}\log(q)$. Since $k\,\dgr(o,o')\geq \hdistk(o,o')\geq |\hgammak|$, we also have a.a.s.  $|\hgammak|\leq k\,q^{1/4}\log(q)$, so that: 
\[
\mathbb{P}\left(|\hLk(\hgammak)-\tLk(\hgammak)|\geq \eps q^{1/4}\right)\xrightarrow[q\rightarrow \infty]{} 0.
\] 
For any $\eps>0$, the event $\hdistk(o,o')\leq \ck\dgr(o,o')+\eps q^{1/4}$ holds a.a.s. according to~\eqref{eq:hKk}. Since $\hLk(\hgammak)=\hdistk(o,o')$ and $\tdistk(o,o')\leq \tLk(\hgammak)$, we conclude that, for any $\eps>0$, 
\begin{equation}
\mathbb{P}\left(\tdistk(o,o')\leq \ck\dgr(o,o')+\eps q^{1/4}\right)\xrightarrow[q\rightarrow \infty]{} 1.
\end{equation}

Symmetrically, applying the argument to $\tgammak$, we obtain that, for every $\eps>0$, we have a.a.s. $|\hLk(\tgammak)-\tLk(\tgammak)|\leq \eps q^{1/4}$. 
For any $\eps>0$, the event $\hdistk(o,o')\geq \ck\dgr(o,o')-\eps q^{1/4}$ holds a.a.s. according to~\eqref{eq:hKk}. Since $\tLk(\tgammak)=\tdistk(o,o')$ and $\hdistk(o,o')\leq \hLk(\tgammak)$, we get that, for any $\eps>0$, 
\begin{equation}
\mathbb{P}\left(\tdistk(o,o')\geq \ck\dgr(o,o')-\eps q^{1/4}\right)\xrightarrow[q\rightarrow \infty]{} 1,
\end{equation}
which concludes the proof.
\end{proof}

\subsection{Lengths of truncated geodesic paths in both models}\label{sub:lengthTrunc}
We now state and prove additional estimates in view of completing the proof of Proposition~\ref{th:2pointGraph}.

\begin{lemma}\label{lem:thL}
%\ctom{Pas défini les $\hgammak$, à moins que ce soient les mêmes que dans la preuve du Lemme 40. Mais dans ce cas, pourquoi les redéfinir dans le Lemme 43 ?}
%\ceri{ $\hgammak$ est défini un peu plus haut, du coup j'ai ajouté "recall" dans le Lemme 43}

Let $k\geq 1$. Then, for every $\eps>0$, we have
\[
\mathbb{P}\big( |  \tL(\hgammak)-\hL(\hgammak) | \geq \eps\, \dgr(o,o')\big) \xrightarrow[q\to\infty]{} 0, 
\]
where $o$ and $o'$ are two random vertices for $\Kq$ (picked independently from anything else), and where we recall that $\hgammak$ is a canonical geodesic path in $\hKqk$ between them.
\end{lemma}
\begin{proof}
In all this proof, we work conditionally on the a.a.s. event $\kE_q$ of Lemma~\ref{lem:Ni}, and consider $q$ large enough such that $a\log(q)>k$. For $e$ an edge chosen using solely 
information given by $\hKqk$ and for $i\geq k$, we have: 
\[
\mathbb{P}(\hell(e)\geq i\,|\, \hellk(e)= k)=\nustar([i,\infty))/\nustar([k,\infty)).
\]

Hence, by Lemma~\ref{lem:nustar}, letting $\kappa=a_\star/\nustar([k,\infty))$, we have $\mathbb{P}(\hell(e)\geq a\log(q))\leq \kappa\,q^{-ab_\star}$. Since 
\[
\E\big(|\hell(e)-\tell(e)|\big)\leq A\,\log(q)\left( \mathbb{P}\Big(\tell(e)\neq\hell(e)|\hell(e)<a \log q\Big)+\mathbb{P}\Big(\hell(e)\geq a \log q\Big)\right),
\] 
by Corollary~\ref{coro:coupling}, we get that: 
\[
\E\big(|\hell(e)-\tell(e)|\big)\leq (q^{-1/12}+\kappa\, q^{-ab_\star})A\log(q). \]
We further work conditionally on the a.a.s. event $\kE_q'$ under which $q^{1/4}/\log(q)\leq \dgr(o,o')\leq q^{1/4}\log(q)$. Since $|\hgammak|\leq k\cdot\dgr(o,o')$, we have
\[
\E\big( |  \tL(\hgammak)-\hL(\hgammak) |\big) \leq k(q^{-1/12}+\kappa\, q^{-ab_\star})A\,q^{1/4}\log(q)^2=O(q^{1/4-\xi}),
\]
where $\xi=\mathrm{min}(1/12,ab_\star)/2$. Hence, by Markov's inequality we have
\begin{align*}
\mathbb{P}\Big( |  \tL(\hgammak)-\hL(\hgammak) | \geq \eps\, \dgr(o,o')\Big)&\leq \mathbb{P}\Big( |  \tL(\hgammak)-\hL(\hgammak) | \geq \eps\, q^{1/4}/\log(q)\Big)+\mathbb{P}({}^c(\kE_q\cap \kE_q'))\\
& =O\big(\epsilon^{-1}\log(q)\,q^{-\xi}\big)+o(1)=o(1).
\end{align*}
\end{proof}

\begin{lemma}\label{lem:Nprime}
Let $X$ be a random variable drawn under $\nustar$. There exist positive constants $a',b'$, and an infinite subset $\N'\subset\N$ such that, for every $k\in\N'$ and every $i\geq 0$, we have
\[
\mathbb{P}(X\geq k+i\ |\ X\geq k)\leq a'\,\exp(-b'\,i).
\]
\end{lemma}
\begin{proof}
Let $p_i=\mathbb{P}(X=i)$. We have seen in the proof of Lemma~\ref{lem:Ni} that $p_i>0$ for any $i\geq 0$ and that $p_i=O(\exp(-bi))$ for some $b>0$. 
Set $p_i':=p_i\exp(bi/2)$. Then $p_i'\to 0$ as $i\to\infty$. Hence, we can define a subset $\N'$ of integers such that, for every $k\in\N'$, 
 $p_k'=\mathrm{max}(p_i',\ i\geq k)$ holds.  Hence, for $k\in\N'$ and $i\geq k$, we have $p_i\leq p_k\exp(-b(i-k)/2)$. This gives
 \[
\mathbb{P}(X\geq k+i\ |\ X\geq k)\leq \frac{\sum_{j\geq i}p_k\,e^{-bj/2}}{\mathbb{P}(\nustar\geq k)}\leq \sum_{j\geq i}e^{-bj/2}=\frac{1}{1-e^{-b/2}}e^{-bi/2}.
\] 
\end{proof}

\begin{lemma}\label{lem:LLk}
There exists a positive constant $\ha$ such that, for every $k\in\N'$, 
\[
\mathbb{P}\Big(\,\hL(\hgammak)\leq \hdistk(o,o')+\frac{\ha}{k}\dgr(o,o')\Big) \xrightarrow[q\to\infty]{} 1,
\]
where $o$ and $o'$ are two random vertices for $\Kq$ (picked independently from anything else), and where we recall that $\hgammak$ is a canonical geodesic path in $\hKqk$ between them.

Moreover, we have $\lim_{k\to\infty}\ck=c$. 
\end{lemma}
\begin{proof}
Let $\Delta=\hL(\hgammak)-\hLk(\hgammak)$. Then $\Delta=\sum_{e\in E_k}(\ell(e)-k)$, where $E_k$ is the set of edges $e\in\hgammak$ with $\ellk(e)= k$. 
Since $\hgammak$ is selected based solely on $\hKqk$, $\Delta$ is the sum of $|E_k|$ i.i.d. copies of the random variable $\Xk$, whose distribution is given by: 
\[
\mathbb{P}(\Xk\geq i)=\frac{\nustar([k+i,\infty))}{\nustar([k,\infty))}.
\] 
Lemma~\ref{lem:Nprime} gives $\E(\Xk)\leq \frac{a'}{1-e^{-b'}}$.  
For $m\geq 1$, let $S_m$ be the sum of $m$ i.i.d. copies of $\Xk$. 
Given the universal exponential bound on $\Xk$ in Lemma~\ref{lem:Nprime},  
Chernoff's bound ensures that there exist positive constants $a'',b''$ (not depending on $k$) such that
\[
\mathbb{P}\Big(S_m\geq \frac{2a'}{1-e^{-b'}}m\Big) \leq a''\,e^{-b''m},
\]
 and thus $\mathbb{P}\Big(S_{m'}\geq \frac{2a'}{1-e^{-b'}}m\Big) \leq a''\,e^{-b''m}$ whenever $m'\leq m$.

Next, let us define the event $\kH_q$ as: 
\[
    \kH_q:=\left\{\hLk(\hgammak)\leq 2c\,\dgr(o,o') \right\}\,\cap\, \left\{\dgr(o,o')\geq \frac{q^{1/6}k}{2c}\right\}.
\] 
By~\eqref{eq:hKk} and Theorem~\ref{thm:3connectedScaling}, we know that $\mathbb{P}(\kH_q)\to 1$ as $q\to \infty$.
If $\kH_q$ holds, we have $|E_k|\leq \frac{2c}{k}\dgr(o,o')$, so that:
 \[
 \mathbb{P}\left(\Big\{\Delta\geq  \frac{2a'}{1-e^{-b'}}\frac{2c}{k}\dgr(o,o') \Big\}\cap \kH_q \right)  \leq a''\, e^{-b''\,q^{1/6}}=o(1),
 \]
 so that the first statement is proved (with $\hat{a}=\frac{2a'}{1-e^{-b'}}2c$). 
 
It ensures that  we have a.a.s. $\hdist(o,o')\leq \hdistk(o,o')+\frac{\hat{a}}{k}\dgr(o,o')$.  Hence, using~\eqref{eq:hKk}, for any $\eps>0$ we have a.a.s. 
$\hdist(o,o')\leq (\ck+\frac{\hat{a}}{k}+\eps)\,\dgr(o,o')$. This implies that $\ck+\frac{\hat{a}}{k}\geq c$. Since $\ck\leq c$, this gives $\lim_{k\to\infty}\ck=c$.
\end{proof}

\subsection{Proof of Proposition~\ref{th:2pointGraph}}\label{sub:proof2pointmod}
We start by establishing the lower bound for $\tdist(o,o')$ given in~\eqref{eq:2pointGraphDist} of Proposition~\ref{th:2pointGraph}. Fix $\epsilon>0$ and $\eta>0$. By Theorem~\ref{thm:3connectedScaling}, we know that there exists $b>0$ such that for $q$ large enough: 
\[
\mathbb{P}(\dgr(o,o')\leq b\,q^{1/4})\geq 1-\eta/2.
\]
On the other hand, by Lemma~\ref{lem:sameLk}, we have for $q$ large enough: 
\[\mathbb{P}\big(\tdistk(o,o')\geq  \ck\dgr(o,o')-\frac{\eps}{2}q^{1/4}\big)\geq 1-\eta/2.\] 
Then, Lemma~\ref{lem:LLk} grants the existence of $k$ such that $\ck \geq c-\eps/(2b)$. 
Combined with the two preceding inequalities and since $\tdistk(o,o')\leq \tdist(o,o'))$ it yields for $q$ large enough: 
 \[\mathbb{P}\big(\tdist(o,o')\geq c\,\dgr(o,o')-\eps q^{1/4}\big)\geq 1-\eta.\]

Regarding the upper bound, let $k\in\N'$ be large enough such that $\frac{\ha}{k}\leq \frac{\eps}{4b}$. By Lemma~\ref{lem:LLk}, there exists $q_1$ such that, for 
$q\geq q_1$, we have 
\[
\mathbb{P}\Big(\,\hL(\hgammak)\leq \hdist(o,o')+\frac{\eps}{4b}\dgr(o,o')\Big)\geq 1-\eta/8.
\] 
By Lemma~\ref{lem:thL}, there exists $q_2$ such that, for $q\geq q_2$,
\[
\mathbb{P}\Big(\,\tL(\hgammak)\leq \hL(\hgammak)+\frac{\eps}{4b}\dgr(o,o')\Big)\geq 1-\eta/8.
\] 
Hence, for $q\geq\mathrm{max}(q_1,q_2)$, we have
\[
\mathbb{P}\Big(\tL(\hgammak)\leq \hdist(o,o')+\frac{\eps}{2b}\dgr(o,o')  \Big)\geq 1-\eta/4.
\]
Since $\tdist(o,o')\leq \tL(\hgammak)$, we have $\mathbb{P}\Big(\tdist(o,o')\leq \hdist(o,o')+\frac{\eps}{2b}\dgr(o,o')  \Big)\geq 1-\eta/4$. 
Hence, $\mathbb{P}\Big(\tdist(o,o')\leq \hdist(o,o')+\frac{\eps}{2}q^{1/4}  \Big)\geq 1-3\eta/4$.  

Finally, by~\eqref{eq:hK}, there exists $q_3$ such that, for $q\geq q_3$, we have 
\[
\mathbb{P}\Big(\hdist(o,o')\leq c\,\dgr(o,o')+\frac{\eps}{2}q^{1/4}\Big)\geq 1-\eta/4.
\] 
Hence, for $q\geq\mathrm{max}(q_1,q_2,q_3)$, we have $\mathbb{P}\big(\tdist(o,o')\leq c\,\dgr(o,o')+\eps q^{1/4}\big)\geq 1-\eta$. This concludes the proof.

\subsection{Proof of Theorem~\ref{th:scaling3connex}}\label{sub:proofTheo3connexe}
Given Proposition~\ref{th:2pointGraph}, the proof of Theorem~\ref{th:scaling3connex} follows the exact same lines as in the proof of~\eqref{eq:supVF}. The only missing point is a result analogous to Proposition~\ref{prop:lemmm28}, which we now state and prove:
\begin{lemma}\label{lem:lemm28bis}
Let $K$ be any constant larger than $\E(\nustar)$. For $\eps\in(0,1/4)$, define $\tcEq$ as the event where the bound 
\[
\tdist(x,y)\leq K \,\dgr(x,y)+q^{\eps}
\]
holds for every pair  of vertices $x,y$ in $\Kq$. Then $\mathbb{P}(\tcEq)\to 1$ as $q\to\infty$.
\end{lemma}
\begin{proof}
We start by proving a similar result but for $\hKq$ rather than $\tKq$. Let $\hcEq$ be the event where the bound
\[
\hdist(x,y)\leq K \,\dgr(x,y)+q^{\eps}
\]
holds for every pair  of vertices $x,y$ in $\Kq$. Let $x,y$ be two fixed 
vertices in $\Kq$. 
Since $\nustar$ has exponential tail, the Chernoff bound ensures that there exist positive constants $a',b'$ (not depending on $x,y$) such that 
\[
\mathbb{P}(\hdist(x,y)>K \,\dgr(x,y))\leq a'\exp(-b'\,\dgr(x,y)).
\] 
In particular, if $\dgr(x,y)>q^{\eps/2}$, this implies that $\mathbb{P}\Big(\hdist(x,y)>K \,\dgr(x,y)\Big)\leq a'\exp(-b'\,q^{\eps/2})$. 

We now deal with the case where $\dgr(x,y)\leq q^{\eps/2}$.
Let $\wmax$ be the maximal weight over the edges in $\hKq$. The union bound and the fact that $\nustar$ has exponential tail ensure that (with $a_\star,b_\star$ the constants in Lemma~\ref{lem:nustar})
\[
\mathbb{P}(\wmax>q^{\eps/2})\leq 3q\cdot a_\star\cdot\exp(-b_\star\,q^{\eps/2}). 
\] 
Hence, if $\dgr(x,y)\leq q^{\eps/2}$, we have
\[
\mathbb{P}(\hdist(x,y)>q^{\eps})\leq 3q\cdot a_\star\cdot\exp(-b_\star\,q^{\eps/2}) 
\]
Letting $b''$ be a positive constant smaller than $b_\star$ and $b'$, this guarantees the existence of a constant $a''$ such that, for every given pair $x,y$ of vertices in $\Kq$,
\[
\mathbb{P}(\hdist(x,y)>K \,\dgr(x,y)+q^{\eps})\leq a''\exp(-b''\,q^{\eps/2}).
\] 
By the union bound, this gives $\mathbb{P}(\neg{\hcEq})\leq q^2\,a''\exp(-b''\,q^{\eps/2})$. Finally, by Lemma~\ref{lem:size_core}, we obtain that:  
$\mathbb{P}(\neg{\tcEq})=O(q^{2/3}\mathbb{P}(\neg{\hcEq}))=O(q^{8/3}\,a''\exp(-b''\,q^{\eps/2}))=o(1)$. 
\end{proof}
Given the preceding discussion, this concludes the proof of Theorem~\ref{th:scaling3connex}.

\section{Gromov--Hausdorff--Prokhorov distance between a random connected cubic planar graph and its 3-connected core}\label{sec:projection}

The aim of this section is to establish the following result:
\begin{proposition}\label{prop:LY}
Let $(q(n),n\geq 1)$ be such that $q(n)=\alpha n + O(n^{2/3})$ (where $\alpha\approx 0.85$ is defined in Proposition~\ref{prop:core}). 
Let $\cn^{(q(n))}$ be a uniformly random element of $\kC_n^{(q(n))}$. 
Write respectively $\mathrm{dist}_{\cn^{(q(n))}}$ and $\dgr$ for the distance on $K(\cn^{(q(n))})$ induced by the graph distance on $\cn^{(q(n))}$ and for the graph distance on $K(\cn^{(q(n))})$.
Then 
\[\dghp \left(\Big(\cn^{(q(n))},\frac1{n^{1/4}}\dgr,\mu_{\cn^{(q(n))}}\Big),\Big(K(\cn^{(q(n))}),\frac1{n^{1/4}}\dist_{\cn^{(q(n))}},\mu_{K(\cn^{(q(n))})}\Big)\right)\xrightarrow[n\to \infty]{(p)} 0,
\]
where, as usual, $\mu_{\cn^{(q(n))}}$ and $\mu_{K(\cn^{(q(n))})}$ denote respectively the uniform distribution on $V(\cn^{(q(n))})$ and on $V(K(\cn^{(q(n))}))$.

A similar results holds for $\mn^{(q(n))}$, with $q(n)=\tfrac{199}{316} n + O(n^{2/3})$. 
\end{proposition}

By Theorem~\ref{theo:largest3comp} of Section~\ref{sec:largest_3comp}, the size of the $3$-connected core of a random cubic planar graph is $\alpha n+O(n^{2/3})$ with high probability. We established in Section~\ref{sec:two_point_dep} (see the discussion below Theorem~\ref{th:scaling3connex}) that $(K(\cnq),\frac1{q^{1/4}}\dist_{\cnq},\mu_{K(\cnq)})$ converges to the Brownian sphere. Hence Proposition~\ref{prop:LY} directly implies
the joint convergence of $(K(\cn),\frac1{n^{1/4}}\dist_{\cn},\mu_{K(\cn)})$ and of $(\cn,\frac1{n^{1/4}}\dgr,\mu_{\cn})$, which corresponds to the joint convergence of the first and second coordinates stated in Theorem~\ref{th:convJointe}. 
\medskip

As in the previous section, in all this section $q,n$ are related by $q=\alpha n+O(n^{2/3})$. To lighten notation, we drop the explicit dependency of $q$ in $n$ and write $q$ instead of $q(n)$.

Roughly speaking, this result relies on two main steps. First, we have to prove that the Gromov--Hausdorff distance between these two measured metric spaces is small, or in other terms, that any vertex of $\cnq$ is at distance $o(n^{1/4})$ from a vertex of $K(\cnq)$. Then, we need to prove that the Levy--Prokhorov distance between both spaces is small. Roughly speaking, it means that the uniform measure on $K(\cnq)$ and the projection of the masses of the uniform measure on $\cnq$ into its 3-connected core are close for the $\dProk$ distance.

Our approach follows closely the strategy developed in an article by Addario--Berry and Wen~\cite{addario2017joint}, dealing with the joint convergence of a quadrangulation and its simple core. We give the arguments only for graphs, but the exact same approach can be developed for multigraphs.
\bigskip

Following~\cite{addario2017joint}, we first  define  a measure on $V(K(\mathrm{C}_n^{(q)}))$ that reflects a ``projection of masses''.  
%For a fixed $\mathrm{C}_n^{(q)}\in \kC_n^{(q)}$, we define the measure $\mu_{K(\mathrm{C}_n^{(q)})}'$ on $V(K(\mathrm{C}_n^{(q)}))$ as follows. 
Let $v_1,\ldots,v_{2q}$ be the vertices of $K(\mathrm{C}_n^{(q)})$, ordered by increasing label. For $i\in[1..2q]$, let $s_i$ be the total number of vertices of (non-pole) vertices
in the attached networks at the $3$ edges incident to $v_i$, and let $n_i=1+s_i/2$. Then the measure $\mu_{K(\mathrm{C}_n^{(q)})}'$ on $V(\mathrm{C}_n^{(q)})$ is defined by: 
\[
\mu_{K(\mathrm{C}_n^{(q)})}'(v_i)=\frac{n_i}{2n}, \text{ for any }i\in [1..2q].
\]
For a vector $\mathbf{x}=(x_1,\ldots,x_{2q})\in [0,\infty)^{2q}$, we write $|\mathbf{x}|_p=\big( \sum_{i=1}^{2q}x_i^p \big)^{1/p}$, for $p\geq 1$.
Note that, with the notations from above and setting $\mathbf{n}=(n_1,\ldots,n_{2q})$, we have $|\mathbf{n}|_1=2n$. 

%\ctom{J'ai utilisé la notation $\dProk$ dans l'intro. Uniformiser.}
\begin{lemma}\label{lem:dPdistK}
Let $\cnq$ be as in Proposition~\ref{prop:LY}.
Write $\dProk$ for the L\'evy-Prokhorov distance on the metric space $(K(\cnq),\frac1{n^{1/4}}\dgr)$. Then we have 
\[
\dProk\left(\mu_{K(\cnq)},\mu_{K(\cnq)}'\right)\xrightarrow[q\to \infty]{(p)} 0.
\]
\end{lemma}
\begin{proof}
Let $\mathbf{n}=(n_1,\ldots,n_{2q})$ be defined from $\cnq$ as above. Lemma~5.3 in~\cite{addario2017joint}\footnote{In~\cite{addario2017joint}, it is assumed that $|\mathbf{n}|_2$ is deterministic, which is not the case here. However, we can apply their lemma to the normalized vector $\tilde{\mathbf{n}}=\mathbf{n}/|\mathbf{n}|_2$, which gives the desired result.} implies that for any vertex-subset $W\subset V(K(\cnq))$, and for any $t>0$, we have
\begin{equation}
\mathbb{P}\left( | \mu_{K(\cnq)}(W) - \mu_{K(\cnq)}'(W) | > \frac{2t|\mathbf{n}|_2}{|\mathbf{n}|_1}+\frac1{2q+1} \right)   \leq 4\exp(-2t^2).
\end{equation}
Here, $|\mathbf{n}|_1=2n$, and by Lemma~\ref{lem:cnq} we know that $\mathrm{max}(n_1,\ldots,n_{2q})\leq n^{3/4}$  a.a.s., which implies $|\mathbf{n}|_2\leq 2\,n^{7/8}$ (since
$\sum_{i=1}^{2q}n_i^2\leq \mathrm{max}(n_1,\ldots,n_{2q})\cdot|\mathbf{n}|_1$). 

Taking $t=n^{1/16}$ and letting $\eps_q=2n^{-1/16}+\frac1{2q+1}=o(1)$, 
we thus have a.a.s. 
$ | \mu_{K(\cnq)}(W) - \mu_{K(\cnq)}'(W) |\leq \eps_q$ for all $W\subseteq V(K(\cnq))$.  

We can then proceed with the exact same arguments as in Corollary 6.2 in~\cite{addario2017joint} (compactness of the Brownian sphere $\map_{\infty}$, and the fact that
 $(K(\cnq),\frac1{q^{1/4}}\dgr,\mu_{K(\cnq)})$ converges toward $\map_{\infty}$ for the GHP topology, as we have already established in Theorem~\ref{thm:3connectedScaling}). 
\end{proof}

\begin{lemma}\label{lem:dPdistC1}
Write $\dtProk$ for the L\'evy-Prokhorov distance on the metric space $(K(\cnq),\frac 1 {n^{1/4}}\dist_{\cnq})$. Then, we have:
\[
\dtProk(\mu_{K(\cnq)},\mu_{K(\cnq)}') \xrightarrow[q\to \infty]{(p)} 0.
\]
\end{lemma}
\begin{proof}
We use the characterization of the L\'evy-Prokhorov distance in terms of couplings.  
Let $V_q=V(K(\cnq))$. Let $\eps>0$ and $\eta>0$.  Our aim is to show that, for $q$ large enough,  with probability at least $1-\eta$ there exists a coupling $\nu$ between 
$\mu_{K(\cnq)}$ and $\mu_{K(\cnq)}'$ such that, for the distance $\frac1{n^{1/4}}\dist_{\cnq}$, we have, with the notation of Section~\ref{sec:defGHP}, 
\[
\nu\Big( (V_q\times V_q)_\eps \Big) \geq 1-\eps.
\] 

Let $\beta=2\E(\nustar)$, and let $\bar{\eps}=\eps/(\beta+1)$. 
By Lemma~\ref{lem:dPdistK} there exists $q_1$ such that, for $q\geq q_1$, with probability at least $1-\eta/2$, there holds the event $\kF_q$ that  
there exists a coupling $\nu$ between 
$\mu_{K(\cnq)}$ and $\mu_{K(\cnq)}'$ such that, for the distance $\frac1{n^{1/4}}\dgr$, we have 
\[
\nu\Big( (V_q\times V_q)_{\bar{\eps}} \Big) \geq 1-\bar{\eps}.
\] 

Let $\kE_q$ be the event that $\dist_{\cnq}(x,y)\leq \beta \dgr(x,y)+q^{1/6}$ for  $x,y\in V_q$. 
Then Lemma~\ref{lem:lemm28bis} guarantees that there exists $q_0$ such that $\mathbb{P}(\kE_q)\geq 1-\eta/2$ for $q\geq q_0$.  
Let $q_2$ be large enough so that
  $n^{-1/4}q^{1/6}\leq \frac{\eps}{\beta+1}$ for $q\geq q_2$.  If $q\geq q_2$ and $\kE_q$ holds, then for every $x,y\in V_q$, 
  we have that $n^{-1/4}\dgr(x,y)\leq \bar{\eps}$ implies $n^{-1/4}\dist_{\cnq}(x,y)\leq \eps$. 
   Hence, if $q\geq q_2$ and if $\kE_q$ and $\kF_q$ hold, then with the same coupling $\nu$ as above, but now for the distance $\frac1{n^{1/4}}\dist_{\cnq}$, we have
  \[
\nu\Big( (V_q\times V_q)_{\eps} \Big) \geq 1-\eps.
\] 
This concludes the proof, since the event $\kE_q\cap \kF_q$ occurs with probability at least $1-\eta$ for $q\geq\mathrm{max}(q_0,q_1,q_2)$.   
\end{proof}

\begin{lemma}\label{lem:dPdistC2}
The GHP distance between  $(K(\cnq),\frac1{n^{1/4}}\dist_{\cnq},\mu_{K(\cnq)}')$ and $(\cnq,\frac1{n^{1/4}}\dgr,\mu_{\cnq})$ is $o(1)$ in probability. 
\end{lemma}
\begin{proof}
We follow the argument of projection of masses given in~\cite{addario2017joint}. Let $\kR$ be the correspondence on $V(K(\cnq)),V(\cnq)$ where $u\kR v$ if $u=v$ or if $v$ is a 
non-pole vertex in one of the $3$ networks substituted at the edges incident to $u$. Let $\nu$ be the distribution on pairs in $\kR$ corresponding to drawing a vertex $v\in\cnq$
uniformly, then letting  
$u=v$ if $v\in K(\cnq)$, and otherwise letting $u$ be a random extremity of the edge $e\in K(\cnq)$ in which the network containing $v$ is substituted. Clearly, $\nu$ gives
a coupling between $\mu_{\cnq}$ and $\mu_{K(\cnq)}'$. 

For any $\eps>0$, by Lemma~\ref{lem:cnq} 
and Proposition~\ref{prop:diam}, and the union bound, we have a.a.s. the property that all networks
attached to the edges of $K(\cnq)$ have diameter at most $n^{(2/3+\eps)(1/4+\eps)}$. Taking $\eps$ small enough (e.g. $\eps=1/30$) so that $(2/3+\eps)(1/4+\eps)\leq 1/5$, 
we thus have the a.a.s. property $\kE_q$ that all networks attached to the edges of $K(\cnq)$ have diameter at most $n^{1/5}$. 
When $\kE_q$ holds, the distorsion $\mathrm{dis}(\kR)$ with respect to $\frac1{n^{1/4}}\dgr$ thus satisfies $\mathrm{dis}(\kR)\leq 2\,n^{1/5}n^{-1/4}=2\,n^{-1/20}$.  
Since the coupling gives zero probability to pairs that are not in $\kR$, we conclude that, when $\kE_q$ holds, the GHP distance between 
$(K(\cnq),\frac1{n^{1/4}}\dist_{\cnq},\mu_{K(\cnq)}')$ and $(\cnq,\frac1{n^{1/4}}\dgr,\mu_{\cnq})$ is at most $2\,n^{-1/20}=o(1)$. Since $\kE_q$ holds a.a.s.,  the GHP 
distance between $(K(\cnq),\frac1{n^{1/4}}\dist_{\cnq},\mu_{K(\cnq)}')$ and $(\cnq,\frac1{n^{1/4}}\dgr,\mu_{\cnq})$ is $o(1)$ in probability. 
\end{proof}

Proposition~\ref{prop:LY} then directly follows from Lemma~\ref{lem:dPdistC1} 
 and Lemma~\ref{lem:dPdistC2} (and the fact that the GHP distance satisfies the triangle inequality).

%%%%%%%%%%%%%%%%%%%%%%%%%%%%%%%%%%%%%%%%%%%%%%%%%%%%%%%%%%%%%%%%%%%
%%                                                               %%
%% Use the two commands below for producing your bibliography    %%
%% with bibtex, then comment again the commands and include the  %%
%% content of the .bbl file in this file below the commands.     %%
%%                                                               %%
%%%%%%%%%%%%%%%%%%%%%%%%%%%%%%%%%%%%%%%%%%%%%%%%%%%%%%%%%%%%%%%%%%%

\addcontentsline{toc}{section}{Index of notations}
\printindex

%\newpage

% \bibliographystyle{amsplain}
% \bibliography{GraphesCubiques}

% add below the content of your .bbl file produced by bibtex.

\providecommand{\bysame}{\leavevmode\hbox to3em{\hrulefill}\thinspace}
\providecommand{\MR}{\relax\ifhmode\unskip\space\fi MR }
\providecommand{\MRhref}[2]{%
  \href{http://www.ams.org/mathscinet-getitem?mr=#1}{#2}
}
\providecommand{\href}[2]{#2}

%%%%%%%%%%%%%%%%%%%%%%%%%%%%%%%%%%%%%%%%%%%%%%%%%%%%%%%%%%%%%%%%%%%
%%                                                               %%
%% You may add acknowledgments (optional).                       %%
%%                                                               %%
%%%%%%%%%%%%%%%%%%%%%%%%%%%%%%%%%%%%%%%%%%%%%%%%%%%%%%%%%%%%%%%%%%%
% \begin{acks}
% \cmar{We are grateful to an anonymous referee, whose careful reading significantly improved the presentation of this paper.}
% \end{acks}

%%%%%%%%%%%%%%%%%%%%%%%%%%%%%%%%%%%%%%%%%%%%%%%%%%%%%%%%%%%%%%%%%%%
%%                                                               %%
%% You have reached the end of your document.                    %%
%%                                                               %%
%%%%%%%%%%%%%%%%%%%%%%%%%%%%%%%%%%%%%%%%%%%%%%%%%%%%%%%%%%%%%%%%%%%

\end{document}